\documentclass[review,3p,square,11pt]{elsarticle}

\usepackage{amssymb}
\usepackage{amsmath}
\usepackage{amsthm}
\usepackage{stmaryrd}
\usepackage{bbm}
\usepackage{enumitem}

\usepackage{fancyhdr}
\usepackage{hyperref}
\hypersetup{%
  pdftitle={BSDEs},%
  pdfsubject={BSDEs},%
  pdfauthor={ShengJun FAN, Ying Hu},%
  pdfkeywords={BSDEs},%
  pdfstartview=FitH,%
  CJKbookmarks=true,%
  bookmarksnumbered=true,%
  bookmarksopen=true,%
  colorlinks=true, linkcolor=blue, urlcolor=blue, citecolor=blue, %
}

\usepackage{cleveref}
\crefname{thm}{Theorem}{Theorems}
\crefname{pro}{Proposition}{Propositions}
\crefname{lem}{Lemma}{Lemmas}
\crefname{rmk}{Remark}{Remarks}
\crefname{cor}{Corollary}{Corollaries}
\crefname{dfn}{Definition}{Definitions}
\crefname{ex}{Example}{Examples}
\crefname{section}{Section}{Sections}
\crefname{subsection}{Subsection}{Subsections}

\newcommand{\eps}{\varepsilon}

\newcommand{\To}{\rightarrow}
\newcommand{\as}{{\rm d}\mathbb{P}\times{\rm d} t-a.e.}

\newcommand{\ps}{\mathbb{P}-a.s.}

\newcommand{\F}{\mathcal{F}}
\newcommand{\E}{\mathbb{E}}

\newcommand{\p}{\mathbb{P}}

\newcommand{\s}{\mathcal{S}}
\newcommand{\lcal}{\mathcal{L}}
\newcommand{\mcal}{\mathcal{M}}
\newcommand{\ecal}{\mathcal{E}}

\newcommand{\hcal}{\mathcal{H}}

\newcommand{\T}{[0,T]}

\newcommand{\R}{{\mathbb R}}

\newcommand{\RE}{\forall}

\newcommand {\Dis}{\displaystyle}

\newtheorem{thm}{Theorem}[section]
\newtheorem{lem}[thm]{Lemma}
\newtheorem{pro}[thm]{Proposition}
\newtheorem{rmk}[thm]{Remark}
\newtheorem{cor}[thm]{Corollary}

\newtheorem{ex}[thm]{Example}

\journal{arXiv}
\begin{document}
\begin{frontmatter}

\title{{Multi-dimensional non-Markovian backward stochastic differential equations of  interactively quadratic generators}\tnoteref{found}}
\tnotetext[found]{This work is supported by National Natural Science Foundation of China (Nos. 12171471, 12031009 and 11631004), by Key Laboratory of Mathematics for Nonlinear Sciences (Fudan University), Ministry of Education, Handan Road 220, Shanghai 200433, China; by Lebesgue Center of Mathematics ``Investissements d'avenir" program-ANR-11-LABX-0020-01, by CAESARS-ANR-15-CE05-0024 and by MFG-ANR-16-CE40-0015-01.
\vspace{0.2cm}}


\author[Fan]{Shengjun Fan} \ead{shengjunfan@cumt.edu.cn}
\author[Hu]{Ying Hu} \ead{ying.hu@univ-rennes1.fr}
\author[Tang]{Shanjian Tang\corref{cor}
} \ead{sjtang@fudan.edu.cn} \vspace{-0.5cm}

\affiliation[Fan]{organization={School of Mathematics, China University of Mining and Technology},
            city={Xuzhou 221116},
            country={China}}

\affiliation[Hu]{organization={Univ. Rennes, CNRS, IRMAR-UMR6625},
            city={F-35000, Rennes},
            country={France}}

\affiliation[Tang]{organization={Department of Finance and Control Sciences, School of Mathematical Sciences, Fudan University},
            city={Shanghai 200433},
            country={China}}\vspace{-0.5cm}

\cortext[cor]{Corresponding author}

\begin{abstract}
This paper is devoted to a general solvability of multi-dimensional  non-Markovian backward stochastic differential equations (BSDEs) with interactively quadratic generators. Some general structures of the generator $g$ are posed for  both local and global existence and uniqueness results on BSDEs, which admit  a general growth of  the generator $g$ in the state variable $y$, and  a quadratic growth of the $i$th component $g^i$ both  in the $j$th row $z^j$ of the state variable $z$ for $j\neq i$ (by which we mean the ``{\it interactively quadratic}" growth) and  in the $i$th row $z^i$ of $z$.
We first establish an existence and uniqueness result on local bounded solutions and then several existence and uniqueness results on global bounded and unbounded solutions. They improve several existing works in the non-Markovian setting, and also incorporate some interesting examples, one of which is a partial answer to the  problem posed in \citet{Jackson2023SPA}. A comprehensive study on the bounded solution of one-dimensional quadratic BSDEs with unbounded stochastic parameters is provided  for deriving our main results.
\vspace{0.2cm}
\end{abstract}

\begin{keyword}
Multi-dimensional BSDE \sep Interactively quadratic generator \sep BMO martingale\sep \\ \hspace*{2cm}  Bounded solution \sep Unbounded solution \sep Existence and uniqueness. \vspace{0.2cm}

\MSC[2010] 60H10\vspace{0.2cm}
\end{keyword}

\end{frontmatter}
\vspace{-0.5cm}

\section{Introduction}
\label{sec:1-Introduction}
\setcounter{equation}{0}

Let $T\in (0,+\infty)$, $(B_t)_{t\in\T}$ be a $d$-dimensional  Brownian motion defined on a complete probability space $(\Omega, \F, \mathbb{P})$, and $(\F_t)_{t\in\T}$ be the augmented  filtration generated by $B$. Consider the following backward stochastic differential equation (BSDE in short):
\begin{equation}\label{eq:1.1}
  Y_t=\xi+\int_t^T g(s,Y_s,Z_s){\rm d}s-\int_t^T Z_s {\rm d}B_s, \ \ t\in\T,
\end{equation}
where the terminal value $\xi$ is a $d$-dimensional $\F_T$-measurable random vector, the generator function $g(\omega, t, y, z):\Omega\times\T\times\R^n\times\R^{n\times d}\To \R^n$
is $(\F_t)$-progressively measurable for each pair $(y,z)$, and the solution $(Y_t,Z_t)_{t\in\T}$ is a pair of $(\F_t)$-progressively measurable processes taking values in $\R^n\times\R^{n\times d}$ which almost surely satisfies BSDE \eqref{eq:1.1}.

\citet{Bismut1973JMAA} initially introduced in 1973 the above BSDE in the linear case and further in  \cite{Bismut1976SIAM} studied in 1976 the backward stochastic Riccati differential equation as a specifically structured  matrix-valued nonlinear BSDE, and \citet{PardouxPeng1990SCL} established in 1990 the existence and uniqueness result for a multidimensional nonlinear BSDE, where the generator $g$ is uniformly Lipschitz continuous with respect to the state variables $(y,z)$. Since then, there has been an increasing interest in BSDEs and their applications in  numerous fields such as mathematical finance, stochastic control, partial differential equations (PDEs), etc. The interested reader is refereed to,   among others, \citet{KarouiPengQuenez1997SPA,Kobylanski2000AP, KarouiHamadene2003SPA,HuImkeller2005AAP,FreiDosReis2011MFE,KramkovPulido2016AAP,EscauriazaSchwarzXing2022AAP,
Tian2023SIAMJFM}, and \citet{Weston2024FS}.\vspace{0.2cm}

When the generator $g$ has a quadratic growth in the state variable $z$, the above BSDE is called quadratic.  A quadratic BSDE naturally arises from non-Markovian  linear quadratic control, risk-sensitive  control, differential games and economic equilibrium (see~\citet{Bismut1976SIAM} and~\citet{KarouiHamadene2003SPA}).   When $n=1$ and the terminal value is bounded, the  existence and uniqueness result is first given by \citet{Kobylanski2000AP}, and subsequently by~\citet{Tevzadze2008SPA,BriandElie2013SPA,Fan2016SPA} and \citet{LuoFan2018SD}. When $n=1$ and the terminal value is allowed to be unbounded, existence and uniqueness of unbounded solutions of quadratic BSDEs was given in \citet{BriandHu2006PTRF,BriandHu2008PTRF,
DelbaenHuRichou2011AIHPPS,BarrieuElKaroui2013AoP,DelbaenHuRichou2015DCD}, and \citet{FanHuTang2020CRMA}.
A system of quadratic BSDEs is widely known to be  more difficult to be solved (see \citet{Peng1999Open}), and a general existence and uniqueness, without any structural assumption on the quadratic growth of the generator, is impossible, as shown by~\citet{FreiDosReis2011MFE} via a counterexample. On the other hand, many systems of quadratic BSDEs have been derived from various problems such as financial price-impact modeling, financial market equilibrium  among  interacting agents, nonzero-sum risk-sensitive stochastic differential games, and stochastic equilibrium in an incomplete financial markets. Many results on  them are available now. When  the terminal value has a very small bound, \citet{Tevzadze2008SPA} established the first existence and uniqueness result, which has inspired numerous studies on the bounded or unbounded solution under some ``smallness" assumptions on the terminal value, the terminal time and the generator; see for example \citet{Frei2014SPA,JamneshanKupperLuo2017ECP,KramkovPulido2016AAP,
KramkovPulido2016SIAM,HarterRichou2019EJP} and \citet{Kardaras2022SA}. In the particular Markovian setting, \citet{cheridito2015Stochastics} proved the global solvability of a special system of quadratic BSDEs associated to a forward SDE, and \citet{XingZitkovic2018AoP} formulated a large class of systems of quadratic BSDEs via the Bensoussan-Frehse (BF) condition (inspired by \citet{BensoussanFrehse2002ESAIM}) and the a priori boundedness (AB) condition on the generator (see \ref{A:AB} in \cref{sec:2-statement of main result} for details); see \citet{WestonZitkovic2020FS} and \citet{EscauriazaSchwarzXing2022AAP} for more developments. Recently, \citet{JacksonZitkovic2022SIAMC} and \citet{Jackson2023SPA} further extended the result of \cite{XingZitkovic2018AoP} to a non-Markovian system of quadratic BSDEs, but assuming an additional Malliavin-regularity on the generator, in addition to the (BF) and  (AB) conditions.\vspace{0.2cm}

Numerous  works are also available on bounded solutions in the non-Markovian setting, with  neither smallness  nor Malliavin-regular assumptions on the terminal time, the terminal value and the generator.  \citet{cheridito2015Stochastics} considered multidimensional BSDEs with projectable quadratic generators and sub-quadratic generators. \citet{HuTang2016SPA} established the global solvability of multi-dimensional BSDEs, assuming that  the $i$th component $g^i(t,y,z)$ of the generator $g(t,y,z)$ admits a quadratic growth only in the $i$th row $z^i$ of the matrix variable $z$ (the so-called ``{\it diagonally quadratic}" generator) and is bounded in $z^j$ for each $j\neq i$, and consequently answered the open problem arising from a nonzero-sum risk-sensitive stochastic differential game posed by \citet{KarouiHamadene2003SPA}.  When $g^i(t,y,z)$ has a strictly quadratic growth  in $z^i$ for each $i=1,\ldots,n$,  \citet{FanHuTang2023JDE}  further improved the preceding result of \cite{HuTang2016SPA} by allowing the generator $g(t,y,z)$ to have a sub-quadratic growth in $z$.   \citet{Luo2020EJP} solved a system of BSDEs, assuming that $g^i(t,y,z)$ depends quadratically on $z^j$ for any $1\leq j\leq i$ ( the so-called ``{\it triangularly quadratic}" generator) and has  a special form. Very recently, \citet{FanWangYong2022arXiv} generalized the result of \cite{FanHuTang2023JDE} to the case of multi-dimensional quadratic backward stochastic volterra integral equations, and \citet{Weston2024FS} applied the result of \cite{FanHuTang2023JDE} to prove the global existence of a Radner equilibrium in a limited participation economy models.\vspace{0.2cm}

The present paper focuses on both local and global solvability of systems of BSDEs with interactively quadratic generators in the non-Markovian setting, when $g^i$ depends quadratically both on $z^i$ and on $z^j$ for each $j\neq i$. Our results unify and strengthen some existing results in the non-Markovian setting such as \cite{HuTang2016SPA}, \cite{Luo2020EJP} and \cite{FanHuTang2023JDE}, and show that under some circumstances, a system of quadratic BSDEs has a unique solution when the generator $g(t,y,z)$ has a  quadratic growth  in the product of $\sqrt{\theta}$ and  some rows of the matrix variable $z$ for a  small $\theta$, which is compared to some existing results mentioned in the second last paragraph. More specifically, we first establish an existence and uniqueness result on the local bounded solution of systems of quadratic BSDEs with bounded terminal values (see \cref{thm:2.3} in \cref{sec:2-statement of main result}) under general assumptions, see assumptions \ref{A:B1} and \ref{A:B2} in \cref{sec:2-statement of main result} for details. The assumption \ref{A:B1} states that $g^i$  has any of the following three different ways of growth in $z^i$: strictly quadratic, generally quadratic and linear.   $g^i$ also has a general growth in $y$, and has a quadratic growth in $\sqrt{\theta} z^j$ for $j\neq i$  for a small  $\theta$, and an interacting quadratic growth like the inner product of $z^i$ and $z^j$ for $j<i$. We would like to mention that generally speaking, the parameters of processes defined in assumptions \ref{A:B1} and \ref{A:B2} (namely, $\alpha, \bar\alpha, \tilde\alpha$ and $v$) are all unbounded and stochastic, and satisfy different integrability conditions, which is endogenous  for our desired conclusions, and on the other hand, provides convenience for the study on unbounded solutions of systems of  quadratic BSDEs. Then, under, respectively, an a priori boundedness assumption and two stronger assumptions than \ref{A:B1} (see \ref{A:AB}, \ref{A:C1a} and \ref{A:C1b} in \cref{sec:2-statement of main result}), based on \cref{thm:2.3} we further establish three existence and uniqueness results on global bounded solutions of systems of quadratic BSDEs, see \cref{thm:2.5e,thm:2.5,thm:2.8a} in \cref{sec:2-statement of main result} for details. These results include  all those of  \citet{HuTang2016SPA}, \citet{Luo2020EJP} and \citet{FanHuTang2023JDE}. The proofs involve the contraction mapping argument and the utilization of the induction technique based on some delicate a priori estimates on the solution, where the problem on the bounded solution of one-dimensional quadratic BSDEs with unbounded stochastic parameters is comprehensively investigated by virtue of the BMO martingale (bounded oscillation martingale) tool, Girsanov's transform and some useful inequalities, and some new results are explored, see \cref{pro:A.1,pro:A.2,pro:A.3} in the Appendix for details. As a natural application of \cref{thm:2.5,thm:2.8a}, by means of the invertible linear transformation method we address solvability of three kinds of special structured systems of BSDEs, see \cref{thm:2.14g,thm:2.12b,thm:2.14b} in \cref{sec:2-statement of main result} for more details, which can be respectively compared with Theorems 6.9 and 6.19 of \citet{Jackson2023SPA} and Theorem 3.1 of \citet{XingZitkovic2018AoP}. By the way, by \cref{pro:2.16b,pro:2.15b} in \cref{sec:2-statement of main result} we answer partially the open problem 6.25 of \citet{Jackson2023SPA} on solvability of a special system of BSDEs. Finally, according to \cref{thm:2.5e,thm:2.5,thm:2.8a}, we establish three existence and uniqueness results for the unbounded solution of multi-dimensional quadratic BSDEs with the unbounded BMO terminal values, where the generator $g$ needs to be bounded in $y$, see assumptions \ref{A:D1} and \ref{A:D2} as well as \cref{thm:2.22e,thm:2.9,thm:2.16a} in \cref{sec:2-statement of main result} for details. The method of proof is to transfer the BSDE with an unbounded terminal value into one with a bounded terminal value and a generator satisfying the assumptions of \cref{thm:2.5e,thm:2.5,thm:2.8a}. These results also improve some existing results in the non-Markovian setting.\vspace{0.2cm}

The rest of the paper is organized as follows. In the next section, we introduce notations used, state the main results, and present some interesting examples and remarks to illustrate the novelty of our results. Some proofs  of the results are also given in this section when they are not lengthy. In \cref{sec:3-local solution} we prove the solvability of local bounded solutions (\cref{thm:2.3}), and in \cref{sec:4-global bounded solution} we prove solvability of global bounded solutions (\cref{thm:2.5,thm:2.8a}). The solvability of global unbounded solution (\cref{thm:2.22e,thm:2.9}) is finally proved in \cref{sec:5-global unbounded solution}. In Appendix, we present and prove some auxiliary results for  bounded solutions of one-dimensional quadratic BSDEs with unbounded stochastic parameters, including existence, uniqueness and several important a priori estimates.

\section{Notations and statement of main results}
\label{sec:2-statement of main result}
\setcounter{equation}{0}

\subsection{Notations\vspace{0.2cm}}

First, let us fix a terminal time $T\in (0,\infty)$ and two positive integers $n$ and $d$. Let $a\wedge b:=\min\{a,b\}$, $a\vee b:=\max\{a,b\}$, $a^+:=a\vee 0$, $a^-:=-(a\wedge 0)$ and $$\sum_{j=1}^{0}b^j=\sum_{j=n+1}^{n}b^j=0$$ for any real $b^j$. And, denote by $|\cdot|$ the Euclidean norm, $z^\top$ the transpose of vector (or matrix) $z$, ${\bf 1}_A$ the indicator of set $A$, and
${\rm sgn}(x):={\bf 1}_{x>0}-{\bf 1}_{x\leq 0}$.\vspace{0.2cm}

In the whole paper, all processes are $(\F_t)_{t\in\T}$-progressively measurable, and all equalities and inequalities between
random variables and processes are understood in the senses of $\ps$ and $\as$, respectively. We need the following spaces of random variables and processes:
\begin{itemize}
\item $L^{\infty}(\R^n)$: all $\R^n$-valued  and $\F_T$-measurable random variables $\xi$ satisfying
$$\|\xi\|_{\infty}:=\mathop{\rm ess\ sup}_{\omega\in \Omega}|\xi(\omega)| <+\infty.\vspace{0.1cm}$$

\item $\s^p(\R^n)$ for $p\geq 1$: all $\R^n$-valued  continuous adapted processes $(Y_t)_{t\in\T}$ such that
$$\|Y\|_{{\s}^p}:=\E\Bigg[\sup_{t\in\T} |Y_t|^p\Bigg]^{1\over p}<+\infty.\vspace{0.1cm}$$

\item $\s^{\infty}(\R^n)$: all $Y\in \bigcap_{p\geq 1}\s^p(\R^n)$ such that
$$\|Y\|_{{\s}^{\infty}}:=\mathop{\rm ess\ sup}_{(\omega,t)\in \Omega\times \T}|Y_t(\omega)| =\Big\|\sup_{t\in\T} |Y_t| \Big\|_{\infty}<+\infty.\vspace{0.1cm}$$

\item $\mcal^{\infty}$: all real-valued non-negative progressively measurable process $(Y_t)_{t\in\T}$ satisfying
$$
\|Y\|_{{\mcal}^{\infty}}:=\sup_{\tau\in \mathcal {T}_{\T}}\left\|\E_{\tau}\left[\int_{\tau}^T Y_s {\rm d}s\right]\right\|_{\infty}<+\infty,\vspace{0.1cm}
$$
here and hereafter $\mathcal {T}_{[a,b]}$ denotes the set of all $(\F_t)$-stopping times $\tau$ valued in $[a,b]\subset \T$, and $\E_{\tau}$ stands for the conditional expectation with respect to $\F_\tau$.

\item $\lcal^{\infty}$: all $Y\in \mcal^{\infty}$ satisfying
$$
\|Y\|_{{\lcal}^{\infty}}:=\left\|\int_0^T Y_s {\rm d}s \right\|_{\infty}<+\infty.
$$

\item $\hcal^p(\R^{n\times d})$ for $p\geq 1$: all $\R^{n\times d}$-valued  progressively measurable processes $(Z_t)_{t\in\T}$ such that
$$
\|Z\|_{\hcal^p}:=\E\left[\left(\int_0^T |Z_s|^2{\rm d}s\right)^{p\over 2}\right]^{1\over p}<+\infty.\vspace{0.1cm}
$$

\item ${\rm BMO}(\R^{n\times d})$: all $Z\in \hcal^2(\R^{n\times d})$ such that
$$
\|Z\|_{\rm BMO}:=\sup_{\tau\in \mathcal {T}_{\T}}\left\|\E_{\tau}\left[\int_{\tau}^T |Z_s|^2 {\rm d}s\right]\right\|_{\infty}^{1\over 2}<+\infty.\vspace{0.1cm}
$$
\end{itemize}

We write $L^{\infty}(\R):= L^{\infty}(\R^1)$,  $\s^{\infty}(\R):= \s^{\infty}(\R^1)$, $\s^p(\R):= \s^p(\R^1)$ for $p\ge 1$, $\hcal^p(\R) :=\hcal^p(\R^{1\times 1})$ for $p\ge 1$, and ${\rm BMO}(\R):={\rm BMO}(\R^{1\times 1})$. By $\hcal_{[a,b]}$, we denote the restriction to the subinterval $[a,b]$ of the space $\hcal$ of processes on $[0,T]$, i.e., the space of all processes which are restrictions to the interval $[a,b]$ of processes in $\hcal$. It is noted that the process $(\int_0^t Z_s{\rm d}B_s)_{t\in\T}$ is an $n$-dimensional BMO martingale for each $Z\in {\rm BMO}(\R^{n\times d})$. We refer to \citet{Kazamaki1994book} for more details on the BMO theory. For the reader's  convenience, here we would like to recall the well-known John-Nirenberg inequality (see for example Lemma A.1 in \citet{HuTang2016SPA}): Let the process $Z$ belong to the space of ${\rm BMO}(\R)$. If $\|Z\|_{\rm BMO}^2<1$, then for any $\tau\in \mathcal {T}_{\T}$,
$$
\E_{\tau}\left[\exp\left(\int_{\tau}^T |Z_s|^2{\rm d}s\right)\right]\leq \frac{1}{1-\|Z\|_{\rm BMO}^2}.\vspace{0.2cm}
$$

Throughout the paper, we also always fix four real-valued non-negative progressively measurable processes
$$
(v_t)_{t\in\T}\in {\rm BMO}(\R),\ \ (\tilde\alpha_t)_{t\in\T}\in \lcal^{\infty},\ \ (\bar\alpha_t)_{t\in\T}\in \mcal^{\infty}
$$
and $(\alpha_t)_{t\in\T}$ as well as a deterministic nondecreasing continuous function $\phi(\cdot):[0,+\infty)\To [0,+\infty)$ and several real constants $\beta,\lambda,\bar\lambda,\theta,c,\bar c\geq 0, \ 0<\bar\gamma\leq \gamma \ {\rm and}\ \delta\in [0,1).$\vspace{0.2cm}

Finally, for $i=1,\cdots, n$, denote by $(z^i, y^i, g^i)$  the $i$th row of $(z,y,g)\in\R^{n\times d} \times  \R^n\times  \R^n$. In addition, for some real $r>0$ and some time sub-interval $[a,b]\subset \T$, we let $\ecal^{\infty}_{[a,b]}(r)$ represent the set of all real-valued non-negative progressively measurable process $(\hat\alpha_t)_{t\in \T}$ such that
$$\|\hat\alpha\|_{\ecal^{\infty}_{[a,b]}(r)}:={1\over r}\ln\left(\sup_{\tau\in \mathcal{T}_{[a,b]}}\left\|\E_{\tau}\left[\exp\left(r\int_\tau^b \hat\alpha_s {\rm d}s\right)\right]\right\|_{\infty}\right)<+\infty.$$
We also denote $\ecal^{\infty}_{[0,T]}(r)$ by $\ecal^{\infty}(r)$ simply.

\begin{rmk}\label{rmk:2.1}
Let $[a,b]\subset \T$. By virtue of Jensen's inequality, the John-Nirenberg inequality and H\"{o}lder's inequality, it is not difficult to verify the following assertions.
\begin{itemize}
\item [(i)] The assertion $\hat\alpha \in {\rm BMO}(\R)$ holds if and only if $\hat\alpha^2\in \mcal^{\infty}$, and $\|\hat\alpha^2\|_{\mcal^{\infty}}=\|\hat\alpha\|_{\rm BMO}^2$.\\
If $\hat\alpha\in \lcal^{\infty}_{[a,b]}$, then $\hat\alpha\in \ecal^{\infty}_{[a,b]}(r)$ for each $r>0$, and
$\|\hat\alpha\|_{\ecal^{\infty}_{[a,b]}(r)}\leq \|\hat\alpha\|_{\lcal^{\infty}_{[a,b]}}$.\vspace{0.1cm} \\
If $\hat\alpha\in \ecal^{\infty}_{[a,b]}(r)$ for some $r>0$, then $\hat\alpha\in \mcal^{\infty}_{[a,b]}$, and $\|\hat\alpha\|_{\mcal^{\infty}_{[a,b]}}\leq \|\hat\alpha\|_{\ecal^{\infty}_{[a,b]}(r)}$.\vspace{0.1cm} \\
If $\hat\alpha\in \mcal^{\infty}_{[a,b]}$, then $\hat\alpha\in \ecal^{\infty}_{[a,b]}(r)$ for each $0<r<r_0:=\|\hat\alpha\|_{\mcal^{\infty}_{[a,b]}}$.\vspace{0.1cm} \\
In addition, for any $0<r<\bar r$, we have $\ecal^{\infty}(r)\subset \ecal^{\infty}_{[a,b]}(r)$ and $\ecal^{\infty}_{[a,b]}(\bar r) \subset \ecal^{\infty}_{[a,b]}(r)$.
\item [(ii)] If $\hat\alpha^1\in \ecal^{\infty}_{[a,b]}(r)$ for some $r>0$ and $\hat\alpha^2 \in \lcal^{\infty}_{[a,b]}$, then $\hat\alpha:=\hat\alpha^1+\hat\alpha^2\in \ecal^{\infty}_{[a,b]}(r)$. Moreover,
$$
\|\tilde\alpha\|_{\ecal^{\infty}_{[a,b]}(r)}\leq \|\hat\alpha^1\|_{\ecal^{\infty}_{[a,b]}(r)}
+\|\hat\alpha^2\|_{\lcal^{\infty}_{[a,b]}}.
$$
\item [(iii)] If $\hat\alpha^1\in \ecal^{\infty}_{[a,b]}(pr)$ and $\hat\alpha^2\in \ecal^{\infty}_{[a,b]}(qr)$ for some $r>0$ and $p,q>1$ satisfying ${1\over p}+{1\over q}=1$, then $\hat\alpha:=\hat\alpha^1+\hat\alpha^2\in \ecal^{\infty}_{[a,b]}(r)$. Moreover, we have
$$
\|\hat\alpha\|_{\ecal^{\infty}_{[a,b]}(r)}\leq \|\hat\alpha^1\|_{\ecal^{\infty}_{[a,b]}(pr)}
+\|\hat\alpha^2\|_{\ecal^{\infty}_{[a,b]}(qr)}.
$$
\end{itemize}
\end{rmk}

Let the $\R^n$-valued function $g(\omega,t,y,z): \Omega\times\T\times\R^n\times\R^{n\times d}\To \R^n$ be $(\F_t)$-progressively measurable for each $(y,z)\in \R^n\times\R^{n\times d}$, and consider the following multidimensional BSDE:
\begin{equation}\label{eq:2.1}
Y_t=\xi+\int_t^T g(s,Y_s,Z_s){\rm d}s-\int_t^T Z_s {\rm d}B_s, \ \ t\in\T,
\end{equation}
or, equivalently,
\begin{equation}\label{eq:2.2}
Y_t^i=\xi^i+\int_t^T g^i(s,Y_s,Z_s) {\rm d}s-\int_t^T Z_s^i {\rm d}B_s, \ \ t\in\T,\ \ i=1,\cdots,n,
\end{equation}
where $\xi$ is an $\F_T$-measurable $\R^n$-valued random vector, and the solution $(Y,Z)$ is a pair of $(\F_t)$-progressively measurable processes with values in $\R^n\times \R^{n\times d}$.\vspace{0.2cm}

\subsection{Local bounded solution of multi-dimensional quadratic BSDEs\vspace{0.2cm}}

We first introduce the following two assumptions on the quadratical generator $g$.\vspace{-0.1cm}

\begin{enumerate}
\renewcommand{\theenumi}{(B\arabic{enumi})}
\renewcommand{\labelenumi}{\theenumi}
\item\label{A:B1} For each fixed $i=1,\cdots,n$, either of the following three conditions holds:
\begin{enumerate}
\item [(i)] $f:=g^i$ or $f:=-g^i$ satisfies that $\as$, for any $(y,z)\in \R^n\times\R^{n\times d}$,
$$
\begin{array}{l}
\Dis\frac{\bar \gamma}{2} |z^i|^2-\bar\alpha_t(\omega)-\phi(|y|)-\sum_{j=i+1}^{n} (\bar\lambda |z^j|^{1+\delta}+\theta |z^j|^2)-\bar c\sum_{j=1}^{i-1}|z^j|^2\leq f(\omega,t,y,z)\\[3mm]
\quad\quad\quad \Dis \leq \alpha_t(\omega)+\phi(|y|)+ \sum_{j\neq i} (\lambda |z^j|^{1+\delta}+\theta |z^j|^2)+\frac{\gamma}{2} |z^i|^2;
\end{array}
$$

\item [(ii)] Almost everywhere in $\Omega\times\T$, for any $(y,z)\in \R^n\times\R^{n\times d}$, we have
$$
|g^i(\omega,t,y,z)|\leq \tilde\alpha_t(\omega)+\phi(|y|)+[v_t(\omega)+\phi(|y|)]|z^i|+
c\sum_{j=1}^{i-1}|z^i(z^j)^\top|+\frac{\gamma}{2} |z^i|^2;
$$

\item [(iii)] Almost everywhere in $\Omega\times\T$, for any $(y,z)\in \R^n\times\R^{n\times d}$, we have
$$
|g^i(\omega,t,y,z)|\leq \bar\alpha_t(\omega)+\phi(|y|)+\bar\lambda |z|+\theta\sum_{j\neq i}|z^j|^2.\vspace{-0.2cm}
$$
\end{enumerate}

\item\label{A:B2} For $i=1,\cdots,n$, $g^i$ satisfies that $\as$, for any $(y,\bar y, z, \bar z)\in \R^n\times\R^n\times\R^{n\times d}\times\R^{n\times d}$,
    $$
    \begin{array}{l}
   |g^i(\omega,t,y,z)-g^i(\omega,t,\bar y,\bar z)|\\
   \ \ \leq \Dis \phi(|y|\vee|\bar y|) \Dis \Big\{\left([v_t(\omega)]^{1+\delta}+|z|^{1+\delta}+|\bar z|^{1+\delta}\right)|y-\bar y|+\left(v_t(\omega)+|z|+|\bar z|\right)\sum_{j=1}^i |z^i-\bar z^i|\\
     \Dis \hspace{3.4cm} +\left[\left([v_t(\omega)]^\delta+|z|^\delta+|\bar z|^\delta\right)+\theta \left(v_t(\omega)+|z|+|\bar z|\right)\right]\sum_{j=i+1}^n |z^j-\bar z^j|\Big\}.\vspace{0.2cm}
    \end{array}
    $$
\end{enumerate}

\begin{rmk}\label{rmk:2.2}
Concerning Assumptions \ref{A:B1} and \ref{A:B2}, we make the following remarks.
\begin{itemize}
\item [(i)] Both assumptions \ref{A:B1} and \ref{A:B2} imply that $g^i$ admits a strictly quadratic, generally quadratic and  linear growth in $z^i$. $g^i$ can also have a general growth in $y$, a quadratic growth in $\sqrt{\theta} z^j$ for $j\neq i$  and a small constant $\theta$, and an interactively quadratic growth like the inner product of $z^i$ and $z^j$ for $j<i$.

\item [(ii)] Assumptions \ref{A:B1}(i) and \ref{A:B2} are more general than those of \citet[Theorem 2.2]{Luo2020EJP}, where for each $i=1,\cdots,n$, $g^i$ has a strictly quadratic growth in $z^i$ and a stronger growth and continuity in both unknown variables $y$ and $z$.

\item [(iii)] The following assumptions \ref{A:B1'} and \ref{A:B2'} are used in Theorem 2.1 of \citet{FanHuTang2023JDE}:
\begin{enumerate}
\renewcommand{\theenumi}{(B\arabic{enumi}')}
\renewcommand{\labelenumi}{\theenumi}
\item\label{A:B1'} For each $i=1,\cdots, n$, $g^i$ satisfies that $\as$, for any $(y,z)\in \R^n\times\R^{n\times d}$,
$$
|g^i(\omega,t,y,z)|\leq \tilde\alpha_t(\omega)+\phi(|y|)+\frac{\gamma}{2} |z^i|^2+\lambda\sum_{j\neq i}|z^j|^{1+\delta};
$$
\item\label{A:B2'} For $i=1,\cdots,n$, $g^i$ \vspace{0.1cm}satisfies that $\as$, for any $(y,\bar y, z, \bar z)\in \R^{2n}\times\R^{2(n\times d)}$,
    $$
    \begin{array}{l}
   |g^i(\omega,t,y,z)-g^i(\omega,t,\bar y,\bar z)|\vspace{0.1cm}\\
   \ \ \leq \Dis \phi(|y|\vee|\bar y|) \Dis \Big[\left(1+|z|+|\bar z|\right)(|y-\bar y|+|z^i-\bar z^i|)+\left(1+|z|^\delta+|\bar z|^\delta\right)\sum_{j\neq i} |z^j-\bar z^j|\Big].\vspace{0.1cm}
    \end{array}
    $$
\end{enumerate}
It is easy to check that Assumption \ref{A:B2} is strictly weaker than Assumption \ref{A:B2'}, while both assumptions \ref{A:B1} and \ref{A:B1'} do not cover each other.

\item [(iv)] In Assumption \ref{A:B1}, the set of integers $\{1,\cdots, n\}$ is divided into two disjoint parts $I_1$ and $I_2$, both of which can be $\emptyset$, such that $I_1+I_2=\{1,\cdots, n\}$, $g^i$ satisfies either of \ref{A:B1}(i), \ref{A:B1}(ii) and \ref{A:B1}(iii) for $i\in I_1$, and $-g^i$ satisfies \ref{A:B1}(i) for $i\in I_2$. Now, we define, for each $(\omega,t,y,z)\in \Omega\times \T\times \R^n\times \R^{n\times d}$,
    $$
    \bar g^i(\omega,t,y,z):=\left\{
    \begin{array}{rl}
    g^i(\omega,t,\bar y,\bar z), & \Dis i\in I_1;\vspace{0.2cm}\\
    -g^i(\omega,t,\bar y,\bar z), & \Dis i\in I_2
    \end{array}
    \right.
   $$
with
$$
\bar y^i:=\left\{
    \begin{array}{rl}
    y^i, & \Dis i\in I_1;\vspace{0.2cm}\\
    -y^i, & \Dis i\in I_2
    \end{array}
    \right.
\ \ {\rm and}\ \ \bar z^i:=\left\{
    \begin{array}{rl}
    z^i, & \Dis i\in I_1;\vspace{0.2cm}\\
    -z^i, & \Dis i\in I_2.\vspace{0.3cm}
    \end{array}
    \right.
$$
It is not hard to verify that $\bar g$ satisfies~\ref{A:B1}. Therefore, in Assumption \ref{A:B1} we can assume, without loss of generality,  that $g^i$ satisfies either of \ref{A:B1}(i), \ref{A:B1}(ii) and \ref{A:B1}(iii) for all $i=1,\cdots, n$. Furthermore, it is clear that $\bar g$ satisfies Assumption \ref{A:B2} as $g$ does.
\vspace{0.2cm}
\end{itemize}
\end{rmk}

We have the following general existence and uniqueness result on the local bounded solution of systems of BSDEs with interactively quadratic generators.

\begin{thm}\label{thm:2.3}
Let $\xi\in L^{\infty}(\R^n)$, $\alpha\in \ecal^{\infty}(p\gamma)$ for some $p>1$ and the generator $g$ satisfy Assumptions \ref{A:B1} and \ref{A:B2}. Then, there exist two constants $\eps_0>0$ and $\theta_0>0$ depending only on $(\|\xi\|_{\infty}, \|\alpha\|_{\ecal^{\infty}(p\gamma)}, \|\bar\alpha\|_{\mcal^{\infty}}, \|\tilde\alpha\|_{\lcal^{\infty}}, \|v\|_{{\rm BMO}}, n,\gamma,\bar\gamma,\lambda,\bar\lambda, c,\bar c,\delta,T,p)$ and $\phi(\cdot)$ together with a bounded subset $\mathcal{B}_{\eps_0}$ of the product space $\s^\infty_{[T-\eps_0,T]}(\R^n)\times {\rm BMO}_{[T-\eps_0,T]}(\R^{n\times d})$ such that when $\theta\in [0,\theta_0]$, BSDE \eqref{eq:2.1} has a unique local solution $(Y,Z)$ on the time interval $[T-\eps_0,T]$ with $(Y,Z)\in \mathcal{B}_{\eps_0}$. Moreover, the above conclusion holds still for $p=1$ when $\lambda=0$ and $\theta_0=0$.\vspace{0.2cm}
\end{thm}

\begin{rmk}\label{rmk:2.3-a}
It follows from (ii) and (iii) of \cref{rmk:2.2} that \cref{thm:2.3} strengthens \citet[Theorem 2.2]{Luo2020EJP}, and that \cref{thm:2.3} and \citet[Theorem 2.1]{FanHuTang2023JDE} do not cover each other. In addition, it follows from \citet[Remark 2.2]{FanHuTang2023JDE} that \citet[Theorem 2.1]{FanHuTang2023JDE} extends Theorem 2.2 of \citet{HuTang2016SPA}. However, it can be easily checked that \cref{thm:2.3} and \citet[Theorem 2.2]{HuTang2016SPA} also do not cover each other.\vspace{0.2cm}
\end{rmk}

\subsection{Global bounded solution of multi-dimensional quadratic BSDEs\vspace{0.2cm}}

In this subsection, we will present three existence and uniqueness results on the global bounded solution of multi-dimensional BSDE \eqref{eq:2.1}. For the first one, let us introduce the following a priori boundedness assumption \ref{A:AB} on the generator $g$, which was used in \citet{Jackson2023SPA} and \citet{JacksonZitkovic2022SIAMC}. Interested readers are refereed to \citet{XingZitkovic2018AoP} and \citet{EscauriazaSchwarzXing2022AAP} for the other versions of this assumption.\vspace{-0.1cm}

\begin{enumerate}
\renewcommand{\theenumi}{(AB)}
\renewcommand{\labelenumi}{\theenumi}
\item\label{A:AB} There exists a finite collection $\{a_k\}=(a_1,\cdots,a_K)$ of vectors in $\R^n$ such that
\begin{itemize}
\item [(i)] the collection of $(a_1,\cdots,a_K)$ positively span $\R^n$, i.e., for any $a\in \R^n$ there exist nonnegative constants $\lambda_1,\cdots,\lambda_K$ such that $\lambda_1a_1+\cdots+\lambda_K a_K=a$;
\item [(ii)] Almost everywhere in $\Omega\times\T$, for any $k=1,\cdots, K$ and any $(y,z)\in\R^n\times \R^{n\times d}$, we have \ \
    $
    a_k^\top g(\omega,t,y,z)\leq \tilde\alpha_t(\omega)+\gamma |a_k^\top z|^2.\vspace{0.2cm}
    $
\end{itemize}
\end{enumerate}

Identical to the proof of  Lemma 6.6 of \citet{Jackson2023SPA} and Proposition 3.8 of \citet{JacksonZitkovic2022SIAMC}, we have

\begin{pro}\label{pro:2.4e}
Let $\xi\in L^{\infty}(\R^n)$ and the generator $g$ satisfy Assumption \ref{A:AB}. Assume that for some $h\in (0,T]$, BSDE \eqref{eq:2.1} has a solution $(Y,Z)\in \s^{\infty}_{[T-h,T]}(\R^n)\times {\rm BMO}_{[T-h,T]}(\R^{n\times d})$ on the time interval $[T-h,T]$. Then, there exists a positive constant $\widetilde K>0$ depending only on $(\|\xi\|_{\infty}, \|\tilde\alpha\|_{\lcal^{\infty}}, \{a_k\}, \gamma)$ and being independent of $h$ such that
$$
\|Y\|_{\s^{\infty}_{[T-h,T]}}+\|Z\|_{{\rm BMO}_{[T-h,T]}}^2 \leq \widetilde K.\vspace{0.2cm}
$$
\end{pro}

With \cref{thm:2.3} and \cref{pro:2.4e} in hands, we can closely follow the proof of Theorem 4.1 in \citet{cheridito2015Stochastics} to prove the following \cref{thm:2.5e}, which is our first result on the global bounded solution of multi-dimensional BSDE \eqref{eq:2.1}. All details are omitted here.

\begin{thm}\label{thm:2.5e}
Let $\xi\in L^{\infty}(\R^n)$, $\alpha\in \ecal^{\infty}(p\gamma)$ for some $p>1$ and the generator $g$ satisfy Assumptions \ref{A:B1}, \ref{A:B2} and \ref{A:AB}. Then, there exists a positive constant $\theta_0>0$ depending only on $(\|\xi\|_{\infty}, \|\alpha\|_{\ecal^{\infty}(p\gamma)}, \|\bar\alpha\|_{\mcal^{\infty}}, \|\tilde\alpha\|_{\lcal^{\infty}}, \|v\|_{{\rm BMO}}, n,\gamma,\bar\gamma,\lambda,\bar\lambda, c,\bar c,\delta,T,p)$ and $\phi(\cdot)$ such that when $\theta\in [0,\theta_0]$, BSDE \eqref{eq:2.1} admits a unique global solution $(Y,Z)\in \s^\infty(\R^n)\times {\rm BMO}(\R^{n\times d})$. Moreover, the above assertion is still true for $p=1$ when $\lambda=0$ and $\theta_0=0$. \vspace{0.1cm}
\end{thm}

By \cref{thm:2.5e} and (i) of \cref{rmk:2.1}, we immediately have

\begin{cor}\label{cor:2.6e}
Let $\alpha,\bar\alpha\in \lcal^{\infty}$ and the generator $g$ satisfy Assumptions \ref{A:B1}, \ref{A:B2} and \ref{A:AB} with $\theta=0$. Then for each $\xi\in L^{\infty}(\R^n)$, BSDE \eqref{eq:2.1} admits a unique global solution $(Y,Z)\in \s^\infty(\R^n)\times {\rm BMO}(\R^{n\times d})$.
\end{cor}

\begin{rmk}\label{rmk:2.6e} We have the following remarks.
\begin{itemize}
\item [(i)] Compared with \citet[Theorem 3.5]{JacksonZitkovic2022SIAMC}, the Malliavin regular condition on the generator $g$ is not required in \cref{thm:2.5e} via a new proof.

\item [(ii)] By \cite[Theorem 2.1]{FanHuTang2023JDE} and \cref{pro:2.4e}, identical to  the proof of
    \citet[Theorem 4.1]{cheridito2015Stochastics}, we still have the assertion of \cref{cor:2.6e} when \ref{A:B1} and \ref{A:B2} are replaced with \ref{A:B1'} and \ref{A:B2'} in (iii) of \cref{rmk:2.2}.

\item [(iii)] The \ref{A:AB} condition of the generator $g$ seems difficult to be verified in the multi-dimensional case. In particular,  it seems to imply that  the generator is bounded in the state variable $y$, which is a too strong restriction. Consequently, in the sequel we will search for some better conditions on the generator $g$ for a unique global solution of systems of  BSDEs \eqref{eq:2.1}.
\end{itemize}
\end{rmk}

To obtain the second result on the global bounded solution of multi-dimensional BSDE \eqref{eq:2.1}, the assumption \ref{A:B1} needs to be strengthened to the following assumption \ref{A:C1a}. Before that, we introduce the following notations. For each $i=1,\cdots,n$ and any $M\in \R^n$ and $x\in \R$, denote by $M(x;i)$ the vector in $\R^n$ whose $i$th component is $x$ and whose $j$th component is $M^j$ for $j\neq i$. And, for each $i=1,\cdots,n$ and any $H\in \R^{n\times d}$ and $w\in \R^{1\times d}$, denote by $H(w;i)$ the matrix in $\R^{n\times d}$ whose $i$th row is $w$ and whose $j$th row is $H^j$ for $j\neq i$.

\begin{enumerate}
\renewcommand{\theenumi}{(C1a)}
\renewcommand{\labelenumi}{\theenumi}
\item\label{A:C1a} For each fixed $i=1,\cdots,n$, either of the following three conditions holds:
\begin{enumerate}
\item [(i)] The random field $f:=g^i$ or $f(\omega,t,y,z):=-g^i(\omega,t,y(-y^i;i),z(-z^i;i))$ satisfies that $\as$, for any $(y,z)\in \R^n\times\R^{n\times d}$,
$$
\begin{array}{l}
\Dis \frac{\bar \gamma}{2} |z^i|^2-\bar\alpha_t(\omega)-\beta |y|-\sum_{j=i+1}^{n} (\bar\lambda |z^j|^{1+\delta}+\theta |z^j|^2)-\bar c\sum_{j=1}^{i-1}|z^j|^2\leq f(\omega,t,y,z)\\
\quad\quad \leq  \Dis \alpha_t(\omega)+\left[\beta |y|\,  {\bf 1}_{y^i>0}+\phi(|y|)\, {\bf 1}_{y^i<0}\right]+\sum_{j\neq i} (\lambda |z^j|^{1+\delta \, {\bf 1}_{y^i<0}}+\theta |z^j|^2)+\frac{\gamma}{2} |z^i|^2;
\end{array}
$$

\item [(ii)] Almost everywhere in $\Omega\times\T$, for any $(y,z)\in \R^n\times\R^{n\times d}$, we have
$$
\begin{array}{l}
\Dis -\left[\beta |y^i|\,  {\bf 1}_{y^i<0}+\phi(|y|)\, {\bf 1}_{y^i>0}\right]-l_i(\omega,t,z)\leq g^i(\omega,t,y,z)\\[3mm]
\hspace{1.5cm}\Dis \leq \left[\beta |y^i|\,  {\bf 1}_{y^i>0}+\phi(|y|)\, {\bf 1}_{y^i<0}\right]+l_i(\omega,t,z)
\end{array}
$$
with
$$
l_i(\omega,t,z):=\tilde\alpha_t(\omega)+v_t(\omega)|z^i|+
c\sum_{j=1}^{i-1}|z^i(z^j)^\top|+\frac{\gamma}{2} |z^i|^2;
$$

\item [(iii)] Almost everywhere in $\Omega\times\T$, for any $(y,z)\in \R^n\times\R^{n\times d}$, we have
$$
\begin{array}{l}
\Dis -\left[\beta |y^i|\,  {\bf 1}_{y^i<0}+\phi(|y|)\, {\bf 1}_{y^i>0}\right]-\bar l_i(\omega,t,z)\leq g^i(\omega,t,y,z)\\[3mm]
\hspace{1.5cm}\Dis \leq \left[\beta |y^i|\,  {\bf 1}_{y^i>0}+\phi(|y|)\, {\bf 1}_{y^i<0}\right]+\bar l_i(\omega,t,z)
\end{array}
$$
with
$$
\bar l_i(\omega,t,z):=\bar\alpha_t(\omega)+\bar\lambda |z|+\theta\sum_{j\neq i}|z^j|^2.
$$
\end{enumerate}
\end{enumerate}

\begin{rmk}\label{rmk:2.4} Assumptions \ref{A:C1a} and \ref{A:B2} allow the generator $g(t,y,z)$ to have a general growth in $y$ and a general quadratic growth in  $z$. Assumption \ref{A:C1a} requires $g$ to be diagonally quadratic  in $z$ when $\theta=0$ and $c=\bar c=0$, and to be triangularly quadratic  in $z$ when $\theta=0$ and $c>0$ or $\bar c>0$. In addition, similar  to (iv) of \cref{rmk:2.2}, in Assumption \ref{A:C1a} we can assume without loss of generality that $g^i$ satisfies either of conditions \ref{A:C1a}(i), \ref{A:C1a}(ii) and \ref{A:C1a}(iii) for all $i=1,\cdots,n$.\vspace{0.2cm}
\end{rmk}

\begin{thm}\label{thm:2.5}
Let $\xi\in L^{\infty}(\R^n)$, $\alpha\in \ecal^{\infty}(p\gamma\exp(\beta T))$ for some $p>1$ and the generator $g$ satisfy Assumptions \ref{A:C1a} and \ref{A:B2}. Then, there exists a positive constant $\theta_0>0$ depending only on $(\|\xi\|_{\infty}, \|\alpha\|_{\ecal^{\infty}(p\gamma\exp(\beta T))}, \|\bar\alpha\|_{\mcal^{\infty}}, \|\tilde\alpha\|_{\lcal^{\infty}}, \|v\|_{{\rm BMO}}, n,\beta,\gamma,\bar\gamma,\lambda,\bar\lambda, c,\bar c, \delta, T, p)$ such that when $\theta\in [0,\theta_0]$, BSDE \eqref{eq:2.1} admits a unique global solution $(Y,Z)\in \s^\infty(\R^n)\times {\rm BMO}(\R^{n\times d})$. Moreover, the above assertion is still true for $p=1$ when $\lambda=0$ and $\theta_0=0$.\vspace{0.2cm}
\end{thm}

Note that \cref{thm:2.5} can be compared with Theorem 2.3 and Remark 2.5 of \citet{JamneshanKupperLuo2017ECP} to observe the role of $\theta$ in Assumptions \ref{A:C1a} and \ref{A:B2} for existence and uniqueness of the bounded solution of multi-dimensional interacting quadratic BSDEs.\vspace{0.3cm}

By virtue of \cref{thm:2.5} and (i) of \cref{rmk:2.1}, we immediately have

\begin{cor}\label{cor:2.6}
Let $\alpha,\bar\alpha\in \lcal^{\infty}$ and the generator $g$ satisfy Assumptions \ref{A:C1a} and \ref{A:B2} with $\theta=0$. Then, for each $\xi\in L^{\infty}(\R^n)$, BSDE \eqref{eq:2.1} admits a unique global solution $(Y,Z)\in \s^\infty(\R^n)\times {\rm BMO}(\R^{n\times d})$.
\end{cor}

\begin{rmk}\label{rmk:2.11*} If $g$ satisfies those assumptions in \citet[Theorem 2.3]{Luo2020EJP}, then Assumption \ref{A:C1a}(i) with $\lambda=\theta=0$ is true for each $g^i$ with $i=1,\cdots,n$, and Assumption \ref{A:B2} with $\theta=0$ is true. Consequently, \cref{cor:2.6} improves \citet[Theorem 2.3]{Luo2020EJP}. In particular,   the interacting term $z^i(z^j)^\top$ is excluded to  appear in $g^i$ for $1\leq j<i\leq n$ in \citet[Theorem 2.3]{Luo2020EJP}, but included  in \cref{cor:2.6}, see the next subsection for more details. In addition, we especially mention that multidimensional BSDEs with this type of interacting term arise from many problems, such as price impact models (see \citet{KramkovPulido2016AAP}), incomplete stochastic equilibrium (see \citet{XingZitkovic2018AoP,Kardaras2022SA,EscauriazaSchwarzXing2022AAP}, and \citet{
Weston2024FS}), and risk-sensitive nonzero-sum stochastic games (see \citet{XingZitkovic2018AoP,JacksonZitkovic2022SIAMC}, and \citet{Jackson2023SPA}).\vspace{0.2cm}
\end{rmk}

Finally, let us further demonstrate the third existence and uniqueness result on the global bounded solution. In stating it, the following assumption \ref{A:C1b} on the generator will be used, which is strictly stronger than \ref{A:B1} with $\theta=0$. It should be noted that Assumptions \ref{A:C1b} and \ref{A:C1a} with $\theta=0$ do not cover each other.

\begin{enumerate}
\renewcommand{\theenumi}{(C1b)}
\renewcommand{\labelenumi}{\theenumi}
\item\label{A:C1b} The set of integers $\{1,\cdots,n\}$ consists of  three mutually  disjoint subsets  $J_1$, $J_2$ and $J_3$, any of which can vanish, such that $J_1+J_2+J_3=\{1,\cdots,n\}$. For each fixed $i\in \{1,\cdots,n\}$, either of the following three conditions holds:
\begin{enumerate}
\item [(i)] If $i\in J_1$, then the random field $f:=g^i$ or $f(\omega,t,y,z):=-g^i(\omega,t,y(-y^i;i),z(-z^i;i))$ satisfies that $\as$, for any $(y,z)\in \R^n\times\R^{n\times d}$,
$$
\begin{array}{l}
\Dis \frac{\bar \gamma}{2} |z^i|^2-\tilde\alpha_t(\omega)-\beta |y|-\lambda \sum_{j\in J_1}|z^j|^{1+\delta}\leq f(\omega,t,y,z)\\
\quad\quad\quad \leq  \Dis \tilde\alpha_t(\omega)+\left [\beta |y|\,  {\bf 1}_{y^i>0}+\phi(|y|)\, {\bf 1}_{y^i<0}\right ]+\lambda \sum_{j\in J_1}|z^j|^{1+\delta}+\frac{\gamma}{2} |z^i|^2;
\end{array}
$$

\item [(ii)] If $i\in J_2$, then $\as$, for any $(y,z)\in \R^n\times\R^{n\times d}$, we have
$$
\begin{array}{l}
\Dis -\left[\beta |y|\,  {\bf 1}_{y^i<0}+\phi(|y|)\, {\bf 1}_{y^i>0}\right]-\tilde l_i(\omega,t,y,z)\leq g^i(\omega,t,y,z)\\[3mm]
\hspace{1.5cm}\Dis \leq \left[\beta |y|\,  {\bf 1}_{y^i>0}+\phi(|y|)\, {\bf 1}_{y^i<0}\right]+\tilde l_i(\omega,t,y,z)
\end{array}
$$
with
$$
\tilde l_i(\omega,t,y,z):=\tilde\alpha_t(\omega)+\left[v_t(\omega)+\phi(|y|)\right]|z^i|+
c\sum_{j=1}^{i-1}|z^i(z^j)^\top|+\frac{\gamma}{2} |z^i|^2;
$$

\item [(iii)] If $i\in J_3$, then $\as$, for any $(y,z)\in \R^n\times\R^{n\times d}$, we have
$$
\begin{array}{l}
\Dis -\left[\beta |y|\,  {\bf 1}_{y^i<0}+\phi(|y|)\, {\bf 1}_{y^i>0}\right]-\hat l(\omega,t,z)\leq g^i(\omega,t,y,z)\\[3mm]
\hspace{1.6cm}\Dis \leq \left[\beta |y|\,  {\bf 1}_{y^i>0}+\phi(|y|)\, {\bf 1}_{y^i<0}\right]+\hat l(\omega,t,z)
\end{array}
$$
with
$$
\hat l(\omega,t,z):=\tilde\alpha_t(\omega)+\lambda \sum_{j\in J_3}|z^j|.
$$
\end{enumerate}
\end{enumerate}

\begin{rmk}\label{rmk:2.7a} Similar to (iv) of \cref{rmk:2.2}, in Assumption \ref{A:C1b} we can assume without loss of generality  that $g^i$ satisfies either of conditions \ref{A:C1b}(i), \ref{A:C1b}(ii) and \ref{A:C1b}(iii) for all $i=1,\cdots,n$. In addition, in Assumption \ref{A:C1b}(i) it gives no essential difference to replace the sum $\sum_{j\in J_1}|z^j|^{1+\delta}$ with the sum $\sum_{j\in J_1,j\neq i}|z^j|^{1+\delta}$.\vspace{0.2cm}
\end{rmk}

\begin{thm}\label{thm:2.8a}
Let the generator $g$ satisfy Assumptions \ref{A:C1b} and \ref{A:B2} with $\theta=0$. Then, for each $\xi\in L^{\infty}(\R^n)$, BSDE \eqref{eq:2.1} admits a unique global solution $(Y,Z)\in \s^\infty(\R^n)\times {\rm BMO}(\R^{n\times d})$.
\end{thm}

\begin{rmk}\label{rmk:2.9a} We have the following remarks.
\begin{itemize}
\item [(i)] Assumptions \ref{A:C1b}(ii)-(iii) generalize Assumption (H3.2) of \citet{FanWangYong2022arXiv}. Consequently, \cref{thm:2.8a} improves \citet[Theorem 3.6]{FanWangYong2022arXiv},
    and then \citet[Theorems 2.4]{FanHuTang2023JDE} and
    \citet[Theorem 2.3]{HuTang2016SPA}.

\item [(ii)] In addition to~\ref{A:B1'} and \ref{A:B2'} in (ii) of \cref{rmk:2.2}, Theorem 2.5 of \citet{FanHuTang2023JDE} requires the following two assumptions:
\begin{enumerate}
\renewcommand{\theenumi}{(B3')}
\renewcommand{\labelenumi}{\theenumi}
\item\label{A:B3'} For each $i=1,\cdots, n$, $g^i$ satisfies that $\as$, for any $(y,z)\in \R^n\times\R^{n\times d}$,
$$
g^i(\omega,t,y,z)\, {\rm sgn}(y^i)\leq \tilde\alpha_t(\omega)+\beta |y|+\lambda\sum_{j\neq i}|z^j|^{1+\delta}+\frac{\gamma}{2} |z^i|^2;
$$
\end{enumerate}
\begin{enumerate}
\renewcommand{\theenumi}{(B4')}
\renewcommand{\labelenumi}{\theenumi}
\item\label{A:B4'} For $i=1,\cdots,n$, $f:=g^i$ or $f:=-g^i$ satisfies that $\as$,
$$
f(\omega,t,y,z)\geq \Dis\frac{\bar \gamma}{2} |z^i|^2-\tilde\alpha_t(\omega)-\beta |y|-\lambda\sum_{j\neq i}|z^j|^{1+\delta},\quad \RE (y,z)\in \R^n\times\R^{n\times d}.
$$
\end{enumerate}
Assumptions \ref{A:B1'}-\ref{A:B4'} coincides with \ref{A:C1b} where both $J_2$ and $J_3$ vanish. As a consequence, \cref{thm:2.8a} improves \citet[Theorem 2.5]{FanHuTang2023JDE}.

\item [(iii)] Let the assumptions of \cref{thm:2.8a} be satisfied except that $\theta=0$. It can also be proved that the assertion of \cref{thm:2.8a} is still true when $\theta$ is smaller than a constant $\theta_0>0$ depending only on $\|\xi\|_{\infty}$ and those parameters in Assumption \ref{A:C1b}.

\item [(iv)] Let $l\in (1,2]$, $\{(a_t,b_t)\}_{t\in\T}\in \s^\infty(\R^n\times \R^n)$ and the generator be defined as follows:
    $$g(\omega,t,y,z):=a_t(\omega)|y|+b_t(\omega)\sin (|z|^l),\ \ \RE (\omega,t,y,z)\in \Omega\times\T\times\R^n\times\R^{n\times d}.$$
    According to (iii) of this remark, we know that for each $\xi\in L^{\infty}(\R^n)$, there exists a constant $\theta_0>0$ such that when $\|b\|_{\s^\infty}{\bf 1}_{l=2}\leq \theta_0$, BSDE \eqref{eq:2.1} with this generator $g$ admits a unique global solution $(Y,Z)\in \s^\infty(\R^n)\times {\rm BMO}(\R^{n\times d})$, which seems to be new.
\end{itemize}
\end{rmk}

\subsection{Illustrating Examples\vspace{0.2cm}}

First of all, \citet{JamneshanKupperLuo2017ECP} addressed the following two-dimensional BSDE:
\begin{equation}\label{eq:2.4b}
\left\{
\begin{array}{l}
\Dis Y^1_t=\xi^1+\int_t^T \left(\theta_1|Z^1_s|^2+\vartheta_1|Z^2_s|^2\right){\rm d}s-\int_t^T Z^1_s{\rm d}B_s, \vspace{0.1cm}\\
\Dis Y^2_t=\xi^2+\int_t^T \left(\vartheta_2|Z^1_s|^2+\theta_2|Z^2_s|^2\right){\rm d}s-\int_t^T Z^2_s{\rm d}B_s,\quad t\in \T,
\end{array}
\right.
\end{equation}
where $\xi^i\in L^{\infty}(\R)$ and $\theta_i,\vartheta_i\in \R$, $i=1,2$. According to \cref{thm:2.5}, we see that if $\theta_1\theta_2\neq 0$, then for each $\xi\in L^\infty(\R^2)$, there exists a constant $\theta_0>0$ depending only on $(\theta_1,\theta_2,\|\xi\|_\infty)$ such that for each pair $(\vartheta_1,\vartheta_2)$ satisfying $|\vartheta_1|,|\vartheta_2|\leq \theta_0$, the system of BSDEs \eqref{eq:2.4b} admits a unique global solution $(Y,Z)\in \s^\infty(\R^2)\times {\rm BMO}(\R^{2\times d})$ on the time interval $\T$. This conclusion can be compared with Theorem 2.3 and Remark 2.5 posed in \citet{JamneshanKupperLuo2017ECP}.\vspace{0.2cm}

On the other hand, in the case of $\theta_1=\vartheta_1=0$, $\vartheta_2=1$ and $\theta_2=1/2$, \citet{FreiDosReis2011MFE} showed that for some $\xi\in L^\infty(\R^2)$, the system of BSDEs \eqref{eq:2.4b} fails to have a global bounded solution on the time interval $\T$, see Theorem 2.1 of \citet{FreiDosReis2011MFE} for more details. However, for the case of $\vartheta_1=0$ and $\vartheta_2\theta_2<0$ (for example, $\vartheta_2=1$ and $\theta_2=-1/2$), by \cref{cor:2.6} we know that for each $\xi\in L^\infty(\R^2)$, the system of BSDEs \eqref{eq:2.4b} admits a unique global solution $(Y,Z)\in \s^\infty(\R^2)\times {\rm BMO}(\R^{2\times d})$ on the time interval $\T$. \vspace{0.2cm}

Furthermore, we consider the following variant of the system of BSDEs \eqref{eq:2.4b}:
\begin{equation}\label{eq:2.5b}
\left\{
\begin{array}{l}
\Dis Y^1_t=\xi^1+\int_t^T \left(\theta_1|Z^1_s|^2+\vartheta_1|Z^2_s|\right){\rm d}s-\int_t^T Z^1_s{\rm d}B_s, \vspace{0.1cm}\\
\Dis Y^2_t=\xi^2+\int_t^T \left(\vartheta_2|Z^1_s|^2+\theta_2|Z^2_s|^2+l Z^1_s(Z^2_s)^\top \right){\rm d}s-\int_t^T Z^2_s{\rm d}B_s,\quad t\in \T,
\end{array}
\right.
\end{equation}
where $l\in \R$ and $\xi^i\in L^{\infty}(\R)$ and $(\theta_i,\vartheta_i)^T\in \R^2$ for $i=1,2$ such that $\theta_2\vartheta_2<0$. Without loss of generality, we assume that $\vartheta_2<0$ and $\theta_2>0$. Observe that
$$
-\frac{l^2}{2\theta_2}|Z^1_s|^2-\frac{\theta_2}{2}|Z^2_s|^2\leq l Z^1_s(Z^2_s)^\top \leq -\frac{\vartheta_2}{2}|Z^1_s|^2-\frac{l^2}{2\vartheta_2}|Z^2_s|^2
$$
and then
\begin{equation}\label{eq:2.5c}
\frac{\theta_2}{2}|Z^2_s|^2-(\frac{l^2}{2\theta_2}-\vartheta_2)|Z^1_s|^2\leq \vartheta_2|Z^1_s|^2+\theta_2|Z^2_s|^2+l Z^1_s(Z^2_s)^\top \leq (\theta_2-\frac{l^2}{2\vartheta_2})|Z^2_s|^2.\vspace{0.1cm}
\end{equation}
We get that the generator $f:=g^1$ satisfies \ref{A:C1a}(i) for $\theta_1\neq 0$ and \ref{A:C1a}(iii) for $\theta_1=0$, and $f:=g^2$ satisfies \ref{A:C1a}(i). It follows from \cref{cor:2.6} that for each $\xi\in L^\infty(\R^2)$, the system of BSDEs \eqref{eq:2.5b} admits a unique global solution $(Y,Z)\in \s^\infty(\R^2)\times {\rm BMO}(\R^{2\times d})$.
In addition, it is clear that the above assertion is still true as $\vartheta_2=0$.

We proceed by considering the following three-dimensional BSDE, which is another variant of the system of BSDEs \eqref{eq:2.4b}:
\begin{equation}\label{eq:2.5d}
\left\{
\begin{array}{l}
Y^1_t=\xi^1+\int_t^T \left(\vartheta_1|Z^1_s|^2+\theta_1 |Z^3_s|\right){\rm d}s-\int_t^T Z^1_s{\rm d}B_s, \vspace{0.1cm}\\
\Dis Y^2_t=\xi^2+\int_t^T \left(\vartheta_2|Z^1_s|^2+\theta_2|Z^2_s|^2+l_{21} Z^2_s(Z^1_s)^\top +k_2|Z^3_s| \right){\rm d}s-\int_t^T Z^2_s{\rm d}B_s,\vspace{0.1cm}\\
\Dis Y^3_t=\xi^3+\int_t^T \left(\vartheta_3|Z^1_s|^2+\theta_3|Z^2_s|^2+\kappa_3|Z^3_s|^2\right){\rm d}s-\int_t^T Z^3_s{\rm d}B_s\vspace{0.1cm}\\
\Dis \hspace{2cm} +\int_t^T \left(l_{31} Z^3_s(Z^1_s)^\top +l_{32} Z^3_s(Z^2_s)^\top +l_{33} Z^1_s(Z^2_s)^\top +k_3|Z^2_s|\right){\rm d}s; \quad t\in [0,T],
\end{array}
\right.
\end{equation}
where $\xi^i\in L^{\infty}(\R)$, $(\theta_i,\vartheta_i)^\top\in \R^2$ for $i=1,2,3$, and $(\kappa_3, l_{21},l_{31},l_{32},l_{33},k_2,k_3)^\top\in \R^7$ with $$\theta_2\vartheta_2<0,\quad  \kappa_3\theta_3<0,\quad \kappa_3\vartheta_3<0\quad  {\rm and}\quad  l^2_{33}<4\theta_3\vartheta_3.$$
Without loss of generality, we assume that $\theta_2>0$, $\vartheta_2<0$, $\kappa_3>0$, $\theta_3<0$ and $\vartheta_3<0$. Observe that there exists a unique real  number $\eps\in (0,1]$ such that
$$
l^2_{33}=4(1-\eps)^2\theta_3\vartheta_3\vspace{-0.2cm}
$$
and
$$
\vartheta_3(1-\eps)|Z^1_s|^2+\theta_3(1-\eps)|Z^2_s|^2+l_{33} Z^1_s(Z^2_s)^\top =-(1-\eps)\left|\sqrt{|\vartheta_3|}Z^1_s
-{\rm sgn}(l_{33})\sqrt{|\theta_3|}Z^2_s \right|^2.
$$
Combining a similar argument to \eqref{eq:2.5c}, we can directly verify that
$$
\Delta_s\leq \left(\kappa_3-\frac{l_{31}^2}{4\eps\vartheta_3}-\frac{l_{32}^2}{4\eps\theta_3}
\right)|Z^3_s|^2
$$
and
$$
\Delta_s\geq \frac{\kappa_3}{2} |Z^3_s|^2-\left(\frac{1}{2}+\frac{l_{31}^2}{\kappa_3}-\vartheta_3\right)
|Z^1_s|^2-\left(\frac{l_{33}^2}{2}+\frac{l_{32}^2}{\kappa_3}-\theta_3\right)
|Z^2_s|^2,
$$
where
$$
\Delta_s:=\vartheta_3|Z^1_s|^2+\theta_3|Z^2_s|^2+\kappa_3|Z^3_s|^2+l_{31} Z^3_s(Z^1_s)^\top +l_{32} Z^3_s(Z^2_s)^\top +l_{33} Z^1_s(Z^2_s)^\top .
$$
Consequently, the generator $g$ of the system of BSDEs \eqref{eq:2.5d} satisfies Assumption \ref{A:C1a}. It follows from \cref{cor:2.6} that for each $\xi\in L^\infty(\R^2)$, the system of BSDEs \eqref{eq:2.5d} admits a unique global solution $(Y,Z)\in \s^\infty(\R^2)\times {\rm BMO}(\R^{2\times d})$. In addition, it can also be proved that when $\vartheta_2=k_2=0, \kappa_3\theta_3<0, \kappa_3\vartheta_3<0, l_{31}=l_{32}=0$ and $l^2_{33}=4\theta_3\vartheta_3$ holds or $\vartheta_2=k_2=0$ and $\theta_3=\vartheta_3=l_{33}=k_3=0$ holds, \vspace{0.3cm}the above assertion is still true.

Finally, let us present several specific examples of multi-dimensional solvable BSDEs with interactively quadratic generators, to which one of \cref{thm:2.5}, \cref{cor:2.6} and \cref{thm:2.8a} applies, but no existing results could, to our best knowledge.\vspace{0.1cm}

\begin{ex}\label{ex:2.7} We have the following assertions.
\begin{itemize}
\item [(i)] Assume that the generator $g:=(g^1,\cdots,g^n)^\top$ is defined as follows: for each $i=1,\cdots,n$, and $(\omega,t,y,z)\in \Omega\times \T\times \R^n\times\R^{n\times d}$,
$$
g^i(\omega,t,y,z):=\tilde\alpha_t(\omega)+(-1)^i
\Big(\sum_{j=1}^{i-1}a_{ij}|z^j|^2-a_{ii}|z^i|^2
+\sum_{j=i+1}^{n}a_{ij}|z^j|^{1+\delta}\Big)+h^i(y,z),
$$
where any element of matrix $A:=(a_{ij})_{n\times n}$ is positive, and for each $i=1,\cdots,n$, $h^i(y,z):\R^n\times \R^{n\times d}\To \R$ is any Lipschitz continuous function. It is not hard to verify that this generator $g$ satisfies Assumption \ref{A:B2} with $\theta=0$, and for each $i=1,\cdots,n$, $f:=g^i$ or $f(\omega,t,y,z):=-g^i(\omega,t,y(-y^i;i),z(-z^i;i))$
satisfies Assumption \ref{A:C1a}(i) with $\theta=0$. Then, from \cref{cor:2.6} we can conclude that for each $\xi\in L^{\infty}(\R^n)$, BSDE \eqref{eq:2.1} with this generator $g$ admits a unique global solution $(Y,Z)\in \s^\infty(\R^n)\times {\rm BMO}(\R^{n\times d})$. Note that in Assumption \ref{A:C1a}(i), both terms  $|z^i|^2$ and $|z^j|^2$ for  $j<i$ share opposite signs.

\item [(ii)] Let $n=d$ and the generator $g(y,z):=zy$ for each $(y,z)\in \R^n\times \R^{n\times n}$. Then $g$ satisfies Assumption \ref{A:B2} with $\theta=0$, and for each $i=1,\cdots,n$, $g^i(y,z)=z^i y$ satisfies Assumption \ref{A:C1b}(ii). From \cref{thm:2.8a}, we see  that BSDE \eqref{eq:2.1} with this generator $g$ admits a unique global solution $(Y,Z)\in \s^\infty(\R^n)\times {\rm BMO}(\R^{n\times d})$ for each $\xi\in L^{\infty}(\R^n)$. Note that when the terminal value is Markovian, this BSDE is associated to the Burger system. Therefore,   it is reasonable to call it the backward stochastic Burger system or equations.

\item [(iii)] Assume that the generator $g:=(g^1,\cdots,g^n)^\top$ is defined as follows: for each $i=1,\cdots,n$, and $(\omega,t,y,z)\in \Omega\times \T\times \R^n\times\R^{n\times d}$,
$$
g^i(\omega,t,y,z):=\tilde\alpha_t(\omega)+\left(v_t(\omega)+e^{|y|}\right)|z^i|
+e^{-y_i}|z^i|^{\frac{3}{2}}+z^i\sum_{j=1}^{i}c_{ij} (z^j)^\top ,
$$
where $C=(c_{ij})_{n\times n}$ is any real matrix. It is easy to verify that this generator $g$ satisfies Assumption \ref{A:B2} with $\theta=0$, and for each $i=1,\cdots,n$, $f:=g^i$ satisfies Assumption \ref{A:C1b}(ii). Then, from~\cref{thm:2.8a}, we see that BSDE \eqref{eq:2.1} with this generator $g$ admits a unique global solution $(Y,Z)\in \s^\infty(\R^n)\times {\rm BMO}(\R^{n\times d})$  for each $\xi\in L^{\infty}(\R^n)$.

\item [(iv)] Let $n=5$ and $d=2$. Assume that for each $(\omega,t,y,z)\in \Omega\times \T\times \R^n\times\R^{n\times d}$, the generator $g$ is defined as follows:
$$
g(\omega,t,y,z):=\left(
\begin{array}{c}
g^1\vspace{0.1cm}\\
g^2\vspace{0.1cm}\\
g^3\vspace{0.1cm}\\
g^4\vspace{0.1cm}\\
g^5\vspace{0.1cm}
\end{array}
\right) (\omega,t,y,z):=\left(
\begin{array}{c}
e^{-y^1}-|y|+|z^1|^2-|z^2|^{\frac{4}{3}}+\sin|z^3|\vspace{0.1cm}\\
|y|\cos|y|-|z^2|^2+|z^1|^{\frac{5}{4}}-\cos|z^4|\vspace{0.1cm}\\
|y|+z^3(2z^1-3z^2)^\top +z^3A(z^3)^\top -\arcsin|z^5|\vspace{0.1cm}\\
2|y|\sin|y|+|z^4|-|z^5|+\arccos|z^1|\vspace{0.1cm}\\
y^1+3y^3-y^4+y^5-|z^4|+2|z^5|-\arctan|z^2|\vspace{0.1cm}
\end{array}
\right)
$$
with the matrix
$$
A:=\left(\begin{array}{cc}
1&1\\
0&0
\end{array}\right)
$$
It is easy to check that this generator $g$ satisfies \ref{A:B2} with $\theta=0$, and that $f:=g^i$ satisfies \ref{A:C1b}(i) for $i=1,2$, \ref{A:C1b}(ii) for $i=3$, and \ref{A:C1b}(iii) for $i=4,5$. Then, from~\cref{thm:2.8a}, we see that  BSDE \eqref{eq:2.1} with this generator $g$ admits a unique global solution $(Y,Z)\in \s^\infty(\R^n)\times {\rm BMO}(\R^{n\times d})$  for each $\xi\in L^{\infty}(\R^n)$.
\end{itemize}
\end{ex}

\subsection{Connections to existing results\vspace{0.2cm}}

With an invertible linear transformation as in~\citet{XingZitkovic2018AoP} and \citet{Weston2024FS}, we immediately have  the following result after an application of  It\^{o} formula.

\begin{pro}\label{pro:2.11b}
Let $A\in \R^{n\times n}$ be a real invertible matrix, and $A^{-1}$  its inverse. Then, $(Y,Z)\in \s^\infty(\R^n)\times {\rm BMO}(\R^{n\times d})$ is a solution of BSDE \eqref{eq:2.1} if and only if $(\bar Y,\bar Z):=(AY,AZ)\in \s^\infty(\R^n)\times {\rm BMO}(\R^{n\times d})$ is a solution of the following BSDE
\begin{equation}\label{eq:2.6b}
\bar Y_t=\bar \xi+\int_t^T \bar g(s,\bar Y_s,\bar Z_s){\rm d}s-\int_t^T \bar Z_s {\rm d}B_s, \quad \ t\in\T,
\end{equation}
where $\bar\xi:=A\xi$ and
$$
\bar g(\omega,t,\bar y,\bar z):= A g(\omega,t,A^{-1}\bar y,A^{-1}\bar z), \quad \forall (\omega,t,\bar y,\bar z)\in \Omega\times \T\times \R^n\times \R^{n\times d}.
$$
\end{pro}

Since the condition \ref{A:AB} is invariant under an invertible linear transformation of $\R^n$, as shown in Remark 2.12 of \citet{XingZitkovic2018AoP},  we  immediately have from \cref{cor:2.6e} and \cref{pro:2.11b} the following assertion.

\begin{thm}\label{thm:2.11e}
Let $\alpha,\bar\alpha, \tilde\alpha\in \lcal^{\infty}$ and the generator $g$ satisfy \ref{A:AB}. If there is a real invertible matrix $A\in \R^{n\times n}$ such that the generator $\bar g$ defined in \cref{pro:2.11b} satisfies \ref{A:B1} and \ref{A:B2} with $\theta=0$, then for each $\xi\in L^{\infty}(\R^n)$, the system of BSDEs \eqref{eq:2.1} admits a unique global solution $(Y,Z)\in \s^\infty(\R^n)\times {\rm BMO}(\R^{n\times d})$.\vspace{0.1cm}
\end{thm}

From \cref{pro:2.11b}, we easily see the following slight extension of Theorem 2.1 of \citet{FreiDosReis2011MFE} .

\begin{thm}\label{thm:2.12d}
Let $\theta_1=\vartheta_1=0$ and $\theta_2\vartheta_2>0$. Then, there is a terminal value $\xi\in L^\infty(\R^2)$ such that the system of BSDEs \eqref{eq:2.4b} has no global solution $(Y,Z)\in \s^\infty(\R^2)\times {\rm BMO}(\R^{2\times d})$.
\end{thm}

\begin{proof}
We use a contradiction argument. Suppose that for any $\xi\in L^\infty(\R^2),$  the system of BSDEs \eqref{eq:2.4b} admits a global solution $(Y,Z)\in \s^\infty(\R^2)\times {\rm BMO}(\R^{2\times d})$. Define
$$
(\bar Y, \bar Z):=\left(\begin{array}{cc}
\sqrt{2\theta_2\vartheta_2}& 0\\
0& 2\vartheta_2
\end{array}
\right) (Y,Z).
$$
Then, we have from \cref{pro:2.11b} that $(\bar Y, \bar Z)\in \s^\infty(\R^2)\times {\rm BMO}(\R^{2\times d})$ solves the following system of BSDEs
$$
\left\{
\begin{array}{l}
\Dis \bar Y^1_t=\sqrt{2\theta_2\vartheta_2}\xi^1-\int_t^T \bar Z^1_s{\rm d}B_s,\vspace{0.1cm}\\
\Dis \bar Y^2_t=2\vartheta_2 \xi^2+\int_t^T \left(|\bar Z^1_s|^2+\frac{1}{2}|\bar Z^2_s|^2\right){\rm d}s-\int_t^T \bar Z^2_s{\rm d}B_s,\quad t\in\T,
\end{array}
\right.
$$
which is a contradiction to  Theorem 2.1 of \citet{FreiDosReis2011MFE} as $\xi$ is arbitrary. The proof is then complete.
\end{proof}

The following theorem is a nonlinear extension of Theorem 6.9 of \citet{Jackson2023SPA} with a different proof.

\begin{thm}\label{thm:2.12b} Define the generator $g$ as follows: $\RE (\omega,t,y,z)\in \Omega\times \T\times \R^n\times\R^{n\times d}$,
$$
g(\omega,t,y,z):=f(\omega,t,y,z)+zh(b^\top z),
$$
where $b:=(b_1,\cdots,b_n)^\top\in \R^n$ with $b_1\neq 0$, the vector function $h:\R^{1\times d}\To \R^d$ is Lipschitz continuous, and $f:\Omega\times \T\times \R^n\times\R^{n\times d}\To \R^n$ has the following stronger continuity than Assumption \ref{A:B2}: $\as$, for each $(y,\bar y,z,\bar z)\in \R^{n}\times\R^{n}\times \R^{n\times d}\times \R^{n\times d}$,
\begin{equation}\label{eq:2.7e}
|f(\omega,t,y,z)-f(\omega,t,\bar y,\bar z)|\leq \phi(|y|\vee |\bar y|)\left(v_t(\omega)+|z|^\delta+|\bar z|^\delta\right)\left(|y-\bar y|+|z-\bar z|\right),
\end{equation}
and the following growth: $\as$,
\begin{equation}\label{eq:2.7b}
|b^\top f(\omega,t,y,z)|\leq \tilde\alpha_t(\omega)+\beta |y|+\frac{\gamma}{2} |b^\top z|^2,\quad\forall (y,z)\in \R^{n}\times \R^{n\times d}
\end{equation}
and
\begin{equation}\label{eq:2.8b}
|f^i(\omega,t,y,z)|\leq \tilde\alpha_t(\omega)+\beta |y|+\frac{\gamma}{2}|z^i|^2, \quad i=2,\cdots,n.
\end{equation}
Then, for each $\xi\in L^{\infty}(\R^n)$, BSDE \eqref{eq:2.1} with the above generator $g$ admits a unique global solution $(Y,Z)\in \s^\infty(\R^n)\times {\rm BMO}(\R^{n\times d})$.\vspace{0.2cm}
\end{thm}

\begin{proof}
The following matrix
\begin{equation}\label{eq:2.11g}
A:=\left(
\begin{array}{cccc}
b_1& b_2& \cdots & b_n\\
& 1& & \\
&&\ddots &\\
&&& 1
\end{array}
\right)
\end{equation}
 is invertible. For each $(\bar y,\bar z)\in \R^{n}\times\R^{n\times d}$, let $(y,z):=A^{-1}(\bar y, \bar z) $, and then $(\bar y, \bar z)=A(y,z)$. Clearly, $\bar y^1=b^\top y$, $\bar z^1=b^\top z$, $\bar y^i=y^i$ and $\bar z^i=z^i$ for each $i=2,\cdots,n$.
 By \eqref{eq:2.7e}-\eqref{eq:2.8b} together with the Lipschitz continuity of $h$,
  the following generator
$$
\bar g(\omega,t,\bar y,\bar z):=A g(\omega,t,A^{-1}\bar y,A^{-1}\bar z)
= Af(\omega,t,A^{-1}\bar y,A^{-1}\bar z)+\left(
\begin{array}{c}
\bar z^1 h(\bar z^1)\\
\bar z^2 h(\bar z^1)\\
\vdots\\
\bar z^n h(\bar z^1)
\end{array}
\right)
$$
satisfies Assumption \ref{A:B2} with $\theta=0$, and Assumption \ref{A:C1b}(ii).
From \cref{thm:2.8a} we see that for each $\xi\in L^\infty(\R^n)$, the system of BSDEs \eqref{eq:2.6b} admits a unique global solution $(\bar Y,\bar Z)\in \s^\infty(\R^n)\times {\rm BMO}(\R^{n\times d})$. Finally, by \cref{pro:2.11b} we know that $(A^{-1}\bar Y,A^{-1}\bar Z)\in \s^\infty(\R^n)\times {\rm BMO}(\R^{n\times d})$ is just the desired unique solution of \eqref{eq:2.1}. The proof is complete.\vspace{0.2cm}
\end{proof}

\begin{rmk}\label{rmk:2.13b}
If the generator $f$ is Lipschitz continuous in the last two variables $(y,z)$ and bounded in the last variable $z$, then \eqref{eq:2.7e}-\eqref{eq:2.8b} hold naturally. On the other hand, if $b_1=0$ but $b_{i_0}\neq 0$ for some $i_0>1$, \cref{thm:2.12b} is still true when \eqref{eq:2.8b} is satisfied  for all integer $i\not=i_0$ instead of for all $i\ge 2$, just using an obvious  invertible  transformation in the proof.\vspace{0.2cm}
\end{rmk}

Similar to the proof of \cref{thm:2.12b}, we have the following \cref{thm:2.14b}.

\begin{thm}\label{thm:2.14b} Consider  the following generator $g$: $\RE (\omega,t,y,z)\in \Omega\times \T\times \R^n\times\R^{n\times d}$,
$$
g(\omega,t,y,z):=f(\omega,t,y,z)+zh\left(b^\top z \right)+a\bar h\left(b^\top z \right),
$$
where  $a:=(a_1,\cdots, a_n)^\top\in \R^n$ and $b:=(b_1,\cdots, b_n)^\top \in \R^n$ with $a_1\neq 0$ and $b^\top a\neq 0$,
the vector function $h:\R^{1\times d}\To \R^d$ is Lipschitz continuous, the function $\bar h:\R^{1\times d}\To \R$ satisfies that $\bar h(0)=0$ and
\begin{equation}\label{eq:2.11f}
|\bar h(w_1)-\bar h(w_2)|\leq \gamma (1+|w_1|+|w_2|)|w_1-w_2|,\ \ \RE w_1,w_2\in \R^{1\times d},
\end{equation}
and  $f:\Omega\times \T\times \R^n\times\R^{n\times d}\To \R^n$ has the continuity~\eqref{eq:2.7e} and the following growth: $\as$, for each $(y,z)\in \R^{n}\times\R^{n\times d}$,
\begin{equation}\label{eq:2.11b}
|f(\omega,t,y,z)|\leq \tilde\alpha_t(\omega)+\beta |y|.
\end{equation}
Then, for each $\xi\in L^{\infty}(\R^n)$, BSDE \eqref{eq:2.1} with the above generator $g$ admits a unique global solution $(Y,Z)\in \s^\infty(\R^n)\times {\rm BMO}(\R^{n\times d})$.
\end{thm}

It is quite related  to Theorem 6.19 of \citet{Jackson2023SPA}, which  requires the a priori boundedness condition \ref{A:AB} and the Malliavin regular condition on the generator $g$. However, both do not cover each other.

\begin{proof}
The following matrix
$$
A:=\left(
\begin{array}{ccccc}
b_1& b_2& b_3& \cdots & b_n\\
-a_2& a_1& & &\\
-a_3&& a_1&& \\
\vdots &&&\ddots & \\
-a_n&&&& a_1
\end{array}
\right)
$$
 has the determinant ${\rm det} (A)=a_1^{n-2}b^\top a\neq 0$, and is thus invertible.  For each $(\bar y,\bar z)\in \R^{n}\times\R^{n\times d}$, let $(y,z):=A^{-1}(\bar y, \bar z) $, and then $(\bar y, \bar z)=A(y,z)$. Clearly, $\bar y^1=b^\top y$, $\bar z^1=b^\top z$, $\bar y^i=-a_iy^1+a_1y^i$ and $\bar z^i=-a_iz^1+a_1z^i$ for each $i=2,\cdots,n$.  From \eqref{eq:2.7e} and \eqref{eq:2.11f}-\eqref{eq:2.11b} together with the Lipschitz continuity of $h$, we see that the following generator
$$
\begin{array}{lrl}
\Dis \bar g(\omega,t,\bar y,\bar z)&:= & \Dis Af(\omega,t,A^{-1}\bar y,A^{-1}\bar z)+\left(
\begin{array}{c}
\Dis \bar z^1 h(\bar z^1)+b^\top a \bar h (\bar z^1)\\
\Dis \bar z^2 h(\bar z^1)\\
\Dis \vdots\\
\Dis \bar z^n h(\bar z^1)
\end{array}
\right)
\end{array}
$$
satisfies Assumption \ref{A:B2} with $\theta=0$, and Assumption \ref{A:C1b}(ii).
Proceeding identically  as in \cref{thm:2.12b}, we have the desired assertion. \vspace{0.2cm}
\end{proof}

Furthermore, we have the following theorem, which  and Theorem 3.1 of \citet{XingZitkovic2018AoP} do not cover each other.

\begin{thm}\label{thm:2.14g}
Define the generator $g:=(g^1,\cdots,g^n)^\top$ as follows: for each $(\omega,t,y,z)\in \Omega\times \T\times \R^n\times\R^{n\times d}$, \vspace{-0.1cm}
\begin{equation}\label{eq:2.13g}
\left\{
\begin{array}{l}
\Dis g^1(\omega,t,y,z):=f^1(\omega,t,y,z)+z^1 h(b^\top z)-\bar h_1(b^\top z)-\frac{1}{b_1}\sum_{j=2}^{n}
a_jb_j|z^j|^2;\vspace{0.2cm}\\
\Dis g^i(\omega,t,y,z):=f^i(\omega,t,y,z)+z^i h(b^\top z)-\bar h_i(b^\top z)+a_i |z^i|^2,\ \ i=2,\cdots,n,
\end{array}
\right.\vspace{0.1cm}
\end{equation}
where  $a:=(0,a_2,\cdots, a_n)^\top\in \R^n$ and $b:=(b_1,\cdots, b_n)^\top \in \R^n$ with $b_1\neq 0$,
the vector function $h:\R^{1\times d}\To \R^d$ is Lipschitz continuous with Lipschitz constant $L>0$ and $|h(0)|\leq L$, the vector function $\bar h=(\bar h_1,\cdots,\bar h_n)^\top:\R^{1\times d}\To \R^n$ satisfies that $\bar h(0)=0$ and
\begin{equation}\label{eq:2.14g}
|\bar h(w_1)-\bar h(w_2)|\leq L(1+|w_1|+|w_2|)|w_1-w_2|,\ \ \RE w_1,w_2\in \R^{1\times d},
\end{equation}
and $f=(f^1,\cdots,f^n)^\top:\Omega\times \T\times \R^n\times\R^{n\times d}\To \R^n$ satisfies that $\as$, $|f(\omega,t,0,0)|\leq \tilde\alpha_t(\omega)$ and for each $(y_1,y_2,z_1,z_2)\in \R^{n}\times\R^{n}\times\R^{n\times d}\times\R^{n\times d}$,
\begin{equation}\label{eq:2.12g}
|f(\omega,t,y_1,z_1)-f(\omega,t,y_2,z_2)|\leq \beta |y_1-y_2|+\gamma |z_1-z_2|.
\end{equation}
Assume further that
\begin{equation}\label{eq:2.15g}
\begin{array}{l}
{\rm either} \quad \inf_{w\in \R^{1\times d}} \left[wh(w)-b^\top \bar h(w)-\frac{\bar\gamma}{2}|w|^2\right]\ge -c\quad {\rm or}\vspace{0.1cm}\\
\hspace*{1.4cm} \sup_{w\in \R^{1\times d}} \left[wh(w)-b^\top \bar h(w)+\frac{\bar\gamma}{2}|w|^2\right]\le c,
\end{array}
\end{equation}
and for each $i=2,\cdots,n$,
\begin{equation}\label{eq:2.16g}
\begin{array}{l}
{\rm either }\ \ a_i>0\ {\rm and} \ \inf_{w\in \R^{1\times d}}\left[ \bar h_i(w)- \frac{\bar\gamma}{2}|w|^2\right]\ge -c\ \ {\rm or}\\
\hspace*{1.25cm} a_i<0\ {\rm and}\ \sup_{w\in \R^{1\times d}}\left[\bar h_i(w) +\frac{\bar\gamma}{2}|w|^2\right]\le c.
\end{array}
\end{equation}
Then, for each $\xi\in L^{\infty}(\R^n)$, BSDE \eqref{eq:2.1} with the above  generator $g$ admits a unique global solution $(Y,Z)\in \s^\infty(\R^n)\times {\rm BMO}(\R^{n\times d})$.
\end{thm}

\begin{proof} We only give the proof when the first conditions of \eqref{eq:2.15g} and \eqref{eq:2.16g} are satisfied. The other cases are proved identically. \vspace{0.1cm}

Let the invertible matrix $A$ be defined in \eqref{eq:2.11g}. For each $(\bar y,\bar z)\in \R^{n}\times\R^{n\times d}$, let $(y,z):=A^{-1}(\bar y, \bar z) $, and then $(\bar y, \bar z)=A(y,z)$. i.e., $\bar y^1=b^\top y$, $\bar z^1=b^\top z$, $\bar y^i=y^i$ and $\bar z^i=z^i$ for each $i=2,\cdots,n$. From \eqref{eq:2.14g}, \eqref{eq:2.12g} and the Lipschitz continuity of $h$, the following generator
\begin{equation}\label{eq:2.23g}
\begin{array}{lrl}
\Dis \bar g(\omega,t,\bar y,\bar z)&:=& \Dis Af(\omega,t,A^{-1}\bar y,A^{-1}\bar z)+\left(
\begin{array}{c}
\Dis \bar z^1 h(\bar z^1)-b^\top \bar h(\bar z^1)\\
\Dis \bar z^2 h(\bar z^1)-\bar h_i(\bar z^1)+a_i |\bar z^2|^2\\
\Dis \vdots\\
\Dis \bar z^n h(\bar z^1)-\bar h_n(\bar z^1)+a_n |\bar z^n|^2
\end{array}
\right)
\end{array}
\end{equation}
satisfies Assumption \ref{A:B2} with $\theta=0$. Furthermore, we have from \eqref{eq:2.12g} that
\begin{equation}\label{eq:2.24g}
|A f(\omega,t,A^{-1}\bar y,A^{-1}\bar z)|\leq |A|\alpha_t(\omega)+\beta |\bar y|+\gamma |\bar z|.
\end{equation}
From~\eqref{eq:2.14g}, \eqref{eq:2.15g}, \eqref{eq:2.23g} and \eqref{eq:2.24g} together with the Lipschitz continuity of $h$,  we see that $\bar g^1$ satisfies Assumption \ref{A:C1a}(i) with $\theta=0$ since
$$
\bar g^1(\omega,t,\bar y,\bar z)\geq \frac{\bar\gamma}{2}|\bar z^1|^2-c-|A|\alpha_t(\omega)-\beta |\bar y|-\gamma |\bar z|\vspace{-0.2cm}
$$
and
$$
\bar g^1(\omega,t,\bar y,\bar z)\leq |A|\alpha_t(\omega)+\beta |\bar y|+\gamma |\bar z|+L|\bar z^1|(1+|\bar z^1|)+L|b|(1+|\bar z^1|)|\bar z^1|.\vspace{0.1cm}
$$
From \eqref{eq:2.14g}, \eqref{eq:2.16g}, \eqref{eq:2.23g} and \eqref{eq:2.24g} together with the Lipschitz continuity of $h$ and Young's inequality,  we see that for each $i=2,\cdots,n$, $\bar g^i$ satisfies \ref{A:C1a}(i) with $\theta=0$ since
$$
\begin{array}{lll}
\Dis \bar g^i(\omega,t,\bar y,\bar z)&\geq & \Dis a_i|\bar z^i|^2-L|\bar z^i|(1+|\bar z^1|)-L(1+|\bar z^1|)|\bar z^1|-|A|\alpha_t(\omega)-\beta |\bar y|-\gamma |\bar z| \vspace{0.2cm}\\
&\geq & \Dis \frac{a_i}{2}|\bar z^i|^2-\left(\frac{L^2}{2a_i}+L\right)|\bar z^1|^2-L|\bar z^i|-L|\bar z^1|-|A|\alpha_t(\omega)-\beta |\bar y|-\gamma |\bar z|
\end{array}
$$
and
$$
\begin{array}{lll}
\Dis \bar g^i(\omega,t,\bar y,\bar z)&\leq & \Dis a_i|\bar z^i|^2+L|\bar z^i|(1+|\bar z^1|)-\frac{\bar\gamma}{2}|\bar z^1|^2+c+|A|\alpha_t(\omega)+\beta |\bar y|+\gamma |\bar z|\vspace{0.2cm}\\
&\leq & \Dis \left(a_i+\frac{L^2}{2\bar\gamma}\right)|\bar z^i|^2+L|\bar z^i|+c+|A|\alpha_t(\omega)+\beta |\bar y|+\gamma |\bar z|.\vspace{0.1cm}
\end{array}
$$
Consequently, we have  from \cref{cor:2.6} that for each $\xi\in L^\infty$, the system of BSDEs \eqref{eq:2.6b} with the generator $\bar g$ defined in \eqref{eq:2.23g} admits a unique global solution $(\bar Y,\bar Z)\in \s^\infty(\R^n)\times {\rm BMO}(\R^{n\times d})$. Finally, by \cref{pro:2.11b} we know that $(A^{-1}\bar Y,A^{-1}\bar Z)\in \s^\infty(\R^n)\times {\rm BMO}(\R^{n\times d})$ is just the desired unique solution of the system of BSDEs~\eqref{eq:2.1}. The proof is complete.\vspace{0.2cm}
\end{proof}

The following proposition and its corollary  partially answer the problem 6.25 of \citet{Jackson2023SPA}.\vspace{-0.1cm}

\begin{pro}\label{pro:2.16b}
Assume that $n=2$ and $d=1$. Consider  the following generator $g$
$$
g(z):=\left(
\begin{array}{l}
g^1(z)\\
g^2(z)
\end{array}
\right)=\left(
\begin{array}{l}
\Dis z^\top A_1z+z^\top k_1+l_1\vspace{0.2cm}\\
\Dis z^\top A_2z+z^\top k_2+l_2
\end{array}
\right), \quad z\in \R^2,
$$
where $A_i\in \R^{2\times 2}$, $k_i\in \R^2$ and $l_i\in \R$ for $i=1,2$. Assume further that there exist three constants $a,b,\iota\in \R$ with $a\neq 0$ such that
\begin{equation}\label{eq:2.13b}
aA_1+bA_2=\iota\left(\begin{array}{c}
a\\
b
\end{array}\right)(a,b)=\left(
\begin{array}{cc}
a^2\iota & ab\iota\\
ab\iota & b^2\iota
\end{array}
\right)
\end{equation}
and either $\alpha_{11}=0$ and $\alpha_{22}\neq 0$ or $\alpha_{11}\alpha_{22}<0$ with
$$
\begin{pmatrix} \alpha_{11} & \alpha_{12}\\
\alpha_{12}&\alpha_{22}
\end{pmatrix}:= \frac{1}{a^2}\  \left(
\begin{array}{cc}
1 & 0\\
-b & a
\end{array}
\right)A_2\left(
\begin{array}{cc}
1 & -b\\
0 & a
\end{array}
\right)
$$
Then for each $\xi\in L^{\infty}(\R^2)$, BSDE \eqref{eq:2.1} with the above generator $g$ admits a unique global solution $(Y,Z)\in \s^\infty(\R^2)\times {\rm BMO}(\R^{2\times 1})$.\vspace{0.2cm}
\end{pro}

\begin{proof}
For simplicity, we only prove the case of $k_1=k_2=0$ and $l_1=l_2=0$.  Let
$$
A:=\left(
\begin{array}{cc}
a & b \\
0& 1
\end{array}
\right).
$$
For each $\bar z\in \R^{2}$, let
$$
z:=A^{-1}\bar z=\frac{1}{a}\left(
\begin{array}{cc}
1 & -b \\
0& a
\end{array}
\right)\bar z.
$$
Then, in view of \eqref{eq:2.13b}, we have
$$
\begin{array}{lrl}
\bar g^1(\bar z)&:=& \Dis ag^1(z)+bg^2(z)=\frac{1}{a^2}\ \bar z^\top \left(
\begin{array}{cc}
1 & 0\\
-b & a
\end{array}
\right)\left(aA_1+bA_2\right)\left(
\begin{array}{cc}
1 & -b\\
0 & a
\end{array}
\right)\bar z=\iota \ (\bar z^1)^2
\end{array}
$$
and
$$
\bar g^2(\bar z)=g^2(z)=z^\top A_2z=\frac{1}{a^2}\ \bar z^\top \left(
\begin{array}{cc}
1 & 0\\
-b & a
\end{array}
\right)A_2\left(
\begin{array}{cc}
1 & -b\\
0 & a
\end{array}
\right)\bar z.\vspace{0.2cm}
$$
According to the conditions of \cref{pro:2.16b} combined with an argument similar to \eqref{eq:2.5c}, we know that $\bar g$ satisfies Assumptions \ref{A:B2} and \ref{A:C1b} with $\theta=0$. Proceeding identically  as in \cref{thm:2.12b}, we have the desired assertion.
\end{proof}

The following result is a direct consequence of \cref{pro:2.16b}.

\begin{cor}\label{pro:2.15b}
Assume that $n=2$ and $d=1$. Consider  the following generator $g$
$$
g(z):=\left(
\begin{array}{l}
g^1(z)\\
g^2(z)
\end{array}
\right)=\left(
\begin{array}{l}
\Dis z^1(z^1+z^2)-\frac{\alpha}{2}(z^1)^2\vspace{0.2cm}\\
\Dis z^2(z^1+z^2)-\frac{\beta}{2}(z^2)^2
\end{array}
\right), \quad z\in \R^2, \vspace{0.1cm}
$$
where $\alpha, \beta\in \R^*$ are two constants. If $1/\alpha+1/\beta=1$, then for each $\xi\in L^{\infty}(\R^2)$, BSDE \eqref{eq:2.1} with this generator $g$ admits a unique global solution $(Y,Z)\in \s^\infty(\R^2)\times {\rm BMO}(\R^{2\times 1})$.
\end{cor}

\begin{proof}
In \cref{pro:2.16b}, let $k_1=k_2=0$, $l_1=l_2=0$,
$$
A_1:=\left(
\begin{array}{cc}
\Dis 1-\frac{\alpha}{2} & \Dis \frac{1}{2}\vspace{0.2cm}\\
\Dis \frac{1}{2} & 0
\end{array}
\right)\ \ {\rm and}\ \
A_2:=\left(
\begin{array}{cc}
0 & \Dis \frac{1}{2} \vspace{0.2cm}\\
\Dis \frac{1}{2}& \Dis 1-\frac{\beta}{2}
\end{array}
\right).
$$
There are two cases: $\alpha=2$ and $\alpha\neq 2$. In the case of $\alpha=2$, we have $A_1=A_2$. When $a=1$, $b=-1$ and $\iota=0$, we have \eqref{eq:2.13b}, $\alpha_{11}=0$ and $\alpha_{22}=1$, and thus all assumptions of \cref{pro:2.16b} are satisfied. In the other case of $\alpha\neq 2$, we take $a=1-\alpha/2\neq 0$, $b=1-\beta/2$ and $\iota=1$. Since $\alpha+\beta=\alpha\beta$, we have
$$
\frac{a+b}{2}=1-\frac{\alpha+\beta}{4}
=1-\frac{\alpha+\beta}{2}+\frac{\alpha\beta}{4}
=\left(1-\frac{\alpha}{2}\right)\left(1-\frac{\beta}{2}\right)=ab.
$$
It can also be directly verified that all of \eqref{eq:2.13b}, $\alpha_{11}=0$ and $\alpha_{22}\neq 0$ hold, and then all assumptions of \cref{pro:2.16b} are satisfied. Thus, the desired assertion follows immediately from \cref{pro:2.16b}.\vspace{0.2cm}
\end{proof}

\subsection{Global unbounded solution of multi-dimensional quadratic BSDEs\vspace{0.2cm}}

In this subsection, we will present three existence and uniqueness results on the global unbounded solution of multi-dimensional BSDE \eqref{eq:2.1}. For this, let us first introduce the following assumption \ref{A:D2}, which is strictly stronger than the previous assumption \ref{A:B2}. In particular, the generator $g$ satisfying \ref{A:B2} can have a general growth in the state variable $y$, while the generator $g$ satisfying \ref{A:D2} can only have a linear growth in $y$.

\begin{enumerate}
\renewcommand{\theenumi}{(D2)}
\renewcommand{\labelenumi}{\theenumi}
\item\label{A:D2} For $i=1,\cdots,n$, $g^i$ satisfies that $\as$, for each $(y,\bar y, z, \bar z)\in \R^n\times\R^n\times\R^{n\times d}\times\R^{n\times d}$,
    $$
    \begin{array}{l}
    \Dis |g^i(\omega,t,y,z)-g^i(\omega,t,\bar y,\bar z)|\leq \Dis \Dis  \gamma\left(v_t(\omega)+|z|+|\bar z|\right)\Big(|y-\bar y|+\sum_{j=1}^i |z^i-\bar z^i|\Big)\\
     \Dis \hspace{3.5cm}+\left[\gamma \left([v_t(\omega)]^\delta+|z|^\delta+|\bar z|^\delta\right)+\theta \left(v_t(\omega)+|z|+|\bar z|\right)\right]\sum_{j=i+1}^n |z^j-\bar z^j|.\vspace{0.2cm}
    \end{array}
    $$
\end{enumerate}

The following is the first existence and uniqueness result of this subsection.

\begin{thm}\label{thm:2.22e}
Let the generator $g$ satisfy Assumptions \ref{A:B1}, \ref{A:D2} and \ref{A:AB} with $\theta=0$ and $\phi(\cdot)\equiv 0$, and let
$$\xi=\bar\xi+\int_0^T H_s {\rm d}B_s$$
with $\bar\xi\in L^{\infty}(\R^n)$ and $H\in {\rm BMO}(\R^{n\times d})$. When $g^i$ satisfies (ii) in Assumption \ref{A:B1} and violates \ref{A:B1} (i) and (iii), and $-g^i$ violates \ref{A:B1} (i),
we further assume that $H^i\equiv 0$. If $\alpha,\bar\alpha, |H|^2\in \lcal^{\infty}$, then BSDE \eqref{eq:2.1} admits a unique global solution $(Y,Z)$ such that
$$
\Big(Y-\int_0^\cdot H_s {\rm d}B_s, Z\Big)\in \s^\infty(\R^n)\times {\rm BMO}(\R^{n\times d}).\vspace{0.2cm}
$$
\end{thm}

The second existence and uniqueness result requires the following assumption \ref{A:D1}, which is strictly stronger than the previous assumption \ref{A:C1a},  for $g$ admits a general growth in the state variable $y$ in \ref{A:C1a},  instead of only a linear growth in $y$ in \ref{A:D1}.

\begin{enumerate}
\renewcommand{\theenumi}{(D1)}
\renewcommand{\labelenumi}{\theenumi}
\item\label{A:D1} For each fixed $i=1,\cdots,n$, either of the following three conditions holds:
\begin{enumerate}
\item [(i)] $f:=g^i$ or $f:=-g^i$ satisfies that $\as$, for any $(y,z)\in \R^n\times\R^{n\times d}$,
$$
\begin{array}{l}
\Dis \frac{\bar \gamma}{2} |z^i|^2-\bar\alpha_t(\omega)-\beta |y|-\sum_{j=i+1}^{n} (\bar\lambda |z^j|^{1+\delta}+\theta |z^j|^2)-\bar c\sum_{j=1}^{i-1}|z^j|^2 \leq f(\omega,t,y,z)\\
\quad\quad\quad \leq  \Dis \alpha_t(\omega)+\beta |y|+\sum_{j\neq i} (\lambda |z^j|+\theta |z^j|^2)+\frac{\gamma}{2} |z^i|^2;
\end{array}
$$

\item [(ii)] Almost everywhere in $\Omega\times\T$, for any $(y,z)\in \R^n\times\R^{n\times d}$, we have
$$
|g^i(\omega,t,y,z)|\leq \tilde\alpha_t(\omega)+\beta |y^i|+\Big(v_t(\omega)+c\sum_{j=1}^{i-1}|z^j|\Big)|z^i|+\frac{\gamma}{2} |z^i|^2;
$$

\item [(iii)] Almost everywhere in $\Omega\times\T$, for any $(y,z)\in \R^n\times\R^{n\times d}$, we have
$$
|g^i(\omega,t,y,z)|\leq \bar\alpha_t(\omega)+\beta |y^i|+\bar\lambda |z|+\theta\sum_{j\neq i}|z^j|^2.\vspace{0.2cm}
$$
\end{enumerate}
\end{enumerate}

The following \cref{thm:2.9} and
\citet[Theorem 2.6]{JamneshanKupperLuo2017ECP} do not cover each other, and so are \cref{thm:2.9} and \citet[Proposition 2.1]{Frei2014SPA}.

\begin{thm}\label{thm:2.9}
Let the generator $g$ satisfy Assumptions \ref{A:D1} and \ref{A:D2} with $\beta=0$, and let
$$\xi=\bar\xi+\int_0^T H_s {\rm d}B_s$$
with $\bar\xi\in L^{\infty}(\R^n)$ and $H\in {\rm BMO}(\R^{n\times d})$. When $g^i$ satisfies (ii) in Assumption \ref{A:D1} and violates \ref{A:D1} (i) and (iii), and $-g^i$ violates \ref{A:D1} (i),
we further assume that $H^i\equiv 0$. If $\alpha\in \ecal^{\infty}(p\gamma)$ for some $p>1$, and $|H|^2\in \ecal^{\infty}(2\bar p (q\gamma)^2)$ for some $\bar p>1$ with $q=p/(p-1)$ such that $1/p+1/q=1$, then there exists a positive constant $\theta_0>0$ depending only on $(\|\bar\xi\|_{\infty}, \|\alpha\|_{\ecal^{\infty}(p\gamma)}, \|\bar\alpha\|_{\mcal^{\infty}}, \|\tilde\alpha\|_{\lcal^{\infty}}, \|v\|_{{\rm BMO}}, \||H|^2\|_{\ecal^{\infty}(2\bar p (q\gamma)^2)},$
$ n,\gamma, \bar\gamma,\lambda,\bar\lambda,c,\bar c,\delta, T, p,\bar p)$ such that when $\theta\in [0,\theta_0]$, BSDE \eqref{eq:2.1} admits a unique global solution $(Y,Z)$ on $\T$ such that
$$
\Big(Y-\int_0^\cdot H_s {\rm d}B_s, Z\Big)\in \s^\infty(\R^n)\times {\rm BMO}(\R^{n\times d}).\vspace{0.2cm}
$$
\end{thm}

\begin{rmk}\label{rmk:2.10} By the John-Nirenberg inequality, if $\|H\|_{\rm BMO}<1/(\sqrt{2\bar p}q\gamma)$, then
    $|H|^2\in \ecal^{\infty}(2\bar p (q\gamma)^2)$. Then, we have a unique solution for the multi-dimensional BSDE, as stated in Theorem 3.1 of \citet{KramkovPulido2016AAP}.
\end{rmk}

In view of (i) of \cref{rmk:2.1},  we have immediately from \cref{thm:2.9}  the following corollary.

\begin{cor}\label{cor:2.11}
Let the generator $g$ satisfy \ref{A:D1} and \ref{A:D2} with $\beta=0$ and $\theta=0$, and let
$$\xi=\bar\xi+\int_0^T H_s {\rm d}B_s$$
with $\bar\xi\in L^{\infty}(\R^n)$ and $H\in {\rm BMO}(\R^{n\times d})$. When $g^i$ satisfies (ii) in Assumption \ref{A:D1} and violates \ref{A:D1} (i) and (iii), and $-g^i$ violates \ref{A:D1} (i),
we further assume that $H^i\equiv 0$. If $\alpha,|H|^2\in \lcal^{\infty}$, then BSDE \eqref{eq:2.1} admits a unique global solution $(Y,Z)$ on the time interval $\T$ such that
$$
\Big(Y-\int_0^\cdot H_s {\rm d}B_s, Z\Big)\in \s^\infty(\R^n)\times {\rm BMO}(\R^{n\times d}).\vspace{0.2cm}
$$
\end{cor}

\begin{rmk}\label{rmk:2.12}
If the constant $c$ appearing in (ii) of Assumption \ref{A:D1} vanishes, then the condition that  $H_i\equiv 0$ appearing in \cref{thm:2.9} and \cref{cor:2.11} can be weakened to that both $|H^i|$ and $|H^i|v$ belong to $\lcal^{\infty}$.
\end{rmk}

 We end this subsection with the following theorem.

\begin{thm}\label{thm:2.16a}
Let the generator $g$ satisfy Assumptions \ref{A:B1} and \ref{A:D2} with $\theta=c=\bar c=0$ and $\phi(\cdot)\equiv 0$, and
$$\xi=\bar\xi+\int_0^T H_s {\rm d}B_s\vspace{0.1cm}$$
with $\bar\xi\in L^{\infty}(\R^n)$ and $H\in {\rm BMO}(\R^{n\times d})$.
When $g^i$ satisfies (ii) in Assumption \ref{A:B1} and violates \ref{A:B1} (i) and (iii), and $-g^i$ violates \ref{A:B1} (i),
we further assume that $|H^i|^2, |H^i|v\in \lcal^{\infty}$. If $\alpha\in \ecal^{\infty}(p\gamma)$ for some $p>1$, and $|H|^2\in \ecal^{\infty}(2\bar p (q\gamma)^2)$ for some $\bar p>1$ with $q=p/(p-1)$ such that $1/p+1/q=1$, then BSDE \eqref{eq:2.1} admits a unique global solution $(Y,Z)$ such that
$$
\Big(Y-\int_0^\cdot H_s {\rm d}B_s, \ Z\Big)\in \s^\infty(\R^n)\times {\rm BMO}(\R^{n\times d}).\vspace{0.2cm}
$$
\end{thm}

The proof is  identical to that of \cref{thm:2.9} with the help of \cref{thm:2.8a}, and thus is omitted here.


\section{Local unbounded solution: proof of \cref{thm:2.3}}
\label{sec:3-local solution}
\setcounter{equation}{0}

Let us begin with the following technical lemma for subsequent arguments.

\begin{lem}\label{lem:3.1}
Let $p>1$, $r>0$, $\bar\delta\in [0,1)$, $i\in \{1,\cdots,n\}$, $\bar t\in\T$ and $V,\bar V\in {\rm BMO}_{[\bar t,T]}(\R^{n\times d})$ such that, with $q:=\frac{p}{p-1}>1$ satisfying that $\frac{1}{p}+\frac{1}{q}=1$,
$$
2n(qr\vee 1)\theta \Big(\sum_{j=1}^{i-1}\|\bar V^j\|^2_{{\rm BMO}_{[\bar t,T]}}+\|V\|^2_{{\rm BMO}_{[\bar t,T]}}\Big)\leq 1.
$$
\begin{itemize}
\item [(i)] Define
\begin{equation}\label{eq:3.1}
\breve\alpha_s:=\hat\alpha_s+\sum_{j=1}^{i-1}\left(\lambda |\bar V^j_s|^{1+\bar\delta}+\theta |\bar V^j_s|^2\right)+\sum_{j=i+1}^{n}\left(\lambda |V^j_s|^{1+\bar\delta}+\theta |V^j_s|^2\right),\ \ \ s\in [\bar t,T].
\end{equation}
If $\hat\alpha\in \ecal^{\infty}_{[\bar t,T]}(pr)$, then $\breve\alpha \in \ecal^{\infty}_{[\bar t,T]}(r)$. And, for each $t\in [\bar t,T]$, we have
\begin{equation}\label{eq:3.2}
\begin{array}{lll}
\Dis \|\breve\alpha\|_{\ecal^{\infty}_{[t,T]}(r)} &\leq &\Dis \|\hat\alpha\|_{\ecal^{\infty}_{[t,T]}(pr)}+{\ln 2\over r}+C_{p,n,r,\lambda,\bar\delta} \sum_{j=1}^{i-1}\|\bar V^j\|_{{\rm BMO}_{[t,T]}}^{2\frac{1+\bar\delta}{1-\bar\delta}}(T-t)\\
&&\Dis+C_{p,n,r,\lambda,\bar\delta}\sum_{j=i+1}^{n} \|V^j\|_{{\rm BMO}_{[t,T]}}^{2\frac{1+\bar\delta}{1-\bar\delta}}(T-t),
\end{array}
\end{equation}
where
\begin{equation}\label{eq:3.3}
C_{p,n,r,\lambda,\bar\delta}:=\lambda\frac{1-\bar\delta}{2}\left({pnr\lambda (1+\bar\delta)\over p-1}\right)^{\frac{1+\bar\delta}{1-\bar\delta}}.\vspace{0.1cm}
\end{equation}
Moreover, it is clear that the previous conclusion holds also for $p=1$ and $q=0$ when $\lambda=0$ and $\theta=0$ if we let $C_{p,n,r,\lambda,\bar\delta}:=0$ when $\lambda=0$ and $p=1$.
\item [(ii)] Define
\begin{equation}\label{eq:3.4}
\breve\alpha_s:=\hat\alpha_s+\bar c\sum_{j=1}^{i-1}|\bar V^j_s|^2+\sum_{j=i+1}^{n} \left(\bar\lambda |V^j_s|^{1+\bar\delta}+\theta |V^j_s|^2\right) ,\ \ \ s\in [\bar t,T].\vspace{0.1cm}
\end{equation}
If $\hat\alpha\in \mcal^{\infty}_{[\bar t,T]}$, then $\breve\alpha \in \mcal^{\infty}_{[\bar t,T]}$. And, for each $t\in [\bar t,T]$,
\begin{equation}\label{eq:3.5}
\Dis \left\|\breve\alpha\right\|_{\mcal^{\infty}_{[t,T]}} \leq \Dis 1+\left\|\hat\alpha\right\|_{\mcal^{\infty}_{[t,T]}}+\bar c\sum_{j=1}^{i-1}\|\bar V^j\|^2_{{\rm BMO}_{[t,T]}}+n\bar\lambda \|V\|_{{\rm BMO}_{[t,T]}}^{1+\bar\delta}(T-t)^{1-\bar\delta\over 2}.
\end{equation}
\item [(iii)] Define
\begin{equation}\label{eq:3.6}
\breve\alpha_s:=\hat\alpha_s+\sum_{j=1}^{i-1}\left(\bar\lambda |\bar V^j_s|+\theta |\bar V^j_s|^2\right)+\sum_{j=i+1}^{n}\left(\bar\lambda |V^j_s|+\theta |V^j_s|^2\right),\ \ s\in [\bar t,T].
\end{equation}
If $\hat\alpha\in \mcal^{\infty}_{[\bar t,T]}$, then $\breve\alpha \in \mcal^{\infty}_{[\bar t,T]}$. And, for each $t\in [\bar t,T]$, we have
\begin{equation}\label{eq:3.7}
\Dis \left\|\breve\alpha\right\|_{\mcal^{\infty}_{[t,T]}} \leq \Dis 1+\left\|\hat\alpha\right\|_{\mcal^{\infty}_{[t,T]}}+\bar\lambda \Big(\sum_{j=1}^{i-1}\|\bar V^j\|_{{\rm BMO}_{[t,T]}}+\sum_{j=i+1}^{n}\|V^j\|_{{\rm BMO}_{[t,T]}}\Big)(T-t)^{1\over 2}.
\end{equation}
\end{itemize}
\end{lem}

\begin{proof}
(i)
Since $V,\bar V\in {{\rm BMO}}(\R^{n\times d})$, it follows from Young's inequality that for $p>1$, $\lambda\geq 0$, $0\leq \bar t\leq t\leq s\leq T$ and $j=1,\cdots,n$, we have
\begin{equation}\label{eq:3.8}
\begin{array}{lll}
\Dis \lambda |V_s^j|^{1+\bar\delta} &=& \Dis \lambda \Big(\Big({pnr\lambda (1+\bar\delta)\over p-1}\Big)^{\frac{1+\bar\delta}{1-\bar\delta}}\|V^j\|_{{\rm BMO}_{[t,T]}}^{2\frac{1+\bar\delta}{1-\bar\delta}} \Big)^{\frac{1-\bar\delta}{2}}\vspace{0.2cm}\\
&& \Dis  \times\,  \Big(\frac{(p-1)|V_s^j|^2}{pnr\lambda(1+\bar\delta)\|V^j\|_{{\rm BMO}_{[t,T]}}^2} \Big)^{\frac{1+\bar\delta}{2}}\vspace{0.1cm}\\
&\leq & \Dis \frac{p-1}{2pnr\|V^j\|_{{\rm BMO}_{[t,T]}}^2}|V_s^j|^2+ C_{p,n,r,\lambda,\bar\delta} \|V^j\|_{{\rm BMO}_{[t,T]}}^{2\frac{1+\bar\delta}{1-\bar\delta}}
\end{array}
\end{equation}
and
\begin{equation}\label{eq:3.9}
\lambda |\bar V_s^j|^{1+\bar\delta} \leq \frac{p-1}{2pnr\|\bar V^j\|_{{\rm BMO}_{[t,T]}}^2}|\bar V_s^j|^2+ C_{p,n,r,\lambda,\bar\delta} \|\bar V^j\|_{{\rm BMO}_{[t,T]}}^{2\frac{1+\bar\delta}{1-\bar\delta}},\vspace{0.2cm}
\end{equation}
where the constant $C_{p,n,r,\lambda,\bar\delta}$ is defined in \eqref{eq:3.3}. Note that $q:={p\over p-1}>1$ satisfying that ${1\over p}+{1\over q}=1$. By \eqref{eq:3.8} and \eqref{eq:3.9} we know that for each $t\in [\bar t,T]$, $\breve\alpha_s \leq \hat\alpha_s+\breve\alpha^{1,t}_s+\breve\alpha^{2,t}_s, s\in [t,T]$, where
$$
\breve\alpha^{1,t}_s:= \sum_{j=1}^{i-1}\frac{1}{2qnr\|\bar V^j\|_{{\rm BMO}_{[t,T]}}^2}|\bar V_s^j|^2+\sum_{j=i+1}^{n}\frac{1}{2qnr\|V^j\|_{{\rm BMO}_{[t,T]}}^2}|V_s^j|^2+\theta \Big(\sum_{j=1}^{i-1}|\bar V_s^j|^2+|V_s|^2\Big)
$$
and
$$
\breve\alpha^{2,t}_s:= C_{p,n,r,\lambda,\bar\delta} \Big(\sum_{j=1}^{i-1}\|\bar V^j\|_{{\rm BMO}_{[t,T]}}^{2\frac{1+\bar\delta}{1-\bar\delta}}+ \sum_{j=i+1}^{n}\|V^j\|_{{\rm BMO}_{[t,T]}}^{2\frac{1+\bar\delta}{1-\bar\delta}}\Big).
$$
It follows from H\"older's inequality and the John-Nirenberg inequality that for each $t\in [\bar t,T]$ and $\tau\in \mathcal{T}_{[t,T]}$,
$$
\begin{array}{lll}
\Dis\E_\tau\left[\exp\left(qr\int_\tau^T \breve\alpha^{1,t}_s {\rm d}s \right)\right] &\leq & \Dis \prod_{j=1}^{i-1}
\Bigg(\E_\tau\Big[\exp\Big(\frac{1}{2\|\bar V^j\|_{{\rm BMO}_{[t,T]}}^2}\int_\tau^T |\bar V^j_s|^2{\rm d}s\Big)\Big]\Bigg)^{1\over n} \\
&&\Dis \times\,  \Bigg(\E_\tau\Big[\exp\Big(nqr\theta\int_\tau^T \Big(\sum_{j=1}^{i-1}|\bar V_s^j|^2+|V_s|^2\Big){\rm d}s\Big)\Big]\Bigg)^{1\over n}\\
&& \Dis \times\,  \prod_{j=i+1}^{n}
\Bigg(\E_\tau\Big[\exp\Big(\frac{1}{2\|V^j\|_{{\rm BMO}_{[t,T]}}^2}\int_\tau^T |V^j_s|^2{\rm d}s\Big)\Big]\Bigg)^{1\over n}\vspace{0.1cm}\\
&\leq & 2,
\end{array}
$$
which means that $\breve\alpha^{1,t}\in \ecal^{\infty}_{[t,T]}(qr)$ and $\|\breve\alpha^{1,t}\|_{\ecal^{\infty}_{[t,T]}(qr)}\leq {\ln 2\over qr}$. On the other hand, it is clear that
$$\left\|\breve\alpha^{2,t}\right\|_{\lcal^{\infty}_{[t,T]}}= C_{p,n,r,\lambda,\bar\delta} \Big(\sum_{j=1}^{i-1}\|\bar V^j\|_{{\rm BMO}_{[t,T]}}^{2\frac{1+\bar\delta}{1-\bar\delta}}+ \sum_{j=i+1}^{n} \|V^j\|_{{\rm BMO}_{[t,T]}}^{2\frac{1+\bar\delta}{1-\bar\delta}}\Big)(T-t),\ \ t\in [\bar t,T].
$$
It then follows from (ii) and (iii) of \cref{rmk:2.1} that for each $t\in [\bar t,T]$, $\breve\alpha\in \ecal^{\infty}_{[t,T]}(r)$ and
$$
\begin{array}{lll}
\Dis \|\breve\alpha\|_{\ecal^{\infty}_{[t,T]}(r)}&\leq &\Dis \|\hat\alpha\|_{\ecal^{\infty}_{[t,T]}(pr)}+{\ln 2\over qr}\vspace{0.2cm}\\
&&\Dis +C_{p,n,r,\lambda,\bar\delta} \Big(\sum_{j=1}^{i-1}\|\bar V^j\|_{{\rm BMO}_{[t,T]}}^{2\frac{1+\bar\delta}{1-\bar\delta}}+ \sum_{j=i+1}^{n} \|V^j\|_{{\rm BMO}_{[t,T]}}^{2\frac{1+\bar\delta}{1-\bar\delta}}\Big)(T-t),
\end{array}
$$
from which the desired conclusion \eqref{eq:3.2} follows immediately.\vspace{0.2cm}

(ii) It follows from H\"older's inequality that for each $t\in [\bar t,T]$, $j=1,\cdots,n$ and $\tau\in \mathcal{T}_{[t,T]}$,
$$
\E_\tau\left[\int_\tau^T|V^j_s|^{1+\bar\delta}{\rm d}s\right]\leq \left(\E_\tau\left[\int_\tau^T|V^j_s|^2{\rm d}s\right]\right)^{1+\bar\delta\over 2} (T-t)^{1-\bar\delta\over 2}\leq \|V^j\|_{{\rm BMO}_{[t,T]}}^{1+\bar\delta}(T-t)^{1-\bar\delta\over 2}.
$$
Then, the desired conclusion \eqref{eq:3.5} follows from \eqref{eq:3.4} and the last inequality.\vspace{0.2cm}

(iii) It follows from H\"older's inequality that for each $t\in [\bar t,T]$, $j=1,\cdots,n$ and $\tau\in \mathcal{T}_{[t,T]}$,
$$
\E_\tau\left[\int_\tau^T|V^j_s|{\rm d}s\right]\leq \|V^j\|_{{\rm BMO}_{[t,T]}}(T-t)^{1\over 2}\ \ {\rm and}\ \ \E_\tau\left[\int_\tau^T|\bar V^j_s|{\rm d}s\right]\leq \|\bar V^j\|_{{\rm BMO}_{[t,T]}}(T-t)^{1\over 2}.
$$
Then, the desired conclusion \eqref{eq:3.7} follows from \eqref{eq:3.6} and the last two inequalities.\vspace{0.1cm}
\end{proof}

Now, we can give the proof of \cref{thm:2.3}. Assume first that $\xi\in L^{\infty}(\R^n)$, $\alpha\in \ecal^{\infty}(p\gamma)$ for some real $p>1$ and the generator $g$ satisfies Assumptions \ref{A:B1} and \ref{A:B2}. Define\vspace{0.1cm}
$$
\begin{array}{lll}
C_1 &:=& \Dis {2+\bar\gamma\over \bar\gamma}\bigg\{8+{\ln 2\over \gamma}+3\|\xi\|_{\infty}+\left\|\alpha\right\|_{\ecal^{\infty}
(p\gamma)}+2\left\|\bar\alpha\right\|_{\mcal^{\infty}}+nC_{p,n,\gamma,\lambda,\delta} T+2n\bar c\bigg\}\vspace{0.2cm}\\
&& \Dis +{4(\gamma +1) \over \gamma^2}\exp\left\{4\gamma \left(\|\xi\|_{\infty}+\|\tilde\alpha\|_{\lcal^{\infty}}+1\right)\right\}
\left(5+\|\tilde\alpha\|_{\lcal^{\infty}}+3\|v\|^2_{{\rm BMO}}+3n^2c^2
\right)\vspace{0.2cm}\\
&& \Dis +1+c_0 \exp\left(2\bar\lambda^2 T\right)\left(54+\|\xi\|_\infty^2+6\|\bar\alpha\|_{\mcal^\infty}^2+6n^2\bar\lambda^2 T\right)
\end{array}
$$
and
$$
\begin{array}{lll}
C_2 &:=& \Dis {2+\bar\gamma\over \bar\gamma}\left(2C_{p,n,\gamma,\lambda,\delta} T+4\bar c\right)+{12nc^2 (\gamma +1) \over \gamma^2}\exp\left\{4\gamma \left(\|\xi\|_{\infty}+\|\tilde\alpha\|_{\lcal^{\infty}}+1\right)\right\}\vspace{0.2cm}\\
&& \Dis +12nc_0\bar\lambda^2 T \exp\left(2\bar\lambda^2 T\right),
\end{array}
$$
where the uniform constant $c_0>0$ is defined in (i) of \cref{pro:A.3} in Appendix, and
\begin{equation}\label{eq:3.10}
C_{p,n,\gamma,\lambda,\delta}:=\lambda\frac{1-\delta}{2}
\Big({pn\gamma\lambda (1+\delta)\over p-1}\Big)^{\frac{1+\delta}{1-\delta}}+1.
\end{equation}
Furthermore, let $C^0_1:=0$ and for $i=1,\cdots,n$, recursively define
\begin{equation}\label{eq:3.11}
C^i_1:=C^{i-1}_1+C_1+C_2\left[C^{i-1}_1\right]^{\frac{1+\delta}{1-\delta}}.
\end{equation}
Finally, define $K:=C^n_1$ and let the constant $\eps$ satisfy
\begin{equation}\label{eq:3.12}
0<\eps\leq \min\left\{T,\ {1\over [1+\phi(K)]^2},\ {1\over nC_{p,n,\gamma,\lambda,\delta} K^{\frac{1+\delta}{1-\delta}}},\  \left({1\over 1+2n\bar\lambda K^{1+\delta\over 2}}\right)^{2\over 1-\delta},\ {1\over 1+n^2\bar\lambda^2K}\right\}.
\end{equation}
It is clear that all these constants defined above (except $\eps$)
depend only on $\|\xi\|_{\infty}$, $\|\alpha\|_{\ecal^{\infty}(p\gamma)}$, $\|\bar\alpha\|_{\mcal^{\infty}}$, $\|\tilde\alpha\|_{\lcal^{\infty}}$, $\|v\|_{{\rm BMO}}$, $n,\gamma,\bar\gamma$, $\lambda,\bar\lambda, c,\bar c, \delta,T$ and $p$, and that $\eps$ also depends on $\phi(\cdot)$.\vspace{0.2cm}

In the sequel, for each $H\in \R^{n\times d}$, $i\in \{1,\cdots,n\}$ and $z^{(1)}\cdots,z^{(i)}\in \R^{1\times d}$, define by
$$[z^{(1)},\cdots,z^{(i)},H^{i+1},\cdots,H^n]\vspace{-0.1cm}$$
the matrix in $\R^{n\times d}$ whose $j$th row is $z^{(j)}$ for  $j=1,\cdots,i$ and $H^j$ for $j=i+1,\cdots,n$ with $i\neq n$. Let $q:=\frac{p}{p-1}$.
Given a pair of processes $(U,V)\in \s^\infty_{[T-\eps,T]}(\R^n)\times {\rm BMO}_{[T-\eps,T]}(\R^{n\times d})$ and a real $\theta$ satisfying
\begin{equation}\label{eq:3.13}
\|U\|_{\s^\infty_{[T-\eps,T]}}+\|V\|_{{\rm BMO}_{[T-\eps,T]}}^2\leq K\ \ \ {\rm and}\ \ \ 0\leq \theta \leq {1\over 4n(q\gamma\vee 1)K}.
\end{equation}
We will first prove that for each $i=1,\cdots,n$, the following one-dimensional BSDE
\begin{equation}\label{eq:3.14}
  Y_t^i=\xi^i+\int_t^T f^i(s,Z_s^i){\rm d}s-\int_t^T Z_s^i {\rm d}B_s, \ \ t\in [T-\eps,T]
\end{equation}
admits, successively, a unique solution $(Y^i,Z^i)$ in the space $\s^\infty_{[T-\eps,T]}(\R)\times {\rm BMO}_{[T-\eps,T]}(\R^{1\times d})$, where for each $(\omega,t,z)\in \Omega\times [T-\eps,T] \times\R^{1\times d}$,
$$
\left\{
\begin{array}{l}
\Dis f^1(\omega,t,z):=g^1\left(\omega,t,U_t(\omega),
[z,V^{2}_t(\omega),\cdots, V^{n}_t(\omega)]\right),\vspace{0.2cm}\\
\Dis f^i(\omega,t,z):=g^i\left(\omega,t,U_t(\omega),
[Z^1_t(\omega),\cdots,Z^{i-1}_t(\omega),z,V^{i+1}_t(\omega),\cdots, V^{n}_t(\omega)]\right),\ \ i=2,\cdots,n-1,\vspace{0.2cm}\\
\Dis f^n(\omega,t,z):=g^n\left(\omega,t,U_t(\omega),
[Z^1_t(\omega),\cdots,Z^{n-1}_t(\omega),z]\right).\vspace{0.2cm}
\end{array}
\right.
$$
Moreover, it holds that for each $i=1,\cdots,n$,
\begin{equation}\label{eq:3.15}
\|Y^i\|_{\s^{\infty}_{[T-\eps,T]}}+\|Z^i\|_{{\rm BMO}_{[T-\eps,T]}}^2 \leq C_1+C_2 \sum_{j=1}^{i-1} \|Z^j\|_{{\rm BMO}_{[T-\eps,T]}}^{2\frac{1+\delta}{1-\delta}},
\end{equation}
\begin{equation}\label{eq:3.16}
\sum_{j=1}^{i}\left(\|Y^j\|_{\s^{\infty}_{[T-\eps,T]}}+\|Z^j\|_{{\rm BMO}_{[T-\eps,T]}}^2\right) \leq C^i_1\leq K
\end{equation}
and
\begin{equation}\label{eq:3.17}
2n(q\gamma\vee 1)\theta \Big(\sum_{j=1}^{i}\|Z^j\|^2_{{\rm BMO}_{[T-\eps,T]}}+\|V\|^2_{{\rm BMO}_{[T-\eps,T]}}\Big)\leq 1.\vspace{0.2cm}
\end{equation}
In particular, letting $i=n$ in \eqref{eq:3.16} yields that
\begin{equation}\label{eq:3.18}
\|Y\|_{\s^\infty_{[T-\eps,T]}}+\|Z\|_{{\rm BMO}_{[T-\eps,T]}}^2\leq K.\vspace{0.2cm}
\end{equation}

For this, let us use an induction argument to prove the following proposition.\vspace{-0.1cm}

\begin{pro}\label{pro:3.2c}
We have the following two assertions.
\begin{itemize}
\item [(a)] For $i=1$, BSDE \eqref{eq:3.14} admits a unique solution $(Y^1,Z^1)$ in the space $\s^\infty_{[T-\eps,T]}(\R)\times {\rm BMO}_{[T-\eps,T]}(\R^{1\times d})$, and inequalities \eqref{eq:3.15}-\eqref{eq:3.17} hold.

\item [(b)] For $i=2,\cdots,n$, assume that for each $l=1,\cdots,i-1$, the following BSDE
\begin{equation}\label{eq:3.19}
  Y_t^l=\xi^l+\int_t^T f^l(s,Z_s^l){\rm d}s-\int_t^T Z_s^l {\rm d}B_s, \ \ t\in [T-\eps,T]
\end{equation}
admits, successively, a unique solution $(Y^l,Z^l)$ in the space $\s^\infty_{[T-\eps,T]}(\R)\times {\rm BMO}_{[T-\eps,T]}(\R^{1\times d})$, and the following inequalities hold:
\begin{equation}\label{eq:3.21}
\sum_{j=1}^{l}\left(\|Y^j\|_{\s^{\infty}_{[T-\eps,T]}}+\|Z^j\|_{{\rm BMO}_{[T-\eps,T]}}^2\right) \leq C^l_1\leq K
\end{equation}
and
\begin{equation}\label{eq:3.22}
2n(q\gamma\vee 1)\theta \Big(\sum_{j=1}^{l}\|Z^j\|^2_{{\rm BMO}_{[T-\eps,T]}}+\|V\|^2_{{\rm BMO}_{[T-\eps,T]}}\Big)\leq 1.\vspace{0.2cm}
\end{equation}
Then, BSDE \eqref{eq:3.14} also admits a unique solution $(Y^i,Z^i)\in \s^\infty_{[T-\eps,T]}(\R)\times {\rm BMO}_{[T-\eps,T]}(\R^{1\times d})$, and the inequalities \eqref{eq:3.15}-\eqref{eq:3.17} also hold.\vspace{0.2cm}
\end{itemize}
\end{pro}

\begin{proof}
We first prove the assertion (b). Assume that for some $i\in \{2,\cdots,n\}$ and each $l=1,\cdots,i-1$, BSDE \eqref{eq:3.19} admits a unique solution $(Y^l,Z^l)\in \s^\infty_{[T-\eps,T]}(\R)\times {\rm BMO}_{[T-\eps,T]}(\R^{1\times d})$, and the inequalities \eqref{eq:3.21}-\eqref{eq:3.22} hold. Since the generator $g$ satisfies Assumption \ref{A:B2}, it follows from the definition of $f^i$ that $\as$ on $\Omega\times [T-\eps,T]$, for each $(z, \bar z)\in \R^{1\times d}\times\R^{1\times d}$,
$$
|f^i(t,z)-f^i(t,\bar z)|\leq \phi(\|U\|_{\s^\infty_{[T-\eps,T]}})\Big(v_t+2\sum_{j=1}^{i-1}|Z^j_t|+2|V_t| +|z|+|\bar z|\Big)|z-\bar z|,
$$
which means that on $\Omega\times [T-\eps,T]$, $f^i$ satisfies Assumption \ref{A:A5} in Appendix with $\bar\beta=0$,
$$
k=\phi(\|U\|_{\s^\infty_{[T-\eps,T]}})\ \ {\rm and}\ \ \bar v=v+2\sum_{j=1}^{i-1}|Z^j|+2|V|\in {\rm BMO}_{[T-\eps,T]}(\R).
$$
And, since $g$ satisfies Assumption \ref{A:B1}, in view of (iv) of \cref{rmk:2.2}, we need only to consider the following three cases:\vspace{0.2cm}

(1) $g^i$ satisfies (i) of Assumption \ref{A:B1}. For this case, it follows from the definition of $f^i$ that $\as$ on $\Omega\times [T-\eps,T]$, for each $z\in \R^{1\times d}$, we have
\begin{equation}\label{eq:3.23}
\Dis \frac{\bar \gamma}{2} |z|^2-\dot\alpha_t \leq f^i(t,z) \leq  \check\alpha_t+\frac{\gamma}{2} |z|^2,
\end{equation}
where
$$
\check\alpha_t:=\alpha_t+\phi(|U_t|) +\sum_{j=1}^{i-1}\left(\lambda |Z^j_t|^{1+\delta}+\theta |Z^j_t|^2\right)+\sum_{j=i+1}^{n}\left(\lambda |V^j_t|^{1+\delta}+\theta |V^j_t|^2\right)
$$
and
$$
\dot\alpha_t:=\bar\alpha_t+\phi(|U_t|)+\bar c\sum_{j=1}^{i-1}|Z^j_t|^2+\sum_{j=i+1}^{n} \left(\bar\lambda |V^j_t|^{1+\delta}+\theta |V^j_t|^2\right).\vspace{0.2cm}
$$
Since $\alpha\in \ecal^{\infty}_{[T-\eps,T]}(p\gamma)$, $\bar\alpha\in \mcal^{\infty}_{[T-\eps,T]}$ and $\phi(|U|)\in \lcal^\infty_{[T-\eps,T]}$, in view of (ii) of \cref{rmk:2.1}, \eqref{eq:3.13} and \eqref{eq:3.22} with $l=i-1$, using (i) and (ii) of \cref{lem:3.1} with $r=\gamma$, $\bar\delta=\delta$, $\bar t=T-\eps$, $\bar V=Z$, $\breve\alpha=\check\alpha$ and $\hat\alpha=\alpha+\phi(|U|)$ (resp. $\breve\alpha=\dot\alpha$ and $\hat\alpha=\bar\alpha+\phi(|U|)$) we can deduce that
\begin{equation}\label{eq:3.24}
\begin{array}{lll}
\Dis \|\check\alpha\|_{\ecal^{\infty}_{[t,T]}(\gamma)} &\leq &\Dis \|\alpha\|_{\ecal^{\infty}_{[t,T]}(p\gamma)}+\phi(K)(T-t)+{\ln 2\over \gamma}+nC_{p,n,\gamma,\lambda,\delta} K^{\frac{1+\delta}{1-\delta}}(T-t)\vspace{0.1cm}\\
&&\Dis+C_{p,n,\gamma,\lambda,\delta} T\sum_{j=1}^{i-1}\|Z^j\|_{{\rm BMO}_{[t,T]}}^{2\frac{1+\delta}{1-\delta}}<+\infty,\ \ t\in [T-\eps,T]
\end{array}
\end{equation}
and
\begin{equation}\label{eq:3.25}
\begin{array}{lll}
\Dis \left\|\dot\alpha\right\|_{\mcal^{\infty}_{[t,T]}} &\leq &\Dis 1+\left\|\bar\alpha\right\|_{\mcal^{\infty}_{[t,T]}}+\phi(K)(T-t)+n\bar\lambda K^{1+\delta\over 2}(T-t)^{1-\delta\over 2}\vspace{0.1cm}\\
&&\Dis +\bar c\sum_{j=1}^{i-1}\|Z^j\|^2_{{\rm BMO}_{[t,T]}}<+\infty,\ \ t\in [T-\eps,T],\vspace{0.1cm}
\end{array}
\end{equation}
where the constant $C_{p,n,\gamma,\lambda,\delta}$ is defined in \eqref{eq:3.10}. Combining \eqref{eq:3.23}, \eqref{eq:3.24} and \eqref{eq:3.25} yields that on $\Omega\times [T-\eps,T]$, the generator $f^i$ satisfies Assumption \ref{A:A1} in Appendix with $\bar\beta=0$ and $\varphi(\cdot)\equiv 0$. Since $f^i$ also satisfies Assumption \ref{A:A5}, it follows from (iii) of \cref{pro:A.1} that BSDE \eqref{eq:3.14} admits a unique solution $(Y^i,Z^i)\in \s^{\infty}_{[T-\eps,T]}(\R)\times {\rm BMO}_{[T-\eps,T]}(\R^{1\times d})$. Moreover, by (i) of \cref{pro:A.1} we have
\begin{equation}\label{eq:3.26}
\begin{array}{l}
\Dis \|Y^i\|_{\s^{\infty}_{[t,T]}}+\|Z^i\|_{{\rm BMO}_{[t,T]}}^2 \leq  \Dis {2+\bar\gamma\over \bar\gamma}\bigg\{2+{\ln 2\over \gamma}+3\|\xi^i\|_{\infty}+\left\|\alpha\right\|_{\ecal^{\infty}_{[t,T]}
(p\gamma)}+2\left\|\bar\alpha\right\|_{\mcal^{\infty}_{[t,T]}} \vspace{0.3cm}\\
\Dis \hspace{0.5cm} +3\phi(K)(T-t)+\left(C_{p,n,\gamma,\lambda,\delta} T+2\bar c\right) \sum_{j=1}^{i-1} \Big(\|Z^j\|_{{\rm BMO}_{[t,T]}}^{2\frac{1+\delta}{1-\delta}}+\|Z^j\|_{{\rm BMO}_{[t,T]}}^2\Big)\vspace{0.2cm}\\
\Dis \hspace{0.5cm} \left. +nC_{p,n,\gamma,\lambda,\delta} K^{\frac{1+\delta}{1-\delta}}(T-t) +2n\bar\lambda K^{1+\delta\over 2}(T-t)^{1-\delta\over 2}\right\},\ \ t\in [T-\eps,T].
\end{array}
\end{equation}
Then, in view of the inequality
\begin{equation}\label{eq:3-26h}
a^2\leq 1+a^{2\frac{1+\delta}{1-\delta}},\ \ \RE\ a\geq 0,
\end{equation}
it follows from \eqref{eq:3.26} and \eqref{eq:3.12} together with the definitions of $C_1$ and $C_2$ that the desired inequality \eqref{eq:3.15} holds. Moreover, in view of the following inequality
$$
\sum_{j=1}^{m}a_j^{\frac{1+\delta}{1-\delta}}\leq \Big( \sum_{j=1}^{m}a_j\Big)^{\frac{1+\delta}{1-\delta}},\ \ \RE\ m\geq 1\ \ {\rm and}\ \ \RE\ a_j\geq 0,\ j=1,\cdots, m,\vspace{-0.1cm}
$$
combining \eqref{eq:3.15} and \eqref{eq:3.21} with $l=i-1$ as well as \eqref{eq:3.11} we can derive that inequality \eqref{eq:3.16} also holds. Finally, the desired inequality \eqref{eq:3.17} follows from \eqref{eq:3.13} and \eqref{eq:3.16}. Consequently, the assertion (b) is proved in this case.\vspace{0.2cm}

(2) $g^i$ satisfies (ii) of Assumption \ref{A:B1}. For this case, it follows from the definition of $f^i$ that $\as$ on $\Omega\times [T-\eps,T]$, for each $z\in \R^{1\times d}$, we have
$$
|f^i(t,z)| \leq \ddot\alpha_t+\bar u_t |z|+\frac{\gamma}{2} |z|^2,\vspace{-0.2cm}
$$
where
$$
\ddot\alpha:=\tilde\alpha+\phi(|U|)\in \lcal^\infty_{[T-\eps,T]}\ \ \ {\rm and}\ \ \ \bar u:=c\sum_{j=1}^{i-1}|Z^j|+\phi(|U|)+v\in {\rm BMO}_{[T-\eps,T]}(\R).
$$
This means that on $\Omega\times [T-\eps,T]$, the generator $f^i$ satisfies Assumption \ref{A:A3} in Appendix with $\bar\beta=0$ and $\varphi(\cdot)\equiv 0$. Since $f^i$ also satisfies Assumption \ref{A:A5}, it follows from (iii) of \cref{pro:A.2} in Appendix that BSDE \eqref{eq:3.14} has a unique solution $(Y^i,Z^i)\in \s^{\infty}_{[T-\eps,T]}(\R)\times {\rm BMO}_{[T-\eps,T]}(\R^{1\times d})$. Moreover, in view of (i) of \cref{pro:A.2} with $r=T$, \eqref{eq:3.13} and \eqref{eq:3.12}, we have for each $t\in [T-\eps,T]$,
\begin{equation}\label{eq:3.27}
\begin{array}{l}
\Dis\|Y^i\|_{\s^\infty_{[t,T]}}+\|Z^i\|^2_{{\rm BMO}_{[t,T]}}
\leq \Dis {4(\gamma +1) \over \gamma^2}\exp\left\{4\gamma \left(\|\xi^i\|_{\infty}+\|\tilde\alpha\|_{\lcal^{\infty}_{[t,T]}}
+1\right)\right\}\vspace{0.1cm}\\
\hspace{4cm}\Dis \times \,  \Big(5+\|\tilde\alpha\|_{\lcal^{\infty}_{[t,T]}}
+3\|v\|^2_{{\rm BMO}_{[t,T]}}+3nc^2 \sum_{j=1}^{i-1} \|Z^j\|^2_{{\rm BMO}_{[t,T]}} \Big).
\end{array}
\end{equation}
Then, in view of \eqref{eq:3-26h}, it follows from \eqref{eq:3.27} together with the definitions of constants $C_1$ and $C_2$ that the desired inequality \eqref{eq:3.15} holds. And, in the same way as in (1) we can deduce that \eqref{eq:3.16} and \eqref{eq:3.17} also hold. Thus, the assertion (b) is proved in this case.\vspace{0.2cm}

(3) $g^i$ satisfies (iii) of Assumption \ref{A:B1}. For this case, it follows from the definition of $f^i$ that $\as$ on $\Omega\times [T-\eps,T]$, for each $z\in \R^{1\times d}$, we have
\begin{equation}\label{eq:3.28}
|f^i(t,z)| \leq  \Dis \breve\alpha_t+\bar\lambda |z|,
\end{equation}
where
$$
\breve\alpha_t:=\bar\alpha_t+\phi(|U_t|)+\sum_{j=1}^{i-1}\left(\bar\lambda |Z^j_t|+\theta |Z^j_t|^2\right)+\sum_{j=i+1}^{n}\left(\bar\lambda |V^j_t|+\theta |V^j_t|^2\right).\vspace{0.1cm}
$$
Since $\bar\alpha\in \mcal^{\infty}_{[T-\eps,T]}$ and $\phi(|U|)\in \lcal^\infty_{[T-\eps,T]}$, in view of \eqref{eq:3.12}, \eqref{eq:3.13} and \eqref{eq:3.22} with $l=i-1$, using (iii) of \cref{lem:3.1} with $\bar t=T-\eps$, $\bar V=Z$ and $\hat\alpha=\bar\alpha+\phi(|U|)$ we can deduce that
\begin{equation}\label{eq:3.29}
\hspace*{-0.2cm}
\begin{array}{lll}
\Dis \left\|\breve\alpha\right\|_{\mcal^{\infty}_{[t,T]}} &\leq &\Dis 1+\left\|\bar\alpha\right\|_{\mcal^{\infty}_{[t,T]}}+\phi(K)(T-t)+n\bar\lambda K^{1\over 2}(T-t)^{1\over 2}+\bar\lambda\sqrt{T}\sum_{j=1}^{i-1}\|Z^j\|_{{\rm BMO}_{[t,T]}}\\
&\leq &\Dis 3+\left\|\bar\alpha\right\|_{\mcal^{\infty}_{[t,T]}}+\bar\lambda\sqrt{T}\sum_{j=1}^{i-1}\|Z^j\|_{{\rm BMO}_{[t,T]}}<+\infty,\ \ t\in [T-\eps,T].
\end{array}
\end{equation}
Combining \eqref{eq:3.28} and \eqref{eq:3.29} yields that on $\Omega\times [T-\eps,T]$, the generator $f^i$ satisfies Assumption \ref{A:A4} in Appendix with $\dot\alpha=\breve\alpha$, $\bar\beta=0$ and $\varphi(\cdot)\equiv 0$. Since $f^i$ also satisfies Assumption \ref{A:A5}, it follows from (iii) of \cref{pro:A.3} in Appendix that BSDE \eqref{eq:3.14} admits a unique solution $(Y^i,Z^i)$ in the space $\s^{\infty}_{[T-\eps,T]}(\R)\times {\rm BMO}_{[T-\eps,T]}(\R^{1\times d})$. Moreover, in view of \eqref{eq:3-26h}, by (i) of \cref{pro:A.3} we have for each $t\in [T-\eps,T]$,
\begin{equation}\label{eq:3.30}
\begin{array}{lll}
\Dis \|Y^i\|_{\s^\infty_{[t,T]}}+\left\|Z^i\right\|_{{\rm BMO}_{[t,T]}}^2 &\leq & \Dis 1+c_0 \exp\left(2\bar\lambda^2 T\right)\left(54+\|\xi^i\|_\infty^2
+6\|\bar\alpha\|_{\mcal^\infty_{[t,T]}}^2\right)\\
&& \Dis +12nc_0\bar\lambda^2 T \exp\left(2\bar\lambda^2 T\right)\sum_{j=1}^{i-1}\|Z^j\|^2_{{\rm BMO}_{[t,T]}},
\end{array}
\end{equation}
where the uniform constant $c_0>0$ is defined in (i) of \cref{pro:A.3}. Then, it follows from \eqref{eq:3.30} together with the definitions of $C_1$ and $C_2$ that the desired inequality \eqref{eq:3.15} holds. And, in the same way as in (1) we can deduce that \eqref{eq:3.16} and \eqref{eq:3.17} hold. Thus, the assertion (b) is also proved in this case.\vspace{0.2cm}

Next, we prove the assertion (a). Indeed, in view of \eqref{eq:3.13}, by applying the above argument to $i=1$ we can deduce that for $i=1$, BSDE \eqref{eq:3.14} admits a unique solution $(Y^1,Z^1)$ in the space $\s^\infty_{[T-\eps,T]}(\R)\times {\rm BMO}_{[T-\eps,T]}(\R^{1\times d})$ and \eqref{eq:3.15} holds. In addition, in the case of $i=1$, \eqref{eq:3.16} is just \eqref{eq:3.15}, and \eqref{eq:3.17} is trivially satisfied by \eqref{eq:3.13}. Thus, the assertion (a) is also true, and the proof of \cref{pro:3.2c} is then complete.\vspace{0.2cm}
\end{proof}

Now, for each real $\eps>0$ satisfying \eqref{eq:3.12}, define the following complete metric space
$$
\Dis\mathcal{B}_\eps:=\left\{(U,V)\in \s^\infty_{[T-\eps,T]}(\R^n)\times {\rm BMO}_{[T-\eps,T]}(\R^{n\times d}): \|U\|_{\s^{\infty}_{[T-\eps,T]}}+\|V\|_{{\rm BMO}_{[T-\eps,T]}}^2\leq K \right\},
$$
which is a closed convex subset in the Banach space $\s^\infty(\R^n)\times {\rm BMO}(\R^{n\times d})$ with the norm
$$
\|(U,V)\|_{\mathcal{B}_\eps}:=\sqrt{\|U\|_{\s^{\infty}_{[T-\eps,T]}}^2+\|V\|_{{\rm BMO}_{[T-\eps,T]}}^2},\ \ \ \RE\ (U,V)\in \mathcal{B}_\eps.
$$
Based on \cref{pro:3.2c}, we know that those assertions from \eqref{eq:3.14} to \eqref{eq:3.18} are all true. Thus, in the case of $0\leq \theta\leq 1/[4n(q\gamma\vee 1)K]$ we can define the following map from $\mathcal{B}_\eps$ to itself:
$$
\Gamma: (U,V)\in \mathcal{B}_\eps \mapsto \Gamma(U,V):=(Y,Z)\in \mathcal{B}_\eps,
$$
where for each $i=1,\cdots,n$, $(Y^i,Z^i)$ (the $i$th component of $Y$ and the $i$th row of $Z$) is the unique solution of BSDE \eqref{eq:3.14} in the space $\s^\infty_{[T-\eps,T]}(\R)\times {\rm BMO}_{[T-\eps,T]}(\R^{1\times d})$.\vspace{0.2cm}

It remains to show that there exists a real $\eps_0>0$ satisfying \eqref{eq:3.12} and a real $\theta_0\in (0,1/[4n(q\gamma\vee 1)K]]$ (both depending only on $\|\xi\|_{\infty}$, $\|\alpha\|_{\ecal^{\infty}(p\gamma)}$, $\|\bar\alpha\|_{\mcal^{\infty}}$, $\|\tilde\alpha\|_{\lcal^{\infty}}$, $\|v\|_{{\rm BMO}}$, $n,\gamma,\bar\gamma,\lambda,\bar\lambda,c,\bar c,\delta,T,p$ and $\phi(\cdot)$)
such that in the case of $\theta\in [0,\theta_0]$, $\Gamma$ is a contraction in $\mathcal{B}_{\eps_0}$.\vspace{0.2cm}

In the sequel, let $0\leq \theta\leq 1/[4n(q\gamma\vee 1)K]$. For any fixed $\eps$ satisfying \eqref{eq:3.12} as well as $(U,V)\in \mathcal{B}_{\eps}$ and $(\widetilde U,\widetilde V)\in \mathcal{B}_{\eps}$, we set\vspace{-0.1cm}
$$
(Y,Z):=\Gamma(U,V),\ \ \ \ (\widetilde Y,\widetilde Z):=\Gamma(\widetilde U,\widetilde V).
$$
That is, for $i=1,\cdots,n$ and $t\in [T-\eps,T]$, we have
$$
\Dis Y_t^i=\xi^i+\int_t^T g^i(s,U_s,V_s(Z_s,Z_s^i;i)){\rm d}s-\int_t^T Z_s^i {\rm d}B_s
$$
and
$$
\Dis \widetilde Y_t^i=\xi^i+\int_t^T g^i(s,\widetilde U_s,\widetilde V_s(\widetilde Z_s,\widetilde Z_s^i;i)){\rm d}s-\int_t^T \widetilde Z_s^i {\rm d}B_s.\vspace{0.1cm}
$$
Here and hereafter, for each $i=1,\cdots,n$, $w\in \R^{1\times d}$ and $H,\bar H\in \R^{n\times d}$, we denote by $H(\bar H,w;i)$ the matrix in $\R^{n\times d}$ whose $i$th row is $w$, and whose $j$th row is $\bar H^j$ for $j=1,\cdots,i-1$ with $i\neq 1$ and $H^j$ for $j=i+1,\cdots,n$ with $i\neq n$. Then, for each $i=1,\cdots,n$ and $\tau\in \mathcal{T}_{[T-\eps,T]}$, we have
\begin{equation}\label{eq:3.31}
\begin{array}{l}
\Dis Y_\tau^i-\widetilde Y_\tau^i+\int_\tau^T \left(Z_s^i-\widetilde Z_s^i \right) {\rm d}B_s\vspace{0.3cm}\\
\hspace{1.4cm}\Dis -\int_\tau^T \mathop{\underbrace{\left(g^i(s,U_s,V_s(Z_s,Z_s^i;i))-g^i(s,U_s,V_s(Z_s,\widetilde Z_s^i;i))\right)}}_{:=\Delta_s^{1,i}}{\rm d}s\\
\ \ \ \Dis =\int_\tau^T \mathop{\underbrace{\left(g^i(s,U_s,V_s(Z_s,\widetilde Z_s^i;i))-g^i(s,\widetilde U_s,\widetilde V_s(\widetilde Z_s,\widetilde Z_s^i;i))\right)}}_{:=\Delta_s^{2,i}}{\rm d}s.
\end{array}
\end{equation}
It follows from Assumption \ref{A:B2} that for each $s\in [T-\eps,T]$ and each $i=1,\cdots,n$, we have
\begin{equation}\label{eq:3.32}
|\Delta_s^{1,i}|\leq \phi(|U_s|)\left(v_s+2|V_s|+2|Z_s|+|\widetilde Z_s|\right)|Z_s^i-\widetilde Z_s^i|
\end{equation}
and
\begin{equation}\label{eq:3.33}
\Dis |\Delta_s^{2,i}|\leq \Dis \phi(|U_s|\vee |\widetilde U_s|)\Big[ \tilde v_s |U_s-\widetilde U_s|+\hat v_s \sum_{j=1}^{i-1}|Z_s^j-\widetilde Z_s^j|+\sqrt{n}(\breve v_s+ \theta \hat v_s)|V_s-\widetilde V_s|\Big]
\end{equation}
with
\begin{equation}\label{eq:3.34-1}
\Dis \tilde v_s:=v_s^{1+\delta}+3|V_s|^{1+\delta}+2|\widetilde V_s|^{1+\delta}+3|Z_s|^{1+\delta}+5|\widetilde Z_s|^{1+\delta},\vspace{0.2cm}
\end{equation}
\begin{equation}\label{eq:3.34-2}
\Dis \hat v_s:=v_s+|V_s|+|\widetilde V_s|+|Z_s|+2|\widetilde Z_s|\vspace{-0.1cm}
\end{equation}
and\vspace{-0.1cm}
\begin{equation}\label{eq:3.34-3}
\Dis \breve v_s:=v_s^\delta+|V_s|^{\delta}+|\widetilde V_s|^{\delta}+|Z_s|^{\delta}+2|\widetilde Z_s|^{\delta}.
\vspace{0.3cm}
\end{equation}

For $i=1,\cdots,n$, define $G_s(i):\equiv 0,\ s\in [0,T-\eps)$ and
$$
G_s(i)=:\frac{\left(Z_s^i-\widetilde Z_s^i \right)^\top }{|Z_s^i-\widetilde Z_s^i|^2} {\bf 1}_{|Z_s^i-\widetilde Z_s^i|\neq 0} \Delta^{1,i}_s,\ \ s\in [T-\eps, T].
$$
By \eqref{eq:3.32} we know that for each $i=1,\cdots,n$ and $s\in [T-\eps,T]$,
\begin{equation}\label{eq:3.35}
\Dis \Delta_s^{1,i}=\left(Z_s^i-\widetilde Z_s^i \right)G_s(i)\ \ \ \  {\rm and}\ \ \ \ |G_s(i)|\leq \phi(|U_s|)\left(v_s+2|V_s|+2|Z_s|+|\widetilde Z_s|\right).
\end{equation}
Then for each $i=1,\cdots,n$, $\widetilde B_t(i):=B_t-\int_0^t G_s(i){\rm d}s$ is a Brownian motion under the probability measure $\mathbb{P}^i$ defined by
$$
\frac{{\rm d}\mathbb{P}^i}{{\rm d}\mathbb{P}}:=\exp\left\{\int_0^T [G_s(i)]^\top {\rm d}B_s-{1\over 2}\int_0^T |G_s(i)|^2{\rm d}s\right\},\vspace{0.1cm}
$$
and from the definition of $\mathcal{B}_\eps$, there exists a constant $\bar K>0$ such that
$$\|[G(i)]^\top \|_{{\rm BMO}_{[T-\eps,T]}}^2\leq \bar K:=4[\phi(K)]^2\left(\|v\|_{{\rm BMO}}^2+9K\right).$$
Furthermore, it follows from Theorem 3.3 in \citet{Kazamaki1994book} that there exist constants $0<L_1\leq 1$ and $L_2\geq 1$ depending only on $\bar K$ such that for any $M\in {\rm BMO}(\R^{n\times d})$ or $M\in {\rm BMO}(\R^{1\times d})$, we have, for each $i=1,\cdots,n$,
\begin{equation}\label{eq:3.36}
L_1 \|M\|_{{\rm BMO}_{[T-\eps,T]}}\leq \|M\|_{{\rm BMO}_{[T-\eps,T]}(\mathbb{P}^i)}\leq L_2 \|M\|_{{\rm BMO}_{[T-\eps,T]}},
\end{equation}
where
$$
\|M\|_{{\rm BMO}_{[T-\eps,T]}(\mathbb{P}^i)}:=\sup_{\tau\in \mathcal {T}_{[T-\eps,T]}}\left\|\E_{\tau}^i\left[\int_{\tau}^T |M_s|^2 {\rm d}s\right]\right\|_{\infty}^{1\over 2} \vspace{0.2cm}
$$
and $\E_{\tau}^i$ denotes the conditional expectation with respect to $\F_\tau$ under the measure $\mathbb{P}^i$.\vspace{0.2cm}

It follows from \eqref{eq:3.31} and \eqref{eq:3.35} that
$$
Y_\tau^i-\widetilde Y_\tau^i+\int_\tau^T \left(Z_s^i-\widetilde Z_s^i \right) {\rm d}\widetilde B_s(i)=\int_\tau^T \Delta^{2,i}_s {\rm d}s,\ \ i=1,\cdots,n,\ \ \tau\in \mathcal{T}_{[T-\eps,T]}.
$$
Taking square and then the conditional mathematical expectation under $\mathbb{P}^i$ on both sides of the last equation, in view of \eqref{eq:3.33} and the definition of $\mathcal{B}_\eps$ together with H\"older's inequality we can deduce that for each $i=1,\cdots,n$ and $\tau\in \mathcal{T}_{[T-\eps,T]}$,
\begin{equation}\label{eq:3.37}
\begin{array}{l}
\Dis |Y_\tau^i-\widetilde Y_\tau^i|^2+\E^i_\tau\left[\int_\tau^T \left|Z_s^i-\widetilde Z_s^i \right|^2 {\rm d}s\right]\leq \Dis 3[\phi(K)]^2 \E^i_\tau\left[\left(\int_\tau^T \tilde v_s{\rm d}s\right)^2\right] \|U-\widetilde U\|_{\s^\infty_{[T-\eps,T]}}^2\vspace{0.2cm}\\
\Dis\ \ \  +3n[\phi(K)]^2 \left\{\E^i_\tau\left[\left(\int_\tau^T |\hat v_s|^2{\rm d}s\right)^2\right]\right\}^{1\over 2}\  \sum_{j=1}^{i-1} \left\{\E^i_\tau\left[\left(\int_\tau^T |Z_s^j-\widetilde Z_s^j|^2{\rm d}s\right)^2\right]\right\}^{1\over 2}\vspace{0.2cm}\\
\Dis\ \ \  +3n[\phi(K)]^2 \left\{\E^i_\tau\left[\left(\int_\tau^T |\breve v_s+ \theta \hat v_s|^2{\rm d}s\right)^2\right]\right\}^{1\over 2} \left\{\E^i_\tau\left[\left(\int_\tau^T |V_s-\widetilde V_s|^2{\rm d}s\right)^2\right]\right\}^{1\over 2}.
\end{array}
\end{equation}
It follows from the energy inequality for BMO martingales (see for example Section 2.1 in \citet{Kazamaki1994book}) together with \eqref{eq:3.36} that for each $i=1,\cdots,n$ and $\tau\in \mathcal{T}_{[T-\eps,T]}$,
$$
\begin{array}{lll}
\Dis \sum_{j=1}^{i-1} \left\{\E^i_\tau\left[\left(\int_\tau^T |Z_s^j-\widetilde Z_s^j|^2{\rm d}s\right)^2\right]\right\}^{1\over 2}&\leq & \Dis \sqrt{2}\sum_{j=1}^{i-1}\|Z^j-\widetilde Z^j\|_{{\rm BMO}_{[T-\eps,T]}(\mathbb{P}^i)}^2\\
&\leq & \Dis \sqrt{2} L_2^2\sum_{j=1}^{i-1}\|Z^j-\widetilde Z^j\|_{{\rm BMO}_{[T-\eps,T]}}^2
\end{array}
$$
and
$$
\Dis \left\{\E^i_\tau\left[\left(\int_\tau^T |V_s-\widetilde V_s|^2{\rm d}s\right)^2\right]\right\}^{1\over 2}\leq \Dis \sqrt{2}\|V-\widetilde V\|_{{\rm BMO}_{[T-\eps,T]}(\mathbb{P}^i)}^2\leq  \Dis \sqrt{2}L^2_2 \|V-\widetilde V\|_{{\rm BMO}_{[T-\eps,T]}}^2.\vspace{0.2cm}
$$
Furthermore, in view of \eqref{eq:3.34-1}-\eqref{eq:3.34-3}, using H\"older's inequality and the energy inequality for BMO martingales together with \eqref{eq:3.36} and the definition of $\mathcal{B}_\eps$ we can derive that for each $i=1,\cdots,n$ and $\tau\in \mathcal{T}_{[T-\eps,T]}$,
\begin{equation}\label{eq:3.38a}
\E^i_\tau\left[\left(\int_\tau^T \tilde v_s{\rm d}s\right)^2\right]\leq 10\eps^{1-\delta} L^2_4 \|v\|_{\rm BMO}^4+1000\eps^{1-\delta}\left(1+L^4_2 K^2\right),
\end{equation}
\begin{equation}\label{eq:3.38b}
\Dis \left\{\E^i_\tau\left[\left(\int_\tau^T |\hat v_s|^2{\rm d}s\right)^2\right]\right\}^{1\over 2}\leq 400 L^2_2 \left( \|v\|_{\rm BMO}^2+K \right),
\end{equation}
and
\begin{equation}\label{eq:3.38}
\begin{array}{l}
\Dis \left\{\E^i_\tau\left[\left(\int_\tau^T |\breve v_s+ \theta \hat v_s|^2{\rm d}s\right)^2\right]\right\}^{1\over 2} \vspace{0.2cm}\\
\ \ \Dis \leq 800\theta^2 L^2_2 \left(\|v\|_{\rm BMO}^2+K\right)+800 L^2_2 \left(\|v\|_{\rm BMO}^2+K+1\right)\eps^{1-\delta}.\vspace{0.2cm}
\end{array}
\end{equation}
Their proof will be postponed to the end of this section. Combining those inequalities from \eqref{eq:3.37} to \eqref{eq:3.38} yields that for each $i=1,\cdots, n$,
$$
\begin{array}{l}
\Dis \|Y^i-\widetilde Y^i\|_{\s^\infty_{[T-\eps,T]}}^2+L^2_1 \|Z^i-\widetilde Z^i\|_{{\rm BMO}_{[T-\eps,T]}}^2\vspace{0.2cm}\\
\ \ \leq \Dis  C_3 \eps^{1-\delta}\left(\|U-\widetilde U\|_{\s^\infty_{[T-\eps,T]}}^2+\|V-\widetilde V\|_{{\rm BMO}_{[T-\eps,T]}}^2\right)+2\theta^2C_4 \|V-\widetilde V\|_{{\rm BMO}_{[T-\eps,T]}}^2\vspace{0.2cm}\\
\hspace{0.7cm}\Dis +C_4 \sum_{j=1}^{i-1}\|Z^j-\widetilde Z^j\|_{{\rm BMO}_{[T-\eps,T]}}^2,\ \
\end{array}
$$
and then,
\begin{equation}\label{eq:3.39}
\begin{array}{l}
\Dis \|Y^i-\widetilde Y^i\|_{\s^\infty_{[T-\eps,T]}}^2+\|Z^i-\widetilde Z^i\|_{{\rm BMO}_{[T-\eps,T]}}^2\vspace{0.3cm}\\
\ \ \leq \Dis  {C_3 \eps^{1-\delta}+2\theta^2 C_4 \over L_1^2}\left(\|U-\widetilde U\|_{\s^\infty_{[T-\eps,T]}}^2+\|V-\widetilde V\|_{{\rm BMO}_{[T-\eps,T]}}^2\right) \vspace{0.2cm}\\
\hspace{0.7cm}\Dis +{C_4\over L_1^2} \sum_{j=1}^{i-1}\|Z^j-\widetilde Z^j\|_{{\rm BMO}_{[T-\eps,T]}}^2,
\end{array}
\end{equation}
where
$
\Dis C_3:=\Dis 3[\phi(K)]^2 \left\{10 L_2^4 \|v\|_{\rm BMO}^4+1000\left(1+L_2^4 K^2\right)+800\sqrt{2}n L_2^4 \left(\|v\|_{\rm BMO}^2+K+1\right)\right\}
$
and
$$
C_4 :=1200\sqrt{2}nL_2^4 [\phi(K)]^2 \left(\|v\|_{\rm BMO}^2+K\right).\vspace{0.1cm}
$$

Next, in view of \eqref{eq:3.39}, by induction for $i$ we can deduce that for each $i=1,\cdots,n$,
$$
\begin{array}{ll}
\Dis \sum_{j=1}^{i}\|Y^i-\widetilde Y^i\|_{\s^\infty_{[T-\eps,T]}}^2+\sum_{j=1}^{i}\|Z^i-\widetilde Z^i\|_{{\rm BMO}_{[T-\eps,T]}}^2\vspace{0.2cm}\\
\ \ \leq \Dis  {C_3 \eps^{1-\delta}+2\theta^2 C_4 \over L_1^2} C_5^i \left(\|U-\widetilde U\|_{\s^\infty_{[T-\eps,T]}}^2+\|V-\widetilde V\|_{{\rm BMO}_{[T-\eps,T]}}^2\right),
\end{array}
$$
where
$$
C_5^1:=1\ \ {\rm and}\ \
 C_5^i:=C_5^{i-1}+1+{C_4\over L_1^2}C_5^{i-1}\ {\rm for}\ i=2,\cdots,n,\ {\rm recursively}.\vspace{0.2cm}
$$
In particular, letting $i=n$ in the last equation yields that
$$
\left\|\left(Y-\widetilde Y, Z-\widetilde Z\right)\right\|_{\mathcal{B}_\eps}^2 \leq  {C_3 \eps^{1-\delta}+2\theta^2 C_4 \over L_1^2} C_5^n \left\|\left(U-\widetilde U, V-\widetilde V\right)\right\|_{\mathcal{B}_\eps}^2.
$$
Consequently, there exist two very small positive numbers $\eps_0$ and $\theta_0$ such that for each $\theta\in [0,\theta_0]$, the solution map $\Gamma$ is a contraction on the previously given set $\mathcal{B}_{\eps_0}$. \vspace{0.2cm}

Finally, we remark that in view of (i) of \cref{lem:3.1}, all above arguments remain valid for $p=1$ and $q=0$ when $\lambda=\theta=0$. The proof of \cref{thm:2.3} is then completed.\vspace{0.2cm}

Now, using a similar argument to that in pages 1078-1079 of \citet{HuTang2016SPA}, we give the proof of \eqref{eq:3.38a}-\eqref{eq:3.38}. First of all, from the energy inequality for BMO martingales, \eqref{eq:3.36} and the definition of $\mathcal{B}_{\eps}$, we can deduce that for each $i=1,\cdots,n$ and $\tau\in \mathcal{T}_{[T-\eps,T]}$,
\begin{equation}\label{eq:3.44}
\hspace*{-0.45cm}
\begin{array}{l}
\Dis \left\{\E_{\tau}^i\left[\left(\int_{\tau}^T|V_s|^2{\rm d}s\right)^2\right]\right\}^{1\over 2}\leq \sqrt{2} \|V\|^2_{{\rm BMO}_{[T-\eps,T]}(\p^i)}\leq \sqrt{2} L_2^2  \|V\|^2_{{\rm BMO}_{[T-\eps,T]}}\leq \sqrt{2} L_2^2 K,\vspace{0.1cm}\\
\Dis \left\{\E_{\tau}^i\left[\left(\int_{\tau}^T|\widetilde V_s|^2{\rm d}s\right)^2\right]\right\}^{1\over 2}\leq \sqrt{2} \|\widetilde V\|^2_{{\rm BMO}_{[T-\eps,T]}(\p^i)}\leq \sqrt{2} L_2^2  \|\widetilde V\|^2_{{\rm BMO}_{[T-\eps,T]}}\leq \sqrt{2} L_2^2 K,\vspace{0.1cm}\\
\Dis \left\{\E_{\tau}^i\left[\left(\int_{\tau}^T|Z_s|^2{\rm d}s\right)^2\right]\right\}^{1\over 2}\leq \sqrt{2} \|Z\|^2_{{\rm BMO}_{[T-\eps,T]}(\p^i)}\leq \sqrt{2} L_2^2  \|Z\|^2_{{\rm BMO}_{[T-\eps,T]}}\leq \sqrt{2} L_2^2 K,\vspace{0.1cm}\\
\Dis \left\{\E_{\tau}^i\left[\left(\int_{\tau}^T|\widetilde Z_s|^2{\rm d}s\right)^2\right]\right\}^{1\over 2}\leq \sqrt{2} \|\widetilde Z\|^2_{{\rm BMO}_{[T-\eps,T]}(\p^i)}\leq \sqrt{2} L_2^2  \|\widetilde Z\|^2_{{\rm BMO}_{[T-\eps,T]}}\leq \sqrt{2} L_2^2 K,\vspace{0.1cm}\\
\Dis \left\{\E_{\tau}^i\left[\left(\int_{\tau}^T v_s^2{\rm d}s\right)^2\right]\right\}^{1\over 2}\leq \sqrt{2} \|v\|^2_{{\rm BMO}_{[T-\eps,T]}(\p^i)}\leq \sqrt{2} L_2^2  \|v\|^2_{{\rm BMO}_{[T-\eps,T]}}\leq \sqrt{2} L_2^2  \|v\|^2_{{\rm BMO}}.
\end{array}
\end{equation}
Furthermore, by H\"{o}lder's inequality we have that for each $i=1,\cdots,n$ and $\tau\in \mathcal{T}_{[T-\eps,T]}$,
\begin{equation}\label{eq:3.45}
\begin{array}{l}
\Dis \E^i_\tau\left[\left(\int_\tau^T \tilde v_s{\rm d}s\right)^2\right]\leq 5\eps^{1-\delta}\E^i_\tau\left[\left(\int_\tau^T v_s^2{\rm ds}\right)^{1+\delta}\right]+45\eps^{1-\delta}\E^i_\tau\Bigg[\left(\int_\tau^T |V_s|^2{\rm ds}\right)^{1+\delta}\Bigg]\vspace{0.1cm}\\
\ \ \ \ \ \ \Dis +5\eps^{1-\delta}\E^i_\tau\Bigg[4\left(\int_\tau^T |\tilde V_s|^2{\rm ds}\right)^{1+\delta}+9\left(\int_\tau^T |Z_s|^2{\rm ds}\right)^{1+\delta}+25\left(\int_\tau^T |\tilde Z_s|^2{\rm ds}\right)^{1+\delta}\Bigg]\vspace{0.1cm}\\
\ \ \Dis \leq 240\eps^{1-\delta}+5\eps^{1-\delta}\E^i_\tau\left[\left(\int_\tau^T v_s^2{\rm ds}\right)^2\right]+45\eps^{1-\delta}\E^i_\tau\Bigg[\left(\int_\tau^T |V_s|^2{\rm ds}\right)^2\Bigg]\vspace{0.1cm}\\
\ \ \ \ \ \ \Dis +125\eps^{1-\delta}\E^i_\tau\Bigg[\left(\int_\tau^T |\tilde V_s|^2{\rm ds}\right)^2+\left(\int_\tau^T |Z_s|^2{\rm ds}\right)^2+\left(\int_\tau^T |\tilde Z_s|^2{\rm ds}\right)^2\Bigg],\vspace{0.1cm}
\end{array}
\end{equation}
\begin{equation}\label{eq:3.46}
\begin{array}{l}
\Dis \E^i_\tau\left[\left(\int_\tau^T |\breve v_s|^2{\rm d}s\right)^2\right]\leq 25\E_{\tau}^i\left[\left(\int_{\tau}^T\left(v_s^{2\delta}+|V_s|^{2\delta}+|\widetilde V_s|^{2\delta}+|Z_s|^{2\delta}+4|\widetilde Z_s|^{2\delta}\right){\rm d}s\right)^2\right]\vspace{0.1cm}\\
\ \ \Dis \leq 125\eps^{2(1-\delta)}\E_{\tau}^i\left[\left(\int_{\tau}^T v_s^2 {\rm d}s\right)^{2\delta}+\left(\int_{\tau}^T |V_s|^2 {\rm d}s\right)^{2\delta}+\left(\int_{\tau}^T |\widetilde V_s|^2 {\rm d}s\right)^{2\delta}\right]\vspace{0.1cm}\\
\ \ \ \ \ \ \Dis +125\eps^{2(1-\delta)}\E_{\tau}^i\left[\left(\int_{\tau}^T |Z_s|^2 {\rm d}s\right)^{2\delta}+16\left(\int_{\tau}^T |\widetilde Z_s|^2 {\rm d}s\right)^{2\delta}\right]\vspace{0.1cm}\\
\ \ \Dis \leq 2500\eps^{2(1-\delta)}+125\eps^{2(1-\delta)}\E_{\tau}^i\left[\left(\int_{\tau}^T v_s^2 {\rm d}s\right)^2+\left(\int_{\tau}^T |V_s|^2 {\rm d}s\right)^2\right]\vspace{0.1cm}\\
\ \ \ \ \ \ \Dis +2000\eps^{2(1-\delta)}\E_{\tau}^i\left[\left(\int_{\tau}^T |\widetilde V_s|^2 {\rm d}s\right)^2+\left(\int_{\tau}^T |Z_s|^2 {\rm d}s\right)^2+\left(\int_{\tau}^T |\widetilde Z_s|^2 {\rm d}s\right)^2\right]
\end{array}
\end{equation}
and
\begin{equation}\label{eq:3.47}
\begin{array}{l}
\Dis \E^i_\tau\left[\left(\int_\tau^T |\hat v_s|^2{\rm d}s\right)^2\right]\leq 25\E^i_\tau\Bigg[\left(\int_\tau^T \left(v_s^2+|V_s|^2+|\tilde V_s|^2+|Z_s|^2+4|\tilde Z_s|^2\right){\rm ds}\right)^2\Bigg]\vspace{0.1cm}\\
\ \ \Dis \leq 125\E^i_\tau\Bigg[\left(\int_\tau^T v_s^2{\rm ds}\right)^2+\left(\int_\tau^T |V_s|^2{\rm ds}\right)^2+\left(\int_\tau^T |\tilde V_s|^2{\rm ds}\right)^2+\left(\int_\tau^T |Z_s|^2{\rm ds}\right)^2\Bigg]\vspace{0.1cm}\\
\Dis \ \ \ \ \ \ +2000\E^i_\tau\Bigg[\left(\int_\tau^T |\tilde Z_s|^2{\rm ds}\right)^2\Bigg].
\end{array}
\end{equation}
Finally, in view of
$$
\begin{array}{l}
\Dis \left\{\E^i_\tau\left[\left(\int_\tau^T |\breve v_s+ \theta \hat v_s|^2{\rm d}s\right)^2\right]\right\}^{1\over 2}\vspace{0.1cm}\\
\Dis \ \ \leq 2\left\{\E^i_\tau\left[\left(\int_\tau^T |\breve v_s|^2{\rm d}s\right)^2\right]\right\}^{1\over 2}+2\theta^2 \left\{\E^i_\tau\left[\left(\int_\tau^T |\hat v_s|^2{\rm d}s\right)^2\right]\right\}^{1\over 2},
\end{array}
$$
the desired inequalities \eqref{eq:3.38a}-\eqref{eq:3.38} follows immediately from \eqref{eq:3.44}-\eqref{eq:3.47}.

\section{Global bounded solution: proof of Theorems~\ref{thm:2.5} and \ref{thm:2.8a}}
\label{sec:4-global bounded solution}
\setcounter{equation}{0}

To prove the existence of global bounded solution, we need  some uniform estimates of the solution.

\subsection{Proof of \cref{thm:2.5}\vspace{0.2cm}}

To prove \cref{thm:2.5}, we need the following proposition.

\begin{pro}\label{pro:4.1}
Let $\xi\in L^{\infty}(\R^n)$, $\alpha\in \ecal^{\infty}(p\gamma\exp(\beta T))$ for some $p>1$ and the generator $g$ satisfy Assumption \ref{A:C1a}. Assume that for some $h\in (0,T]$, BSDE \eqref{eq:2.1} has a solution $(Y,Z)\in \s^{\infty}_{[T-h,T]}(\R^n)\times {\rm BMO}_{[T-h,T]}(\R^{n\times d})$ on the time interval $[T-h,T]$. Then, there exists a $\widetilde K>0$ depending only on $(\|\xi\|_{\infty}, \|\alpha\|_{\ecal^{\infty}(p\gamma\exp(\beta T))}, \|\bar\alpha\|_{\mcal^{\infty}}, \|\tilde\alpha\|_{\lcal^{\infty}}, \|v\|_{{\rm BMO}}, n,\beta,\gamma,\bar\gamma,\lambda,\bar\lambda,c,\bar c,\delta,T,p)$ and being independent of $h$ such that for each $\theta\in [0,1/[4n(q\gamma\vee 1)\widetilde K]]$ with $q:={p\over p-1}$, we have
$$
\|Y\|_{\s^{\infty}_{[T-h,T]}}+\|Z\|_{{\rm BMO}_{[T-h,T]}}^2 \leq \widetilde K.
$$
Moreover, the preceding assertion is still true for $p=1$ and $q=0$ when $\lambda=0$ and $\theta=0$.
\end{pro}

\begin{proof}
We only prove the case of $p>1$. The other case can be proved in the same way. \vspace{0.2cm}

First of all, define an increasing non-negative real-valued function $\Phi:[0,+\infty)\To [0,+\infty)$ as follows:\vspace{0.1cm}
$$
\begin{array}{l}
\Phi(x):=\Dis {2(1+\beta T)+\bar\gamma\over \bar\gamma}\exp\left(2\beta T\right)\bigg[2+2n\bar\lambda T^{1-\delta\over 2}+{\ln 2\over \gamma\exp(\beta T) }+\|\alpha\|_{\ecal^{\infty}(p\gamma \exp(\beta T))}+2\left\|\bar\alpha\right\|_{\mcal^{\infty}}\vspace{0.2cm}\\
\hspace{1.5cm} \Dis +3x\bigg]+ {4(\gamma +1) \over \gamma^2}\exp\left\{4\gamma \exp(\beta T)\left(\|\xi\|_{\infty}+\|\tilde\alpha\|_{\lcal^{\infty}}\right)\right\}
\bigg[\beta T \exp(\beta T)\left(\|\xi\|_{\infty}+
\|\tilde\alpha\|_{\lcal^{\infty}}\right)\vspace{0.2cm}\\
\hspace{1.5cm} \Dis +1+\|\tilde\alpha\|_{\lcal^{\infty}}+2\|v\|^2_{{\rm BMO}_{[t,T]}}\bigg]+1+c_0 \exp\left(2\beta T+ 2\bar\lambda^2 T\right)\left(6
+6\|\bar\alpha\|_{\mcal^\infty}^2+x^2\right),\vspace{0.2cm}
\end{array}
$$
and define the following four constants:
$$
\bar C_1:= \Phi(\|\xi\|_{\infty}),\ \ \ \ \ \ \ \bar C_2 := \Dis {6\beta(1+\beta T)+3\beta\bar\gamma\over \bar\gamma} \exp\left(2\beta T\right),
$$
$$
\bar C_3 :=\Dis {2(1+\beta T)+\bar\gamma\over \bar\gamma}\exp\left(2\beta T\right)\left(C_{p,n,\beta,\gamma,\lambda,T}+2n\bar\lambda \right)+12nc_0\bar\lambda^2 T \exp\left(2\beta T+2\bar\lambda^2 T\right)\vspace{0.2cm}
$$
and
$$
\bar C_4 :=\Dis {4\bar c(1+\beta T)+2c\bar\gamma\over \bar\gamma}\exp\left(2\beta T\right)+{8nc^2(\gamma +1) \over \gamma^2}\exp\left\{4\gamma \exp(\beta T)\left(\|\xi\|_{\infty}+
\|\tilde\alpha\|_{\lcal^{\infty}}\right)\right\},\vspace{0.2cm}
$$
where the uniform constant $c_0>0$ is defined in (i) of \cref{pro:A.3} in Appendix, and
\begin{equation}\label{eq:4.1}
C_{p,n,\beta,\gamma,\lambda,T}:=
{n\gamma p\lambda^2\over 2(p-1)}\exp(\beta T).\vspace{0.2cm}
\end{equation}

Note that $(Y,Z)\in \s^{\infty}_{[T-h,T]}\times {\rm BMO}_{[T-h,T]}$ is  a solution of BSDE \eqref{eq:2.2} on $[T-h,T]$. Then for each $i=1,\cdots,n$, $(Y,Z)$ also solves the following BSDE:
\begin{equation}\label{eq:4.2}
Y_t^i=\xi^i+\int_t^T f^i(s,Y_s^i,Z_s^i) {\rm d}s-\int_t^T Z_s^i {\rm d}B_s, \ \ t\in [T-h,T],
\end{equation}
where for each $(\omega,t,y,z)\in \Omega\times [T-h,T]\times\R\times\R^{1\times d}$,
\begin{equation}\label{eq:4.3}
f^i(t,y,z):=g^i(t,Y_t(y;i),Z_t(z;i)).
\end{equation}
In the sequel, let $q:={p\over p-1}$ and the constant $\theta$ always satisfy
\begin{equation}\label{eq:4.4}
4n(q\gamma\vee 1)\theta \|Z\|_{{\rm BMO}_{[T-h,T]}}^2\leq 1.
\end{equation}
We will first prove that for each $i=1,\cdots,n$ and each $t\in [T-h,T]$, if $T-t\leq 1$, then
\begin{equation}\label{eq:4.5}
\hspace*{-0.2cm}
\begin{array}{l}
\Dis \|Y^i\|_{\s^{\infty}_{[T-t,T]}}+\|Z^i\|_{{\rm BMO}_{[T-t,T]}}^2 \\
\ \ \leq \Dis \bar C_1+\bar C_2 \|Y\|_{\s^{\infty}_{[T-t,T]}} (T-t)+\bar C_3 \|Z\|_{{\rm BMO}_{[T-t,T]}}^2 (T-t)^{1-\delta\over 2}+\bar C_4 \sum_{j=1}^{i-1} \|Z^j\|_{{\rm BMO}_{[T-t,T]}}^{2}.
\end{array}
\end{equation}
Indeed, for any fixed $i=1,\cdots,n$, since $g$ satisfies Assumption \ref{A:C1a}, in view of \cref{rmk:2.4}, we need to consider the following three cases.\vspace{0.2cm}

(1) $g^i$ satisfies (i) of Assumption \ref{A:C1a}. In this case, it follows from \eqref{eq:4.3} that $\as$ on $\Omega\times [T-h,T]$, for each $(y,z)\in \R\times \R^{1\times d}$, we have
\begin{equation}\label{eq:4.6}
f^i(t,y,z)\, {\bf 1}_{y>0} \leq  \Dis \check\alpha_t+\beta |y|+\frac{\gamma}{2} |z|^2\quad {\rm and} \quad
f^i(t,y,z)\geq \Dis \frac{\bar \gamma}{2} |z|^2-\dot\alpha_t-\beta |y|,
\end{equation}
where
$$
\check\alpha_t:=\alpha_t+\beta |Y_t| +\sum_{j\neq i}\left(\lambda |Z^j_t|+\theta |Z^j_t|^2\right)
$$
and
$$
\dot\alpha_t:=\bar\alpha_t+\beta |Y_t|+\bar c\sum_{j=1}^{i-1}|Z^j_t|^2+\sum_{j=i+1}^{n} \left(\bar\lambda |Z^j_t|^{1+\delta}+\theta |Z^j_t|^2\right).\vspace{0.2cm}
$$
Since $\alpha\in \ecal^{\infty}_{[T-h,T]}(p\gamma \exp(\beta T))$, $\bar\alpha\in \mcal^{\infty}_{[T-h,T]}$ and $|Y|\in \lcal^\infty_{[T-h,T]}$, in view of (ii) and (i) of \cref{rmk:2.1} as well as \eqref{eq:4.4}, applying (i) of \cref{lem:3.1} with
$$
r=\gamma \exp(\beta T),\ \bar\delta=0,\ \bar t=T-h,\ V=\bar V=Z,\ \breve\alpha=\check\alpha\ \ {\rm and}\ \ \hat\alpha=\alpha+\beta |Y|
$$
and (ii) of \cref{lem:3.1} with
$$\bar\delta=\delta,\ \bar t=T-h,\ V=\bar V=Z,\ \breve\alpha=\dot\alpha\ \ {\rm and}\ \ \hat\alpha=\bar\alpha+\beta |Y|
$$
we can deduce that
\begin{equation}\label{eq:4.7}
\begin{array}{lll}
\Dis \|\check\alpha\|_{\ecal^{\infty}_{[t,T]}(\gamma \exp(\beta T))} &\leq &\Dis \|\alpha\|_{\ecal^{\infty}_{[t,T]}(p\gamma \exp(\beta T))}+\beta \|Y\|_{\s^\infty_{[t,T]}}(T-t)+{\ln 2\over \gamma \exp(\beta T)}\vspace{0.2cm}\\
&&\Dis +C_{p,n,\beta,\gamma,\lambda,T}\|Z\|_{{\rm BMO}_{[t,T]}}^2 (T-t)<+\infty,\ \ t\in [T-h,T]
\end{array}
\end{equation}
and
\begin{equation}\label{eq:4.8}
\begin{array}{lll}
\Dis \left\|\dot\alpha\right\|_{\mcal^{\infty}_{[t,T]}} &\leq &\Dis 1+\left\|\bar\alpha\right\|_{\mcal^{\infty}_{[t,T]}}+\beta \|Y\|_{\s^\infty_{[t,T]}} (T-t)+n\bar\lambda \|Z\|_{{\rm BMO}_{[t,T]}}^{1+\delta}(T-t)^{1-\delta\over 2}\vspace{0.2cm}\\
&&\Dis +\bar c\sum_{j=1}^{i-1}\|Z^j\|^2_{{\rm BMO}_{[t,T]}}<+\infty,\ \ t\in [T-h,T],
\end{array}
\end{equation}
where the constant $C_{p,n,\beta,\gamma,\lambda,T}$ is defined in \eqref{eq:4.1}. Combining \eqref{eq:4.6}, \eqref{eq:4.7} and \eqref{eq:4.8} yields that on $\Omega\times [T-h,T]$, the generator $f^i$ satisfies the second inequality for the case of $y>0$ and the first inequality in Assumption \ref{A:A1} with $\bar\beta=\beta$ and $\varphi(\cdot)\equiv 0$. It then follows from (i) of \cref{pro:A.1} that for each $t\in [T-h,T]$ such that $T-t\leq 1$, we have\vspace{0.2cm}
\begin{equation}\label{eq:4.9}
\begin{array}{l}
\Dis \|Y^i\|_{\s^{\infty}_{[t,T]}}+\|Z^i\|_{{\rm BMO}_{[t,T]}}^2 \leq  \Dis {2(1+\beta T)+\bar\gamma\over \bar\gamma}\exp\left(2\beta T\right)\bigg\{2+2n\bar\lambda T^{1-\delta\over 2}+{\ln 2\over \gamma\exp(\beta T) } \vspace{0.2cm}\\
\Dis \hspace{0.5cm} +3\|\xi^i\|_{\infty} +\|\alpha\|_{\ecal^{\infty}(p\gamma \exp(\beta T))}+2\left\|\bar\alpha\right\|_{\mcal^{\infty}}+3\beta \|Y\|_{\s^\infty_{[t,T]}} (T-t)\vspace{0.2cm}\\
\Dis \hspace{0.5cm} +\left(C_{p,n,\beta,\gamma,\lambda,T}+2n\bar\lambda \right) \|Z\|_{{\rm BMO}_{[t,T]}}^2 (T-t)^{1-\delta\over 2}+2\bar c\sum_{j=1}^{i-1}\|Z^j\|^2_{{\rm BMO}_{[t,T]}}\Big\}.
\end{array}
\end{equation}
Then, it follows from \eqref{eq:4.9} together with the definitions of constants $\bar C_1$, $\bar C_2$, $\bar C_3$ and $\bar C_4$ that the desired inequality \eqref{eq:4.5} holds in this case when $T-t\leq 1$.\vspace{0.2cm}

(2) $g^i$ satisfies (ii) of Assumption \ref{A:C1a}. For this case, it follows from \eqref{eq:4.3} that $\as$ on $\Omega\times [T-h,T]$, for each $(y,z)\in \R\times \R^{1\times d}$, we have
$$
f^i(t,y,z){\rm sgn}(y) \leq \tilde\alpha_t+\beta |y|+\bar u_t |z|+\frac{\gamma}{2} |z|^2,
$$
where
$$
\bar u:=c\sum_{j=1}^{i-1}|Z^j|+v\in {\rm BMO}_{[T-h,T]}(\R).\vspace{0.1cm}
$$
This means that on $\Omega\times [T-h,T]$, the generator $f^i$ satisfies the first inequality for the case of $y<0$ and the second inequality for the case of $y>0$ in \ref{A:A3} with $\ddot\alpha=\tilde\alpha$ and $\bar\beta=\beta$. It then follows from (i) of \cref{pro:A.2} with $r=T$ \vspace{0.2cm}that for each $t\in [T-h,T]$,
\begin{equation}\label{eq:4.10}
\hspace{-0.2cm}
\begin{array}{l}
\Dis\|Y^i\|_{\s^\infty_{[t,T]}}+\|Z^i\|^2_{{\rm BMO}_{[t,T]}}\leq  \Dis {4(\gamma +1) \over \gamma^2}\exp\left\{4\gamma \exp(\beta T)\left(\|\xi^i\|_{\infty}+\|\tilde\alpha\|_{\lcal^{\infty}}\right)\right\} \big(1+\|\tilde\alpha\|_{\lcal^{\infty}}\\
\Dis \hspace{2cm} +\beta T \exp(\beta T)\left(\|\xi^i\|_{\infty}+\|\tilde\alpha\|_{\lcal^{\infty}}\right)
+2\|v\|^2_{{\rm BMO}_{[t,T]}}+2nc^2 \sum_{j=1}^{i-1} \|Z^j\|^2_{{\rm BMO}_{[t,T]}}\big).
\end{array}
\end{equation}
Then, it follows from \eqref{eq:4.10} together with the definitions of constants $\bar C_1$, $\bar C_2$, $\bar C_3$ and $\bar C_4$ that the desired inequality \eqref{eq:4.5} holds for each $t\in [T-h,T]$ in this case.\vspace{0.2cm}

(3) $g^i$ satisfies (iii) of Assumption \ref{A:C1a}. For this case, it follows from \eqref{eq:4.3} that $\as$ on $\Omega\times [T-h,T]$, for each $(y,z)\in \R\times \R^{1\times d}$, we have
\begin{equation}\label{eq:4.11}
f^i(t,y,z)\, {\rm sgn}(y) \leq  \Dis \breve\alpha_t+\beta |y|+\bar\lambda |z|,
\end{equation}
where
$$
\breve\alpha_t:=\bar\alpha_t+\sum_{j\neq i}\left(\bar\lambda |Z^j_t|+\theta |Z^j_t|^2\right).\vspace{0.1cm}
$$
In view of \eqref{eq:4.4}, using (iii) of \cref{lem:3.1} with $\bar\delta=\delta$, $\bar t=T-h$, $V=\bar V=Z$ and $\hat\alpha=\bar\alpha$ we can deduce that
\begin{equation}\label{eq:4.12}
\Dis \left\|\breve\alpha\right\|_{\mcal^{\infty}_{[t,T]}} \leq \Dis 1+\left\|\bar\alpha\right\|_{\mcal^{\infty}_{[t,T]}}+\bar\lambda \sqrt{n} \|Z\|_{{\rm BMO}_{[t,T]}} (T-t)^{1\over 2}<+\infty,\ \ \ t\in [T-h,T].\vspace{0.1cm}
\end{equation}
Combining \eqref{eq:4.11} and \eqref{eq:4.12} yields that on $\Omega\times [T-h,T]$, the generator $f^i$ satisfies the first inequality for the case of $y<0$ and the second inequality for the case of $y>0$ in Assumption \ref{A:A4} with $\dot\alpha=\breve\alpha$ and $\bar\beta=\beta$. It follows from (i) of \cref{pro:A.3} that for each $t\in [T-h,T]$,
\begin{equation}\label{eq:4.13}
\begin{array}{lll}
\Dis \|Y^i\|_{\s^\infty_{[t,T]}}+\left\|Z^i\right\|_{{\rm BMO}_{[t,T]}}^2 &\leq & \Dis 1+c_0 \exp\left(2\beta T+ 2\bar\lambda^2 T\right)\left(6+\|\xi^i\|_\infty^2
+6\|\bar\alpha\|_{\mcal^\infty}^2\right)\\
&& \Dis +12nc_0\bar\lambda^2 T \exp\left(2\beta T+2\bar\lambda^2 T\right) \|Z\|_{{\rm BMO}_{[t,T]}}^2(T-t).
\end{array}
\end{equation}
Then, it follows from \eqref{eq:4.13} together with the definitions of constants $\bar C_1$, $\bar C_2$, $\bar C_3$ and $\bar C_4$ that the desired inequality \eqref{eq:4.5} holds in this case when $T-t\leq 1$.\vspace{0.2cm}

Furthermore, in view of \eqref{eq:4.5}, by induction for $i$ it is not difficult to derive that for each $i=1,\cdots,n$ and each $t\in [T-h,T]$, if $T-t\leq 1$, then
\begin{equation}\label{eq:4.14}
\begin{array}{l}
\Dis \sum_{j=1}^{i}\left(\|Y^j\|_{\s^{\infty}_{[t,T]}}+\|Z^j\|_{{\rm BMO}_{[t,T]}}^2\right)\\
\ \ \leq \Dis \left(\bar C_1+\bar C_2 \|Y\|_{\s^{\infty}_{[T-t,T]}} (T-t)+\bar C_3 \|Z\|_{{\rm BMO}_{[T-t,T]}}^2(T-t)^{1-\delta\over 2}\right)\bar C_5^i,
\end{array}
\end{equation}
where
\begin{equation}\label{eq:4.15}
\bar C_5^1:=1\ \ {\rm and}\ \
\bar C_5^i:=\bar C_5^{i-1}+1+\bar C_4\bar C_5^{i-1}\ {\rm for}\ i=2,\cdots,n,\ {\rm recursively}.
\end{equation}
In particular, letting $i=n$ in \eqref{eq:4.14} yields that
for each $t\in [T-h,T]$, if $T-t\leq 1$, then
$$
\begin{array}{lll}
\Dis \|Y\|_{\s^{\infty}_{[t,T]}}+\|Z\|_{{\rm BMO}_{[t,T]}}^2 &\leq & \Dis \bar C_1\bar C_5^n+\bar C_2 \bar C_5^n\|Y\|_{\s^{\infty}_{[T-t,T]}} (T-t)\\
&&\Dis +\bar C_3 \bar C_5^n\|Z\|_{{\rm BMO}_{[T-t,T]}}^2(T-t)^{1-\delta\over 2}.
\end{array}
$$
Then, we have
\begin{equation}\label{eq:4.16}
\|Y\|_{\s^{\infty}_{[T-\eps,T]}}+\|Z\|_{{\rm BMO}_{[T-\eps,T]}}^2 \leq \Dis 2\bar C_1\bar C_5^n,
\end{equation}
where
$$
\eps:=\min\Bigg\{h,\  1, \ {1\over 2\bar C_2 \bar C_5^n}, \ \Big({1\over 2\bar C_3 \bar C_5^n}\Big)^{2\over 1-\delta}\Bigg\}>0.\vspace{0.2cm}
$$

Finally, for $m\geq 1$, we define successively the following constants:
\begin{equation}\label{eq:4.17}
\bar C_6^1:= 2\bar C_1\bar C_5^n\ \ \ {\rm and}\ \ \ \bar C_6^{m+1}:=\bar C_6^m+2\Phi(\bar C_6^m) \bar C_5^n,
\end{equation}
where $\bar C_5^n$ is defined in \eqref{eq:4.15}. And, we let $m_0$ be the unique positive integer satisfying
$T-h \in [T-m_0\eps,T-(m_0-1)\eps)$, or equivalently,
\begin{equation}\label{eq:4.18}
{h\over \eps}\leq m_0 <{h\over \eps}+1.
\end{equation}
If $m_0=1$, it then follows from \eqref{eq:4.16} and \eqref{eq:4.17} that
$$
\|Y\|_{\s^{\infty}_{[T-h,T]}}+\|Z\|_{{\rm BMO}_{[T-h,T]}}^2 \leq \Dis \bar C_6^1=\bar C_6^{m_0}.
$$
If $m_0=2$, it then follows from \eqref{eq:4.16}  and \eqref{eq:4.17} that
\begin{equation}\label{eq:4.19}
\|Y_{T-\eps}\|_{\infty}\leq \|Y\|_{\s^{\infty}_{[T-\eps,T]}}\leq \bar C_6^1.
\end{equation}
Now, consider the following system of BSDEs
$$
  Y_t=Y_{T-\eps}+\int_t^{T-\eps} g(s,Y_s,Z_s){\rm d}s-\int_t^{T-\eps} Z_s {\rm d}B_s, \ \ t\in [T-h,T-\eps].
$$
In view of \eqref{eq:4.19} and the definition of the function $\Phi(\cdot)$, by virtue of \cref{pro:A.1,pro:A.2,pro:A.3} we can use a similar argument as that obtaining \eqref{eq:4.5} to get that for each $i=1,\cdots,n$,
$$
\begin{array}{l}
\Dis \|Y^i\|_{\s^{\infty}_{[T-h,T-\eps]}}+\|Z^i\|_{{\rm BMO}_{[T-h,T-\eps]}}^2 \\
\ \ \leq \Dis \Phi(\bar C_1^1)+\bar C_2 \|Y\|_{\s^{\infty}_{[T-h,T-\eps]}} \eps +\bar C_3 \|Z\|_{{\rm BMO}_{[T-h,T-\eps]}}^2 \eps^{1-\delta\over 2}+\bar C_4 \sum_{j=1}^{i-1} \|Z^j\|_{{\rm BMO}_{[T-h,T-\eps]}}^{2}.
\end{array}
$$
We would like to mention that for the proof of the last assertion, we need to use $T-\eps$ instead of $T$ and $Y_{T-\eps}$ instead of $\eta$ in \cref{pro:A.1,pro:A.3}, and use $T-\eps$ instead of $r$ in \cref{pro:A.2}. Moreover, in view of the definition of $\eps$, by induction and a similar argument as that from \eqref{eq:4.14} to \eqref{eq:4.16} we can further get that
$$
\|Y\|_{\s^{\infty}_{[T-h,T-\eps]}}+\|Z\|_{{\rm BMO}_{[T-h,T-\eps]}}^2 \leq \Dis 2\Phi(\bar C_6^1) \bar C_5^n.
$$
Combining  the last inequality and \eqref{eq:4.16} yields that, in view of \eqref{eq:4.17},
$$
\|Y\|_{\s^{\infty}_{[T-h,T]}}+\|Z\|_{{\rm BMO}_{[T-h,T]}}^2 \leq \Dis \bar C_6^1+2\Phi(\bar C_6^1) \bar C_5^n = \bar C_6^2=\bar C_6^{m_0},
$$
Proceeding the above computation gives that if $m_0$ satisfies \eqref{eq:4.18}, then
\begin{equation}\label{eq:4.20}
\|Y\|_{\s^{\infty}_{[T-h,T]}}+\|Z\|_{{\rm BMO}_{[T-h,T]}}^2 \leq \bar C_6^{m_0}=\bar C_6^{\left\lceil {h\over \eps}\right\rceil}\leq \bar C_6^{1+\left\lceil {T\over \eps_0}\right\rceil}=:\widetilde K,
\end{equation}
where $\bar C_6^{m_0}$ is defined in \eqref{eq:4.17}, $\lceil x\rceil$ stands for the minimum of integers which is equel to or bigger than $x\in \R$, and
$$
\eps_0:=\min\Bigg\{{1\over 2\bar C_2 \bar C_5^n}, \ \Big({1\over 2\bar C_3 \bar C_5^n}\Big)^{2\over 1-\delta}\Bigg\}\in (0,1).
$$
Then, the desired conclusion follows from \eqref{eq:4.4} and \eqref{eq:4.20}. The proof  is complete.\vspace{0.2cm}
\end{proof}

\begin{proof}[Proof of \cref{thm:2.5}]
With \cref{thm:2.3} and \cref{pro:4.1} in hands, we can closely follow the proof of Theorem 4.1 in \citet{cheridito2015Stochastics} to prove our \cref{thm:2.5}. All the details are omitted here.\vspace{0.2cm}
\end{proof}

\subsection{Proof of \cref{thm:2.8a}\vspace{0.2cm}}

To prove \cref{thm:2.8a}, we need  the following proposition.

\begin{pro}\label{pro:4.2c}
Let $\xi\in L^\infty(\R^n)$ and the generator $g$ satisfy Assumption \ref{A:C1b}. Assume that for some $h\in (0,T]$, BSDE \eqref{eq:2.1} has a solution $(Y,Z)\in \s^{\infty}_{[T-h,T]}(\R^n)\times {\rm BMO}_{[T-h,T]}(\R^{n\times d})$ on the time interval $[T-h,T]$. Then, there exists a positive constant $\widetilde K>0$ depending only on $(\|\xi\|_{\infty}, \|\tilde\alpha\|_{\lcal^{\infty}},  n,\beta,\gamma,\bar\gamma,\lambda,\delta,T)$ and being independent of $h$ such that
$$
\|Y\|_{\s^{\infty}_{[T-h,T]}}\leq \widetilde K.\vspace{0.2cm}
$$
\end{pro}

\begin{proof}
Let $\xi\in L^\infty(\R^n)$ and $\tilde\alpha\in \lcal^{\infty}$ such that
\begin{equation}\label{eq:4.21c}
\|\xi\|_\infty\leq C_1\ \ {\rm and}\ \ \left\|\int_0^T \tilde\alpha_t{\rm d}t\right\|_\infty\leq C_2
\end{equation}
for two positive constants $C_1$ and $C_2$. Denote respectively by $n_1$, $n_2$ and $n_3$ the number of elements in $J_1$, $J_2$ and $J_3$ such that $n_1+n_2+n_3=n$. Since the generator $g$ satisfies Assumption \ref{A:C1b}, in view of (i) of \cref{rmk:2.7a}, we need to consider the following three cases.\vspace{0.2cm}

(1) For $i\in J_1$, $g^i$ satisfies (i) of Assumption \ref{A:C1b}. In this case, for each $t\in [T-h,T]$, we have
$$
\Dis g^i\left(\omega,t,Y_t(\omega),Z_t(\omega)\right)\, {\rm sgn}(Y_t^i(\omega))\leq \tilde\alpha_t(\omega)+\beta |Y_t(\omega)|+ \lambda\sum_{j\in J_1}|Z_t^j(\omega)|^{1+\delta}+\frac{\gamma}{2} |Z_t^i(\omega)|^2
$$
and
$$
g^i\left(\omega,t,Y_t(\omega),Z_t(\omega)\right)\geq \frac{\bar\gamma}{2} |Z_t^i(\omega)|^2-\tilde\alpha_t(\omega)-\beta |Y_t(\omega)|-\lambda \sum_{j\in J_1}|Z_t^j(\omega)|^{1+\delta}.
$$
Note that $(Y^i,Z^i)\in \s^{\infty}_{[T-h,T]}(\R)\times {\rm BMO}_{[T-h,T]}(\R^{1\times d})$ is a solution of BSDE \eqref{eq:2.2} on $[T-h,T]$. According to Lemma A.2 in \citet{FanHuTang2023JDE} and in view of the last two inequalities together with \eqref{eq:4.21c}, we can get that for each $i\in J_1$ and $t\in [T-h,T]$,
\begin{equation}\label{eq:4.22c}
\begin{array}{lll}
\Dis \exp\left(\gamma |Y_t^i|\right)&\leq &\Dis \exp\Big(\gamma(C_1+C_2 )+\beta\gamma \int_t^T \|Y\|_{\s^{\infty}_{[s,T]}}{\rm d}s\Big)\vspace{0.1cm}\\
&&\Dis\  \times \, \E_t\Big[\exp\Big(\gamma\lambda \sum_{j\in J_1}\int_t^T |Z_s^j|^{1+\delta} {\rm d}s\Big)\Big]
\end{array}
\end{equation}
and
\begin{equation}\label{eq:4.23c}
\begin{array}{l}
\Dis \E_t\left[\exp\left(\frac{\bar \gamma}{2}\eps_0 \int_t^T |Z_s^i|^2{\rm d}s \right)\right]\vspace{0.2cm}\\
\ \ \leq \Dis\E_t\Big[\exp\Big(6\eps_0 \|Y^i\|_{\s^{\infty}_{[t,T]}}+3\eps_0\left\|\int_0^T\tilde\alpha_t{\rm d}s\right\|_\infty\\
\hspace{2.4cm}\Dis +3\eps_0 \beta \int_t^T |Y_s|{\rm d}s+3\eps_0 \lambda \sum_{j\in J_1}\int_t^T |Z_s^j|^{1+\delta} {\rm d}s\Big)\Big]\\
\ \ \leq \Dis \exp\Big(6\eps_0 \sum_{j\in J_1}\|Y^j\|_{\s^{\infty}_{[t,T]}}+3\eps_0 C_2+3\eps_0 \beta\int_t^T \|Y\|_{\s^{\infty}_{[s,T]}}{\rm d}s \Big)\\
\hspace{1cm} \Dis \times \,  \E_t\Big[\exp\Big(3\eps_0 \lambda \sum_{j\in J_1}\int_t^T |Z_s^j|^{1+\delta} {\rm d}s\Big)\Big],
\end{array}
\end{equation}
where
$$
\eps_0:=\frac{\bar\gamma}{9}\bigwedge \frac{\gamma}{24}>0.\vspace{0.1cm}
$$
Furthermore, in view of \eqref{eq:4.23c}, by H\"{o}lder's inequality we obtain that for $t\in [T-h,T]$,\vspace{0.1cm}
\begin{equation}\label{eq:4.24c}
\begin{array}{l}
\Dis \E_t\Big[\exp\Big(\frac{\bar \gamma\eps_0}{2n_1} \sum_{j\in J_1}\int_t^T |Z_s^j|^2{\rm d}s \Big)\Big]\\
\ \ \Dis \leq  \exp\Big(6\eps_0 \sum_{j\in J_1}\|Y^j\|_{\s^{\infty}_{[t,T]}}+3\eps_0 C_2+3\eps_0 \beta\int_t^T \|Y\|_{\s^{\infty}_{[s,T]}}{\rm d}s \Big)\\
\hspace{1cm} \Dis \times\,  \E_t\Big[\exp\Big(3\eps_0 \lambda \sum_{j\in J_1}\int_t^T |Z_s^j|^{1+\delta} {\rm d}s\Big)\Big].
\end{array}
\end{equation}
By Young's inequality, observe that for each pair of $a,b>0$,
\begin{equation}\label{eq:4.25c}
ab^{1+\delta}=\Big(\Big(\frac{1+\delta}{2}\Big)^{\frac{1+\delta}{1-\delta}} a^{\frac{2}{1-\delta}}\Big)^{\frac{1-\delta}{2}}
\Big(\frac{2}{1+\delta}b^2\Big)^{\frac{1+\delta}{2}}\leq b^2+\frac{1-\delta}{2}\Big(\frac{1+\delta}{2}\Big)^{\frac{1+\delta}{1-\delta}} a^{\frac{2}{1-\delta}}.
\end{equation}
Letting $a=12n_1\lambda/\bar\gamma$ and $b=|Z_s|$ in \eqref{eq:4.25c}, we have for each $j\in J_1$,
\begin{equation}\label{eq:4.26c}
3\eps_0\lambda |Z_s^j|^{1+\delta}=\frac{\bar\gamma\eps_0}{4n_1}
\Big(\frac{12n_1\lambda}{\bar\gamma} |Z_s^j|^{1+\delta}\Big)\leq \frac{\bar\gamma\eps_0}{4n_1} |Z_s^j|^2+ C_3,\ \ \ s\in\T,
\end{equation}
with
$$
C_3:=\frac{\bar\gamma\eps_0(1-\delta)}{8n_1}
\Big(\frac{1+\delta}{2}\Big)^{\frac{1+\delta}{1-\delta}} \Big(\frac{12n_1\lambda}{\bar\gamma}
\Big)^{\frac{2}{1-\delta}}.\vspace{0.2cm}
$$
Coming back to \eqref{eq:4.24c}, by \eqref{eq:4.26c} and H\"{o}lder's inequality we derive that for $t\in [T-h,T]$,
\begin{equation}\label{eq:4.27c}
\begin{array}{l}
\Dis \E_t\Big[\exp\Big(\frac{\bar \gamma\eps_0}{2n_1} \sum_{j\in J_1}\int_t^T |Z_s^j|^2{\rm d}s \Big)\Big]\\
\ \ \Dis \leq \exp\Big(12\eps_0 \sum_{j\in J_1}\|Y^j\|_{\s^{\infty}_{[t,T]}}+6\eps_0 C_2+ 2C_3 T+6\eps_0 \beta\int_t^T \|Y\|_{\s^{\infty}_{[s,T]}}{\rm d}s\Big).
\end{array}
\end{equation}
On the other hand, it follows from \eqref{eq:4.22c} and Jensen's inequality  that
\begin{equation}\label{eq:4.28c}
\begin{array}{lll}
\Dis \exp\Big(\gamma \sum_{j\in J_1}|Y_t^j|\Big)&\leq &\Dis \exp\left(n_1\gamma(C_1+C_2 )+n_1\gamma\beta \int_t^T \|Y\|_{\s^{\infty}_{[s,T]}}{\rm d}s \right)\\
&&\Dis \ \times \, \E_t\Big[\exp\Big(n_1\gamma\lambda \sum_{j\in J_1}\int_t^T |Z_s^j|^{1+\delta} {\rm d}s\Big)\Big],\ \ \ \ t\in [T-h,T].
\end{array}
\end{equation}
By letting $a=2 n^2_1\gamma\lambda/\bar\gamma\eps_0$ and $b=|Z_s|$ in \eqref{eq:4.25c}, we have for each $j\in J_1$,\vspace{0.1cm}
\begin{equation}\label{eq:4.29c}
n_1\gamma\lambda |Z_s^j|^{1+\delta}=\frac{\bar\gamma\eps_0}{2n_1}\left(\frac{2 n^2_1\gamma\lambda}{\bar\gamma\eps_0} |Z_s|^{1+\delta}\right)\leq \frac{\bar\gamma\eps_0}{2n_1} |Z_s^j|^2+ C_4,\ \ \ s\in\T,
\end{equation}
where
$$
C_4:=\frac{\bar\gamma\eps_0(1-\delta)}{4n_1}\left(\frac{1+\delta}{2}
\right)^{\frac{1+\delta}{1-\delta}} \left(\frac{2 n^2_1\lambda\gamma}{\bar\gamma\eps_0}
\right)^{\frac{2}{1-\delta}}.\vspace{0.2cm}
$$
Combining \eqref{eq:4.27c}-\eqref{eq:4.29c} yields that for each $t\in [T-h,T]$,\vspace{0.1cm}
$$
\begin{array}{lll}
\Dis \sum_{j\in J_1}|Y_t^j|& \leq &\Dis n_1(C_1+C_2 )+{C_4 T\over \gamma}+\frac{6\eps_0 C_2+ 2C_3 T}{\gamma} +\frac{12\eps_0}{\gamma}\sum_{j\in J_1}\|Y^j\|_{\s^{\infty}_{[t,T]}}\\
&&\Dis +\beta \Big(n_1+\frac{6\eps_0}{\gamma}\Big)\int_t^T \|Y\|_{\s^{\infty}_{[s,T]}}{\rm d}s.\vspace{0.1cm}
\end{array}
$$
And, it follows from the definition of $\eps_0$ that
\begin{equation}\label{eq:4.30c}
\sum_{j\in J_1}\|Y^j\|_{\s^{\infty}_{[t,T]}} \leq \Dis C_5+2\beta\left( n_1+\frac{1}{4}\right)\int_t^T \|Y\|_{\s^{\infty}_{[s,T]}}{\rm d}s,\ \ \ \ t\in [T-h,T],
\end{equation}
where
$$
C_5:=2n_1(C_1+C_2)+\frac{2C_4 T}{\gamma}+\frac{12\eps_0 C_2+ 4C_3 T}{\gamma}.\vspace{0.2cm}
$$
By taking square in both sides of \eqref{eq:4.30c} and using H\"{o}lder's inequality, we can conclude that for each $t\in [T-h,T]$,
\begin{equation}\label{eq:4.31c}
\sum_{j\in J_1}\|Y^j\|_{\s^{\infty}_{[t,T]}}^2\leq \Big(\sum_{j\in J_1}\|Y^j\|_{\s^{\infty}_{[t,T]}}\Big)^2 \leq \Dis 2C_5^2+8\beta^2T\left( n_1+1\right)^2\int_t^T \|Y\|_{\s^{\infty}_{[s,T]}}^2{\rm d}s.\vspace{0.2cm}
\end{equation}

(2) For $i\in J_2$, $g^i$ satisfies (ii) of Assumption \ref{A:C1b}. In this case, for each $t\in [T-h,T]$, we have
$$
 \Dis g^i\left(\omega,t,Y_t(\omega),Z_t(\omega)\right)\, {\rm sgn}(Y_t^i(\omega)) \leq \tilde\tilde\alpha_t(\omega)+\beta |Y_t(\omega)|+\bar v_t(\omega)|Z_t^i(\omega)|+\frac{\gamma}{2} |Z_t^i(\omega)|^2
$$
with
$$
\bar v_t(\omega):=v_t(\omega)+\phi(|Y_t(\omega)|)
+c\sum_{j=1}^{i-1}|Z_t^j(\omega)|\in {\rm BMO}_{[T-h,T]}(\R).
$$
Note that $(Y^i,Z^i)\in \s^{\infty}_{[T-h,T]}(\R)\times {\rm BMO}_{[T-h,T]}(\R^{1\times d})$ is a solution of BSDE \eqref{eq:2.2} on $[T-h,T]$. In view of the last inequality, by using It\^{o}-Tanaka's formula and Girsanov's transform, a similar argument to that from \eqref{eq:A.11} to \eqref{eq:A.12} yields that for each $j\in J_2$ and $t\in [T-h,h]$,
$$
|Y^i_t|\leq \|\xi^i\|_\infty+\|\tilde \tilde\alpha\|_{\lcal^\infty}+\beta \int_t^T \|Y\|_{\s^\infty_{[s,T]}} {\rm d}s,
$$
and then, in view of \eqref{eq:4.21},
$$
\sum_{j\in J_2}\|Y^j\|_{\s^\infty_{[t,T]}}\leq n_2(C_1+C_2) +n_2\beta \int_t^T \|Y\|_{\s^\infty_{[s,T]}} {\rm d}s.
$$
By taking square in both sides of the last inequality and using H\"{o}lder's inequality, we can conclude that for each $t\in [T-h,T]$,
\begin{equation}\label{eq:4.32c}
\sum_{j\in J_2}\|Y^j\|_{\s^{\infty}_{[t,T]}}^2\leq \Big(\sum_{j\in J_2}\|Y^j\|_{\s^{\infty}_{[t,T]}}\Big)^2 \leq \Dis 2n_2^2(C_1+C_2)^2+2T n_2^2\beta^2\int_t^T \|Y\|_{\s^{\infty}_{[s,T]}}^2{\rm d}s.\vspace{0.3cm}
\end{equation}

(3) For $i\in J_3$, $g^i$ satisfies (iii) of Assumption \ref{A:C1b}. In this case, for each $t\in [T-h,T]$, we have
$$
 \Dis g^i\left(\omega,t,Y_t(\omega),Z_t(\omega)\right)\, {\rm sgn}(Y_t^i(\omega)) \leq \tilde\alpha_t(\omega)+\beta |Y_t(\omega)|+\lambda \sum_{j\in J_3}|Z_t^j(\omega)|,
$$
and then, in view of inequality $2ab\leq 2\eps a^2+\frac{1}{2\eps}b^2$ for each $a,b\geq 0$ and $\eps>0$,
$$
\begin{array}{l}
\Dis 2\sum_{j\in J_3}Y_t^j(\omega)g^j\left(\omega,t,Y_t(\omega),Z_t(\omega)\right)\\
\ \ \Dis \leq 2\sum_{j\in J_3} |Y_t^j(\omega)|\Big(
\tilde\alpha_t(\omega)+\beta |Y_t(\omega)|+\lambda \sum_{j\in J_3}|Z_t^j(\omega)|\Big)\\
\ \ \Dis \leq \sum_{j\in J_3}\Big[2\tilde\alpha_t(\omega)|Y_t^j(\omega)|
+2\beta |Y_t(\omega)|^2+\sum_{j\in J_3}\Big(2n_3\lambda^2|Y_t^j(\omega)|^2
+\frac{1}{2n_3}|Z_t^j(\omega)|^2\Big)\Big]\\
\ \ \Dis \leq 2\tilde\alpha_t(\omega)\sum_{j\in J_3}|Y_t^j(\omega)|+2n_3\left(\beta+n_3\lambda^2\right)
|Y_t(\omega)|^2+\frac{1}{2}\sum_{j\in J_3}|Z_t^j(\omega)|^2.
\end{array}
$$
Note that for each $i\in J_3$, $(Y^i,Z^i)\in \s^{\infty}_{[T-h,T]}(\R)\times {\rm BMO}_{[T-h,T]}(\R^{1\times d})$ is an adapted solution of BSDE \eqref{eq:2.2} on $[T-h,T]$. In view of the last inequality together with \eqref{eq:4.21c}, applying It\^{o}'s formula to $\sum_{j\in J_3}|Y^j_s|^2$ yields that for each $t\in [T-h,h]$,
$$
\begin{array}{l}
\Dis \sum_{j\in J_3}|Y^j_t|^2+\frac{1}{2}\E_t\Big[\int_t^T \sum_{j\in J_3}|Z_s^j(\omega)|^2 {\rm d}s\Big]\\
\Dis \ \ \leq \E_t\Big[\sum_{j\in J_3}|\xi^j|^2\Big]+\E_t\Big[\int_t^T\Big(2\tilde\alpha_s\sum_{j\in J_3}|Y_s^j|+2n_3\Big(\beta+n_3\lambda^2\Big)
|Y_s|^2\Big){\rm d}s\Big]\\
\Dis \ \ \leq C_1^2+2C_2\sum_{j\in J_3}\|Y^j\|_{\s^\infty_{[t,T]}}+2n_3\left(\beta+n_3\lambda^2\right)
\int_t^T \|Y\|^2_{\s^\infty_{[s,T]}}{\rm d}s,
\end{array}
$$
and then, in view of
$$
2C_2\sum_{j\in J_3}\|Y^j\|_{\s^\infty_{[t,T]}}\leq 2n_3C_2^2+{1\over 2n_3}\Big(\sum_{j\in J_3}\|Y^j\|_{\s^\infty_{[t,T]}}\Big)^2\leq 2n_3C_2^2+{1\over 2}\sum_{j\in J_3}\|Y^j\|_{\s^\infty_{[t,T]}}^2,
$$
we have
\begin{equation}\label{eq:4.33c}
\sum_{j\in J_3}\|Y^j\|_{\s^\infty_{[t,T]}}^2\leq 2C_1^2+4n_3C_2^2 +4n_3\left(\beta+n_3\lambda^2\right)
\int_t^T \|Y\|^2_{\s^\infty_{[s,T]}}{\rm d}s.\vspace{0.1cm}
\end{equation}

Finally, adding \eqref{eq:4.31c}, \eqref{eq:4.32c} and \eqref{eq:4.33c} up together yields that
$$
\|Y\|^2_{\s^\infty_{[t,T]}}\leq C_6+C_7\int_t^T \|Y\|^2_{\s^\infty_{[s,T]}}{\rm d}s,\ \ t\in [T-h,T]
$$
with $C_6:=2C_5^2+2n_2^2(C_1+C_2)^2+2C_1^2+4n_3C_2^2$ and
$$
C_7:=8\beta^2T\left( n_1+1\right)^2+2T n_2^2\beta^2+4n_3\left(\beta+n_3\lambda^2\right).\vspace{-0.1cm}
$$
It then follows from Gronwall's inequality that
$$
\|Y\|^2_{\s^\infty_{[t,T]}}\leq C_6\exp(C_7(T-t))\leq C_6\exp(C_7T)=: \widetilde K, \ \ t\in [T-h,T],
$$
which yields the desired conclusion. The proof  is  complete.\vspace{0.2cm}
\end{proof}

\begin{proof}[Proof of \cref{thm:2.8a}]
With \cref{thm:2.3} and \cref{pro:4.2c} in hands, we can closely follow the proof of Theorem 4.1 in \citet{cheridito2015Stochastics} to prove our \cref{thm:2.5}. All the details are omitted here.\vspace{0.2cm}
\end{proof}

\section{Global unbounded solution: proof of Theorems~\ref{thm:2.22e} and \ref{thm:2.9}}
\label{sec:5-global unbounded solution}
\setcounter{equation}{0}

\subsection{Proof of \cref{thm:2.22e}\vspace{0.2cm}}

\begin{proof}[Proof of \cref{thm:2.22e}]
Define
\begin{equation}\label{eq:4.21e}
\bar g(t,y,z):=g\Big(t,\ y+\int_0^t H_s {\rm d}B_s,\ z+H_t\Big),\ \ (\omega,t,y,z)\in \Omega\times\T\times \R^n\times \R^{n\times d}.
\end{equation}
Consider the following multi-dimensional BSDE
\begin{equation}\label{eq:4.21}
\bar Y_t=\bar \xi+\int_t^T \bar g(s,\bar Y_s,\bar Z_s){\rm d}s-\int_t^T \bar Z_s {\rm d}B_s, \ \ t\in\T,
\end{equation}
or, equivalently,
\begin{equation}\label{eq:4.22}
\bar Y_t^i=\bar\xi^i+\int_t^T \bar g^i(s,\bar Y_s,\bar Z_s) {\rm d}s-\int_t^T \bar Z_s^i {\rm d}B_s, \ \ t\in\T,\ \ i=1,\cdots,n.
\end{equation}
It is not difficult to check that BSDE \eqref{eq:4.21} admits a unique global solution $(\bar Y,\bar Z)\in \s^\infty(\R^n)\times {\rm BMO}(\R^{n\times d})$ if and only if BSDE \eqref{eq:2.1} admits a unique global solution $(Y,Z):=(\bar Y+\int_0^\cdot H_s {\rm d}B_s,\ \bar Z+H)$ on the time interval $\T$ such that
$$
\Big(Y-\int_0^\cdot H_s {\rm d}B_s,\,  Z\Big)\in \s^\infty(\R^n)\times {\rm BMO}(\R^{n\times d}).
$$
Consequently, in view of \cref{cor:2.6e}, for completing the proof of \cref{thm:2.22e} it suffices to prove that the generator $\bar g$ defined in \eqref{eq:4.21e} also satisfies Assumptions \ref{A:B1}, \ref{A:D2} and \ref{A:AB} with $\theta=0$. In view of assumptions of the generator $g$ and parameters $(\alpha,\bar\alpha,\tilde\alpha, H)$, it is straightforward to verify the above assertion. The proof is then complete.
\end{proof}

\subsection{Proof of \cref{thm:2.9}\vspace{0.2cm}}

\begin{proof}[Proof of \cref{thm:2.9}]
Identical to  the proof of \cref{thm:2.22e}, we see that in view of \cref{thm:2.5}, for completing the proof of \cref{thm:2.9} it suffices to prove that the generator $\bar g$ defined in \eqref{eq:4.21e} also satisfies Assumptions \ref{A:D1} and \ref{A:D2} with $\beta=0$ and some other appropriate parameters when the constant $\theta$ is smaller than a given constant $\bar \theta_0$ depending only on $p$, $\bar p$ and $\gamma$. Clearly, $\bar g$ satisfies \ref{A:D2} since $g$ satisfies it. \vspace{0.2cm}

In the sequel, we will prove that the generator $\bar g$ also satisfies Assumption \ref{A:D1} with $\beta=0$ and some other appropriate parameters under the given conditions. Assume that $\alpha\in \ecal^{\infty}(p\gamma)$ for some real $p>1$, and $|H|^2\in \ecal^{\infty}(2\bar p (q\gamma)^2)$ for some $\bar p>1$ with $q=p/(p-1)$ such that $1/p+1/q=1$. Define the following constants:\vspace{0.2cm}
$$
\eps:={p-1\over p+1},\ \ \ \hat p:={(p+3)\bar p+p-1\over 2(p+2\bar p-1)}>1,\ \ \ \tilde p:={p+1\over 2\hat p}>1\ \ {\rm and}\ \ \tilde q:={p+1\over p-2\hat p+1}>1.\vspace{0.2cm}
$$
It is clear that $1/\tilde p+1/\tilde q=1$. Note that $g$ satisfies Assumption \ref{A:D1} with $\beta=0$. By a similar argument to (iv) of \cref{rmk:2.2}, in Assumption \ref{A:D1} we can without loss of generality assume that $f=g^i$ satisfies either of conditions \ref{A:D1}(i), \ref{A:D1}(ii) and \ref{A:D1}(iii) with $\beta=0$ for all $i=1,\cdots,n$. Thus, for $i=1,\cdots,n$, we need to consider the following three cases.\vspace{0.2cm}

(1) $g^i$ satisfies \ref{A:D1}(i) with $\beta=0$. For this case, it follows from the definition of $\bar g$ that $\as$, for each $(y,z)\in\R^n\times \R^{n\times d}$, we have
$$
\Dis \bar g^i(\omega,t,y,z) \leq  \Dis \alpha_t(\omega)+\sum_{j\neq i} \left(\lambda |z^j+H_t^j(\omega)|+\theta |z^j+H_t^j(\omega)|^2\right)+\frac{\gamma}{2} |z^i+H_t^i(\omega)|^2
$$
and
$$
\begin{array}{l}
\Dis \bar g^i(\omega,t,y,z)\geq \Dis \frac{\bar \gamma}{2} |z^i+H_t^i(\omega)|^2 -\bar\alpha_t(\omega)-\sum_{j=i+1}^{n} \left(\bar\lambda |z^j+H_t^j(\omega)|^{1+\delta}+\theta |z^j+H_t^j(\omega)|^2\right)\\
\hspace{2.4cm} \Dis -\bar c\sum_{j=1}^{i-1}|z^j+H_t^j(\omega)|^2.
\end{array}
$$
Note that for each $a,b\geq 0$, it holds that\vspace{0.1cm}
$$
\begin{array}{l}
\Dis (a+b)^2\leq 2a^2+2b^2,\ \ (a+b)^{1+\delta}\leq 2a^{1+\delta}+2b^{1+\delta},\vspace{0.1cm}\\
\Dis (a+b)^2\leq (1+\eps)a^2+\left(1+{1\over \eps}\right)b^2\ \ {\rm and}\ \ (a+b)^2\geq {1\over 2}a^2-b^2.
\end{array}
$$
We know that $\as$, for each $(y,z)\in\R^n\times \R^{n\times d}$,
\begin{equation}\label{eq:4.23}
\begin{array}{l}
\Dis \frac{\bar \gamma}{4} |z^i|^2 -\hat\alpha_t(\omega)-\sum_{j=i+1}^{n} \left(2\bar\lambda |z^j|^{1+\delta}+2\theta |z^j|^2\right)-2\bar c\sum_{j=1}^{i-1}|z^j|^2\leq \Dis \bar g^i(\omega,t,y,z)\\
\quad\quad\quad \leq  \Dis \breve\alpha_t(\omega)+\sum_{j\neq i} \left(\lambda |z^j|+2\theta |z^j|^2\right)+\frac{\gamma(1+\eps)}{2} |z^i|^2
\end{array}
\vspace{0.1cm}
\end{equation}
with
$$
\breve\alpha_t(\omega):=\alpha_t(\omega)+\sum_{j\neq i} \left(\lambda |H_t^j(\omega)|+2\theta |H_t^j(\omega)|^2\right)+\frac{\gamma(1+\eps)}{2\eps} |H_t^i(\omega)|^2
$$
and
$$
\hat\alpha_t(\omega):= \bar\alpha_t(\omega)+2\bar c\sum_{j=1}^{i-1}|H_t^j(\omega)|^2+\frac{\bar \gamma}{2} |H_t^i(\omega)|^2+\sum_{j=i+1}^{n} \left(2\bar\lambda |H_t^j(\omega)|^{1+\delta}+2\theta |H_t^j(\omega)|^2\right).
$$
By virtue of  $\bar\alpha\in \mcal^\infty$ and $H\in {\rm BMO}(\R^{n\times d})$ together with the fact that
$$
\hat\alpha\leq 2n\bar\lambda+\bar\alpha+\left(2\bar c+{\bar\gamma\over 2}+2\bar\lambda+2\theta\right) |H|^2,\vspace{-0.2cm}
$$
we can deduce that\vspace{-0.1cm}
\begin{equation}\label{eq:4.24}
\hat\alpha \in \mcal^\infty.
\end{equation}
On the other hand, from the definitions of constants $q, \hat p, \tilde p, \tilde q,\eps$ and $\breve\alpha$ it is not difficult to verify that $p\gamma=\tilde p \hat p\gamma (1+\eps)$ and that there exists two very small positive constants $\bar\eps$ and $\bar\theta_0$ such that
$$
\breve\alpha\leq \alpha+{n\lambda^2\over 4\bar\eps}+(\bar \eps+2\theta+q\gamma)|H|^2\quad {\rm and} \quad (\bar\eps+2\bar\theta_0+q\gamma)\tilde q \hat p\gamma (1+\eps)=2\bar p(q\gamma)^2.
$$
Then, it follows from the integrability condition of $\alpha$ and $|H|^2$ that $\alpha\in \ecal^{\infty}(\tilde p \hat p\gamma (1+\eps))$ and $(\bar \eps+2\theta+q\gamma) |H|^2\in \ecal^{\infty}(\tilde q \hat p\gamma (1+\eps))$ for each $\theta\in [0,\bar\theta_0]$. Note that $1/\tilde p+1/\tilde q=1$. By (iii) of \cref{rmk:2.1} we know that for each $\theta\in [0,\bar\theta_0]$,
\begin{equation}\label{eq:4.25}
\breve\alpha\in \ecal^{\infty}(\hat p\gamma (1+\eps)).
\end{equation}
Combining \eqref{eq:4.23}, \eqref{eq:4.24} and \eqref{eq:4.25} yields that the generator $\bar g^i$ satisfies (i) of Assumption \ref{A:D1} with parameters $(\breve\alpha, \hat\alpha,0,\gamma(1+\eps),\bar\gamma/2,2c,\lambda,2\bar\lambda,2\theta)$ instead of $(\alpha, \bar\alpha,\beta,\gamma,\bar\gamma,c,\lambda,\bar\lambda,\theta)$.\vspace{0.2cm}

(2) $g^i$ satisfies \ref{A:D1}(ii) with $\beta=0$ and violates \ref{A:D1} (i) and (iii). Note that we have $H^i\equiv 0$ in this case. It follows from the definition of $\bar g$ that $\as$, for each $(y,z)\in\R^n\times \R^{n\times d}$, we have\vspace{0.1cm}
$$
|\bar g^i(\omega,t,y,z)|\leq \tilde\alpha_t(\omega)+|z^i|\Big(\bar v_t(\omega)+c\sum_{j=1}^{i-1}|z^j|\Big)+\frac{\gamma}{2} |z^i|^2
$$
with
$$
\bar v_t(\omega):=v_t(\omega)+c\sum_{j=1}^{i-1}|H_t^j(\omega)|
\leq v_t(\omega)+\sqrt{n}c|H_t(\omega)|.
$$
And, it follows from $v\in {\rm BMO}(\R)$ and $H\in {\rm BMO}(\R^{n\times d})$ that $\bar v\in {\rm BMO}(\R)$, which means that $\bar g^i$ satisfies (ii) of Assumption \ref{A:D1} with
parameters $(\tilde\alpha,0,\bar v,c,\gamma)$ instead of $(\tilde\alpha,\beta, v,c,\gamma)$.\vspace{0.2cm}

(3) $g^i$ satisfies \ref{A:D1}(iii) with $\beta=0$. In this case, it follows from the definition of $\bar g$ that $\as$, for each $(y,z)\in\R^n\times \R^{n\times d}$, we have\vspace{-0.1cm}
$$
|g^i(\omega,t,y,z)|\leq \check\alpha_t(\omega)+\bar\lambda |z|+2\theta\sum_{j\neq i}|z^j|^2\vspace{-0.1cm}
$$
with\vspace{-0.2cm}
$$
\check\alpha_t(\omega):=\bar\alpha_t(\omega)+\bar\lambda |H_t(\omega)|+2\theta\sum_{j\neq i} |H_t^j(\omega)|^2\leq \bar\alpha_t(\omega)+\bar\lambda |H_t(\omega)|+2\theta |H_t(\omega)|^2.
$$
And, it follows from $\bar\alpha\in \mcal^\infty$ and $H\in {\rm BMO}(\R^{n\times d})$ that $\check\alpha\in \mcal^\infty$, which means that $\bar g^i$ satisfies (iii) of Assumption \ref{A:D1} with
parameters $(\check\alpha,0,\bar\lambda,2\theta)$ instead of $(\bar\alpha,\beta,\bar\lambda,\theta)$.

All in all, we have proved the desired conclusion. \cref{thm:2.9} is then proved.
\end{proof}

\appendix
\section{Auxiliary results on bounded solutions of one-dimensional BSDEs with unbounded coefficients}
\renewcommand{\appendixname}{}

We collect here some general results concerning the uniform estimate of (bounded) solution to scalar-valued BSDEs. We give some brief proofs for completeness. Let the real-valued function $f(\omega,t,y,z): \Omega\times\T\times\R\times\R^{1\times d}\To \R$ be $(\F_t)$-progressively measurable for each $(y,z)\in \R\times\R^{1\times d}$, and consider the following one-dimensional BSDE:
\begin{equation}\label{eq:A.1}
Y_t=\eta+\int_t^T  f(s,Y_s,Z_s){\rm d}s-\int_t^T Z_s {\rm d}B_s, \ \ t\in\T,
\end{equation}
where $\eta\in L^{\infty}(\R)$, and the solution $(Y,Z)$ is a pair of $(\F_t)$-progressively measurable processes with values in $\R\times\R^{1\times d}$ such that $Y\in \s^{\infty}(\R)$.\vspace{0.2cm}

Assume that $\bar\beta,\bar\delta\geq 0$ are two given constants, $\varphi(\cdot):[0,+\infty)\To [0,+\infty)$ is a nondecreasing continuous function, and $
\check\alpha\in\ecal^{\infty}(\gamma\exp(\bar\beta T)), \ \dot\alpha\in\mcal^{\infty}, \ \ddot\alpha\in\lcal^{\infty},\ \bar u\in {\rm BMO}(\R), \ \bar v\in {\rm BMO}(\R)$ are five real-valued non-negative $(\F_t)$-progressively measurable processes.\vspace{0.2cm}

We introduce the following assumptions on the generator $f$:
\begin{enumerate}
\renewcommand{\theenumi}{(A\arabic{enumi})}
\renewcommand{\labelenumi}{\theenumi}
\item\label{A:A1} Almost everywhere in $\Omega\times\T$, for any $(y,z)\in \R\times\R^{1\times d}$, we have
$$
\Dis \frac{\bar \gamma}{2} |z|^2-\dot\alpha_t(\omega)-\bar\beta |y|\leq f(\omega,t,y,z) \leq  \Dis \check\alpha_t(\omega)+\left [\bar\beta |y|\,  {\bf 1}_{y>0}+\varphi(|y|)\, {\bf 1}_{y<0}\right ]+\frac{\gamma}{2} |z|^2.
$$

\item\label{A:A2} Almost everywhere in $\Omega\times\T$, for any $(y,z)\in \R\times\R^{1\times d}$, we have
$$
\Dis \frac{\bar \gamma}{2} |z|^2-\dot\alpha_t(\omega)-\bar\beta |y|\leq -f(\omega,t,-y,-z) \leq  \Dis \check\alpha_t(\omega)+\left [\bar\beta |y|\,  {\bf 1}_{y>0}+\varphi(|y|)\, {\bf 1}_{y<0}\right ]+\frac{\gamma}{2} |z|^2.
$$

\item\label{A:A3} Almost everywhere in $\Omega\times\T$, for any $(y,z)\in \R\times\R^{1\times d}$, we have
$$
\begin{array}{l}
\Dis -\left[\bar\beta |y|\,  {\bf 1}_{y<0}+\varphi(|y|)\, {\bf 1}_{y>0}\right]-h(\omega,t,z)\leq f(\omega,t,y,z)\\
\hspace{6.5cm}\Dis \leq \left[\bar\beta |y|\,  {\bf 1}_{y>0}+\varphi(|y|)\, {\bf 1}_{y<0}\right]+h(\omega,t,z)
\end{array}
$$
with
$$
\Dis h(\omega,t,z):=\ddot\alpha_t(\omega)+\bar u_t(\omega) |z|+\frac{\gamma}{2} |z|^2.
$$

\item\label{A:A4} Almost everywhere in $\Omega\times\T$, for any $(y,z)\in \R\times\R^{1\times d}$, we have
$$
\begin{array}{l}
\Dis -\dot\alpha_t(\omega)-\left[\bar\beta |y|\,  {\bf 1}_{y<0}+\varphi(|y|)\, {\bf 1}_{y>0}\right]-\bar\lambda |z|\leq f(\omega,t,y,z)\\
\hspace{7.0cm}\Dis \leq \dot\alpha_t(\omega)+\left[\bar\beta |y|\,  {\bf 1}_{y>0}+\varphi(|y|)\, {\bf 1}_{y<0}\right]+\bar\lambda |z|.
\end{array}
$$

\item\label{A:A5} Almost everywhere in $\Omega\times\T$, for any $(y,\bar y,z,\bar z)\in \R\times\R\times\R^{1\times d}\times\R^{1\times d}$, we have for some $k>0$,
$$
|f(\omega,t,y,z)-f(\omega,t,\bar y,\bar z)|
\leq \bar\beta |y-\bar y|+k\left(\bar v_t(\omega)+|z|+|\bar z|\right)|z-\bar z|.
$$
\end{enumerate}

\begin{pro}\label{pro:A.1}
Assume that the generator $f$ satisfies Assumption \ref{A:A1} (resp. \ref{A:A2}).
\begin{itemize}
\item [(i)] For any solution $(Y,Z)$ of BSDE \eqref{eq:A.1} such that $Y\in \s^{\infty}(\R)$, we have $Z\in {\rm BMO}(\R^{1\times d})$ and for each $t\in \T$,
\begin{equation}\label{eq:A.2}
\begin{array}{l}
\Dis \|Y\|_{\s^\infty_{[t,T]}}+\left\|Z\right\|_{{\rm BMO}_{[t,T]}}^2\vspace{0.3cm}\\
\ \ \leq \Dis {2(1+\bar\beta T)+\bar\gamma\over \bar\gamma}\exp(2\bar\beta T) \left(3\|\eta\|_{\infty}+\|\check\alpha\|_{\ecal^{\infty}_{[t,T]}(\gamma\exp(\bar\beta T))}+2\left\|\dot\alpha\right\|_{\mcal^{\infty}_{[t,T]}}\right).\vspace{0.1cm}
\end{array}
\end{equation}
And, if the generator $f$ only satisfies the second inequality for the case of $y>0$ and the first inequality in \ref{A:A1} (resp. \ref{A:A2}), the above conclusion \eqref{eq:A.2} still holds.

\item [(ii)] BSDE \eqref{eq:A.1} admits a minimal (resp. maximal) solution $(Y,Z)$ such that $Y\in \s^{\infty}(\R)$ in the sense that for any solution $(\bar Y,\bar Z)$ of BSDE \eqref{eq:A.1} such that $\bar Y\in \s^{\infty}(\R)$,  we have for each $t\in \T$, $\ps$, $Y_t\leq \bar Y_t$ (resp. $Y_t\geq \bar Y_t$). Moreover, $Z\in {\rm BMO}(\R^{1\times d})$.

\item [(iii)] If the generator $f$ further satisfies Assumption \ref{A:A5}, then BSDE \eqref{eq:A.1} admits a unique solution $(Y,Z)$ such that $Y\in \s^{\infty}(\R)$. Moreover, $Z\in {\rm BMO}(\R^{1\times d})$.\vspace{0.2cm}
\end{itemize}
\end{pro}

\begin{proof} We only give the proof when the generator $f$ satisfies Assumption \ref{A:A1}. The other case can be proved in the same way.\vspace{0.2cm}

(i) Let $(Y,Z)$ be a solution of BSDE \eqref{eq:A.1} such that $Y\in \s^{\infty}(\R)$.  For each integer $m\geq 1$ and each stopping time $\tau\in \mathcal{T}_{[0,T]}$, define the following stopping time\vspace{0.1cm}
$$
\sigma_m^\tau:=T\wedge \,  \inf\left\{s\in [\tau,T]:\ \int_\tau^s |Z_s|^2{\rm d}s\geq m\right\}
$$
with convention $\inf\emptyset=\infty$. It follows from the first inequality in Assumption \ref{A:A1} that for each $m\geq 1$ and each $t\in\T$,\vspace{0.2cm}
$$
{\bar\gamma\over 2}\E_\tau\left[\int_\tau^{\sigma_m^\tau} |Z_s|^2 {\rm d}s\right]\leq |Y_\tau|+ \E_\tau\left[|Y_{\sigma_m^\tau}| +\int_\tau^{\sigma_m^\tau}\left(\dot\alpha_s+ \bar\beta |Y_s|\right) {\rm d}s\right],\ \ \tau\in \mathcal{T}_{[t,T]}.\vspace{0.1cm}
$$
Sending $m\To +\infty$ in previous inequality and using Fatou's lemma yields that for each $t\in\T$,
$$
{\bar\gamma\over 2}\E_\tau\left[\int_\tau^T |Z_s|^2 {\rm d}s\right]\leq \|\eta\|_{\infty}+\left\|\dot\alpha\right\|_{\mcal^{\infty}_{[t,T]}}+(1+\bar\beta T)\|Y\|_{\s^\infty_{[t,T]}},\ \ \tau\in \mathcal{T}_{[t,T]},
$$
which means that $Z\in {\rm BMO}(\R^{1\times d})$, and for each $t\in\T$,
\begin{equation}\label{eq:A.3}
{\bar\gamma\over 2}\|Z\|_{{\rm BMO}_{[t,T]}}^2 \leq \|\eta\|_{\infty}+\left\|\dot\alpha\right\|_{\mcal^{\infty}_{[t,T]}}+(1+\bar\beta T)\|Y\|_{\s^\infty_{[t,T]}}.
\end{equation}
Furthermore, define the function
$$
u(t,x):=\exp\left(\gamma\exp(\bar\beta t)x+\gamma\int_0^t\exp(\bar\beta s)\check\alpha_s{\rm d}s\right),\ \ (t,x)\in \T\times\R.
$$
In view of the second inequality for the case of $y>0$ in Assumption \ref{A:A1}, by applying It\^o-Tanaka's formula to $u(s,Y^+_s)$ we can deduce that for each $m\geq 1$ and $t\in \T$,
$$
\begin{array}{lll}
\Dis \exp(\gamma Y^+_t) &\leq & \Dis \E_t\left[\exp\left(\gamma\exp(\bar\beta T) \eta^++\gamma\exp(\bar\beta T) \int_t^T \check\alpha_s {\rm d}s\right)\right],\vspace{0.2cm}\\
&\leq & \Dis \exp\left(\gamma\exp(\bar\beta T)\|\eta\|_{\infty}\right)\exp\left(\gamma\exp(\bar\beta T)\|\check\alpha\|_{\ecal^{\infty}_{[t,T]}(\gamma\exp(\bar\beta T))}\right)
\end{array}
$$
and then,
\begin{equation}\label{eq:A.4}
Y^+_t\leq \exp(\bar\beta T)\left(\|\eta\|_{\infty}+\|\check\alpha\|_{\ecal^{\infty}_{[t,T]}(\gamma\exp(\bar\beta T))}\right).
\end{equation}
On the other hand, from the first inequality in Assumption \ref{A:A1} we can also get that for each $m\geq 1$ and $t\in\T$,
$$
Y^-_t=(-Y_t)^+ \leq \E_t\left[\eta^-+\int_t^T \left(\dot\alpha_s+ \bar\beta |Y_s|\right) {\rm d}s\right],
$$
which together with \eqref{eq:A.4} yields that for each $t\in \T$,
$$
|Y_t|\leq \exp(\bar\beta T)\left(\|\eta\|_{\infty}+\|\check\alpha\|_{\ecal^{\infty}_{[t,T]}(\gamma\exp(\bar\beta T))}\right)+\|\eta\|_{\infty}+\left\|\dot\alpha\right\|_{\mcal^{\infty}_{[t,T]}}
+\bar\beta \E_t\left[\int_t^T |Y_s| {\rm d}s\right].
$$
And, it follows from Gronwall's inequality that
\begin{equation}\label{eq:A.5}
\|Y\|_{\s^\infty_{[t,T]}}\leq \exp(2\bar\beta T)\left(2\|\eta\|_{\infty}+\|\check\alpha\|_{\ecal^{\infty}_{[t,T]}
(\gamma\exp(\bar\beta T))}+\left\|\dot\alpha\right\|_{\mcal^{\infty}_{[t,T]}}\right), \ \ t\in\T.
\end{equation}
Finally, the desired conclusion \eqref{eq:A.2} follows from \eqref{eq:A.5} and \eqref{eq:A.3} immediately. \vspace{0.2cm}

(ii) In view of Assumption \ref{A:A1}, it is easy to verify that for each integer $m\geq 1$, the following function
$$
\begin{array}{c}
f^m(\omega,t,y,z):=\inf\left\{f(\omega,t,\bar y,\bar z)+(m+\bar\beta)|y-\bar y|+m|z-\bar z|:\ (\bar y,\bar z)\in \R\times\R^{1\times d}\right\},\vspace{0.1cm}\\
(\omega,t,y,z)\in \Omega\times\T\times \R\times\R^{1\times d}
\end{array}
$$
is well defined and an $(\F_t)$-progressively measurable process for each $(y,z)$. It is also not difficult to prove that $f^m$ is uniformly Lipschitz continuous in the state variables $(y,z)$ and also satisfies Assumption \ref{A:A1} with the same parameters for each $m\geq 1$, and that the sequence $\{f^m\}_{m=1}^{\infty}$ converges increasingly uniformly on compact sets to the generator $f$ as $m$ tends to $+\infty$. Then, $\as$, for each $(y,z)\in \R\times\R^{1\times d}$ and $m\geq 1$, we have
\begin{equation}\label{eq:A.6}
|f^m(\omega,t,y,z)|\leq \check\alpha_t(\omega)+\dot\alpha_t(\omega)+\varphi(|y|)+\bar\beta |y| +{\gamma\over 2}|z|^2,
\end{equation}
and then
\begin{equation}\label{eq:A.7}
|f^m(\omega,t,0,0)|\leq \check\alpha_t(\omega)+\dot\alpha_t(\omega)+\varphi(0),
\end{equation}
which means that
$$
\E\left[\left(\int_0^T |f^m(s,0,0)|{\rm d}s\right)^2\right]<+\infty.\vspace{0.1cm}
$$
Consequently, by the classical results (see for example Theorems 3 and 2 in \citet{FanJiangTian2011SPA}) we know that for each $m\geq 1$, the following BSDE
$$
Y_t^m=\eta+\int_t^T  f^m(s,Y_s^m,Z_s^m){\rm d}s-\int_t^T Z_s^m {\rm d}B_s, \ \ t\in\T
$$
admits a unique solution $(Y^m,Z^m)\in \s^2(\R)\times \hcal^2(\R^{1\times d})$, and $Y^m$ converges increasing pointwisely to a process $Y$. Moreover, by the classical a priori estimate on the $L^2$ solution (see for example Proposition 3.2 in \citet{Bri03}) we know that there exists a uniform constant $c_0>0$ such that for each $m\geq 1$ and $t\in \T$,
$$
\E_t\left[\sup_{s\in [t,T]}|Y^m_s|^2\right]\leq c_0 \exp\left(2(m+\bar\beta)T+2m^2 T\right)\E_t\left[|\eta|^2+\left(\int_t^T |f^m(s,0,0)|{\rm d}s\right)^2\right],
$$
which together with \eqref{eq:A.7} and the facts that $\eta\in L^{\infty}(\R)$, $\check\alpha\in \ecal^\infty(\gamma\exp(\bar\beta T))$ and $\dot\alpha\in \mcal^\infty$ yields that $Y^m\in \s^{\infty}(\R)$ for each $m\geq 1$, and then $Z^m\in {\rm BMO}(\R^{1\times d})$ by (i).\vspace{0.2cm}

In the sequel, it follows from (i) that there exists a uniform constant $A>0$ which is independent of $m$ such that $\as$, we have $\sup_{m\geq 1} |Y^m_t(\omega)|\leq A$ and, in view of \eqref{eq:A.6},
$$
\RE\ (y,z)\in [-A,A]\times \R^{1\times d},\ \ |f^m(\omega,t,y,z)|\leq \check\alpha_t(\omega)+\dot\alpha_t(\omega)+\varphi(A)+\bar\beta A +{\gamma\over 2}|z|^2.
$$
Thus, we can apply the monotonic stability result Proposition 3.1 in \citet{LuoFan2018SD} to obtain the existence of a process $Z\in \hcal^2(\R^{1\times d})$ such that $Y\in \s^\infty(\R)$ and $(Y,Z)$ is a solution of BSDE \eqref{eq:A.1}. And, it follows from (i) that $Z\in {\rm BMO}(\R^{1\times d})$.\vspace{0.2cm}

It remains to show that $(Y,Z)$ is the minimal solution. For this, let $(\bar Y,\bar Z)$ be any solution of BSDE \eqref{eq:A.1} such that $\bar Y\in \s^{\infty}(\R)$. By (i) again we know that $\bar Z\in {\rm BMO}(\R^{1\times d})$. This means that $(\bar Y,\bar Z)\in \s^2(\R)\times\hcal^2(\R^{1\times d})$. Then, since $f^m$ is uniformly Lipschitz continuous in $(y,z)$ and $f^m\leq f$ for each $m\geq 1$, it follows from the classical comparison theorem on the $L^2$-solution that for each $m\geq 1$ and $t\in \T$, $\ps$, $Y^m_t\leq \bar Y_t$, and letting $m\To \infty$ yields that $Y\leq \bar Y$, which is the desired conclusion.\vspace{0.2cm}

(iii) Let the generator $f$ further satisfy Assumption \ref{A:A5}, and $(Y,Z)$ and
$(\bar Y,\bar Z)$ be the solution of BSDE \eqref{eq:A.1} such that $Y\in \s^{\infty}(\R)$ and $\bar Y\in \s^{\infty}(\R)$. First of all, from (i) we know that $Z\in {\rm BMO}(\R^{1\times d})$ and $\bar Z\in {\rm BMO}(\R^{1\times d})$. Furthermore, define $\hat Y:=Y-\bar Y$ and $\hat Z:=Z-\bar Z$. By virtue of It\^o-Tanaka's formula and \ref{A:A5} we can deduce that for each $t\in\T$,
\begin{equation}\label{eq:A.8}
\hspace*{-0.4cm}
\begin{array}{lll}
|\hat Y_t| &\leq & \Dis \int_t^T \left(\bar\beta |\hat Y_s| +k(\bar v_s+|Z_s|+|\bar Z_s|)|\hat Z_s|\right){\rm d}s-\int_t^T {\rm sgn}(\hat Y_s)\hat Z_s {\rm d}B_s\vspace{0.2cm}\\
&\leq & \Dis \bar\beta \int_t^T |\hat Y_s| {\rm d}s-\int_t^T {\rm sgn}(\hat Y_s)\hat Z_s \Big[{\rm d}B_s-k{\rm sgn}(\hat Y_s)(\bar v_s+|Z_s|+|\bar Z_s|)\frac{\hat Z^\top _s }{|\hat Z_s|}{\bf 1}_{|\hat Z_s|\neq 0}{\rm d}s\Big].
\end{array}
\end{equation}
Since all of processes $\bar v, Z,\bar Z$ belong to ${\rm BMO}(\R^{1\times d})$, it is easy to verify that the process
$$
M_t:=k\int_0^t {\rm sgn}(\hat Y_s)(\bar v_s+|Z_s|+|\bar Z_s|)\frac{\hat Z_s }{|\hat Z_s|}{\bf 1}_{|\hat Z_s|\neq 0}{\rm d}B_s,\ \ t\in\T
$$
is a BMO martingale. Define
$$
\frac{{\rm d}\tilde{\mathbb P}}{{\rm d}{\mathbb P}}:=\exp\left\{M_T-\frac{1}{2}\langle M\rangle_T\right\}
$$
and
$$
\tilde B_t:=B_t-k\int_0^t {\rm sgn}(\hat Y_s)(\bar v_s+|Z_s|+|\bar Z_s|)\frac{\hat Z^\top _s }{|\hat Z_s|}{\bf 1}_{|\hat Z_s|\neq 0} {\rm d}s,\ t\in \T.\vspace{0.2cm}
$$
Then, $\tilde{\mathbb P}$ is a new probability, and $\tilde B$ is a Brownian motion with respect to $\tilde {\mathbb P}$. Then, taking the obvious conditional mathematical expectation with respect to $\tilde{\mathbb P}$ in \eqref{eq:A.8} and utilizing Gronwall's inequality yields that $|\hat Y_t|=0$ for each $t\in \T$, which is the desired conclusion.
\end{proof}

\begin{pro}\label{pro:A.2}
Assume that the generator $f$ satisfies Assumption \ref{A:A3}.
\begin{itemize}
\item [(i)] For any solution $(Y,Z)$ of BSDE \eqref{eq:A.1} such that $Y\in \s^{\infty}(\R)$, we have $Z\in {\rm BMO}(\R^{1\times d})$ and for each $0\leq t\leq r\leq T$,
\begin{equation}\label{eq:A.9}
\begin{array}{l}
\Dis\|Y\|_{\s^\infty_{[t,r]}}+\|Z\|^2_{{\rm BMO}_{[t,r]}}\leq  \Dis {4(\gamma +1) \over \gamma^2}\exp\left\{4\gamma \exp(\bar\beta T)\left(\|\eta\|_{\infty}+\|\ddot\alpha\|_{\lcal^{\infty}_{[t,T]}}
\right)\right\}
\vspace{0.2cm}\\
\Dis \hspace{1cm} \times \left\{1+\bar\beta T \exp(\bar\beta T)\left(\|\eta\|_{\infty}+\|\ddot\alpha\|_{\lcal^{\infty}_{[t,T]}}\right)
+\|\ddot\alpha\|_{\lcal^{\infty}_{[t,T]}}+\|\bar u\|^2_{{\rm BMO}_{[t,r]}}\right\}.\vspace{0.2cm}
\end{array}
\end{equation}
And, if the generator $f$ only satisfies the first inequality for the case of $y<0$ and the second inequality for the case of $y>0$ in \ref{A:A3}, the above inequality~\eqref{eq:A.9} is still true.

\item [(ii)] BSDE \eqref{eq:A.1} admits a solution $(Y,Z)$ such that $Y\in \s^{\infty}(\R)$. Moreover, $Z\in {\rm BMO}(\R^{1\times d})$.

\item [(iii)] If the generator $f$ further satisfies Assumption \ref{A:A5}, then BSDE \eqref{eq:A.1} admits a unique solution $(Y,Z)$ such that $Y\in \s^{\infty}(\R)$. Moreover, $Z\in {\rm BMO}(\R^{1\times d})$.\vspace{0.2cm}
\end{itemize}
\end{pro}

\begin{proof} (i) Let $(Y,Z)$ be any solution of BSDE \eqref{eq:A.1} such that $Y\in \s^{\infty}(\R)$. For each $0\leq t\leq r\leq T$, each stopping time $\tau\in \mathcal{T}_{[t,r]}$ and each integer $m\geq 1$, define the following stopping time
$$
\sigma_m^{\tau,r}:=r\wedge \,  \inf\left\{s\in [\tau,r]:\ \int_\tau^s |Z_s|^2{\rm d}s\geq m\right\}.
$$
Using It\^{o}-Tanaka's formula to compute $\exp(2\gamma |Y_t|)$ and utilizing the first inequality for the case of $y<0$ and the second inequality for the case of $y>0$ in Assumption \ref{A:A3}, we see that
$$
\begin{array}{l}
\Dis \exp(2\gamma |Y_\tau|)+2\gamma^2\E_\tau\left[\int_\tau^{\sigma_m^{\tau,r}} \exp(2\gamma |Y_s|) |Z_s|^2\ {\rm d}s\right]\vspace{0.2cm}\\
\ \ \leq \Dis \E_\tau\left[\exp(2\gamma |\eta|)\right]+2\gamma\E_\tau\left[\int_\tau^{\sigma_m^{\tau,r}} \exp(2\gamma |Y_s|)\left(\ddot\alpha_s+\bar\beta |Y_s| +\bar u_t |Z_s|+\frac{\gamma}{2}|Z_s|^2\right){\rm d}s\right].
\end{array}
$$
Therefore, in view of the basic inequality that $2ab\leq 2a^2+b^2/2$ for each $a,b\geq 0$,
$$
\begin{array}{lll}
\Dis \gamma^2\ \E_\tau \left[\int_\tau^{\sigma_m^{\tau,r}} |Z_s|^2{\rm d}s\right]&\leq & \Dis \exp(2\gamma \|\eta\|_\infty)+2\gamma\exp\left(2\gamma \|Y\|_{\s^\infty_{[t,r]}}\right)\left(\|\ddot\alpha\|_{\lcal^{\infty}_{[t,r]}}
+\bar\beta T\|Y\|_{\s^\infty_{[t,r]}}\right)\vspace{0.2cm}\\
&& \Dis +2\exp\left(4\gamma \|Y\|_{\s^\infty_{[t,r]}}\right)\|\bar u\|^2_{{\rm BMO}_{[t,r]}}+{\gamma^2\over 2}\E_\tau \left[\int_\tau^{\sigma_m^{\tau,r}} |Z_s|^2{\rm d}s \right].
\end{array}
$$
Sending $m\To +\infty$  and using Fatou's lemma yields that $Z\in {\rm BMO}(\R^{1\times d})$, and we have
\begin{equation}\label{eq:A.10}
\begin{array}{lll}
\Dis\|Z\|^2_{{\rm BMO}_{[t,r]}}&\leq & \Dis {2\over \gamma^2}\exp(2\gamma \|\eta\|_\infty)+{4\over \gamma}\exp\left(2\gamma \|Y\|_{\s^\infty_{[t,T]}}\right)\left(\|\ddot\alpha\|_{\lcal^{\infty}_{[t,T]}}
+\bar\beta T\|Y\|_{\s^\infty_{[t,T]}}\right)\vspace{0.2cm}\\
&& \Dis +{4\over \gamma^2}\exp\left(4\gamma \|Y\|_{\s^\infty_{[t,T]}}\right)\|\bar u\|^2_{{\rm BMO}_{[t,r]}}.
\end{array}
\end{equation}
Furthermore, using It\^{o}-Tanaka's formula we also have
\begin{equation}\label{eq:A.11}
\begin{array}{lll}
|Y_t| &\leq & \Dis |\eta|+\int_t^T \left(\ddot\alpha_s+\bar\beta |Y_s| +\bar u_t |Z_s|+\frac{\gamma}{2}|Z_s|^2\right){\rm d}s-\int_t^T {\rm sgn}(Y_s) Z_s {\rm d}B_s\vspace{0.2cm}\\
&\leq & \Dis \|\eta\|_{\infty}+\|\ddot\alpha\|_{\lcal^{\infty}_{[t,T]}}+\bar\beta\int_t^T |Y_s|{\rm d}s\vspace{0.2cm}\\
&& \Dis -\int_t^T {\rm sgn}(Y_s) Z_s\Big[{\rm d}B_s-{\rm sgn}(Y_s)\left(\bar u_s \frac{1}{|Z_s|}{\bf 1}_{|Z_s|\neq 0}+\frac{\gamma}{2}\right)Z_s^\top \  {\rm d}s\Big].
\end{array}
\end{equation}
Since both processes $\bar u$ and $Z$ belong to ${\rm BMO}(\R^{1\times d})$, we see that the process
$$
M_{\bar t}:=\int_0^{\bar t} {\rm sgn}(Y_s)\left(\bar u_s \frac{1}{|Z_s|}{\bf 1}_{|Z_s|\neq 0}+\frac{\gamma}{2}\right)Z_s {\rm d}B_s,\quad  \bar t\in\T
$$
is a BMO martingale. Define
$$
\frac{{\rm d}\tilde{\mathbb P}}{{\rm d}{\mathbb P}}:=\exp\left\{M_T-{1\over 2}\langle M\rangle_T\right\}
$$
and
$$
\tilde B_{\bar t}:=B_{\bar t}-\int_0^{\bar t} {\rm sgn}(Y_s)\left(\bar u_s \frac{1}{|Z_s|}{\bf 1}_{|Z_s|\neq 0}+\frac{\gamma}{2}\right)Z_s^\top \ {\rm d}s,\quad \bar t\in \T.\vspace{0.2cm}
$$
Then, $\tilde{\mathbb P}$ is a new probability, and $\tilde B$ is a Brownian motion with respect to $\tilde {\mathbb P}$. Then, taking the obvious conditional mathematical expectation with respect to $\tilde{\mathbb P}$ in \eqref{eq:A.11} and using Gronwall's inequality yields that
\begin{equation}\label{eq:A.12}
\|Y\|_{\s^\infty_{[t,T]}}\leq \exp(\bar\beta T)\left(\|\eta\|_{\infty}+\|\ddot\alpha\|_{\lcal^{\infty}_{[t,T]}}\right).
\end{equation}
Then, the desired inequality~\eqref{eq:A.9} immediately follows from \eqref{eq:A.10} and \eqref{eq:A.12}. \vspace{0.2cm}

(ii) It is easy to check that for each pair of integers $m,l\geq 1$, the following function
$$
\begin{array}{l}
f^{m,l}(\omega,t,y,z):= \Dis \inf\left\{f^+(\omega,t,\bar y,\bar z)+m|y-\bar y|+m|z-\bar z|:\ (\bar y,\bar z)\in \R\times\R^{1\times d}\right\} \vspace{0.2cm}\\
\Dis \hspace{2.8cm} -\inf\left\{f^-(\omega,t,\bar y,\bar z)+l|y-\bar y|+l|z-\bar z|:\ (\bar y,\bar z)\in \R\times\R^{1\times d}\right\},\vspace{0.2cm}\\
\Dis \hspace{4cm} (\omega,t,y,z)\in \Omega\times\T\times \R\times\R^{1\times d}
\end{array}
$$
is well defined and an $(\F_t)$-progressively measurable process for each $(y,z)$. In view of Assumption \ref{A:A3}, we see that $f^{m,l}$ is uniformly Lipschitz continuous in the state variables $(y,z)$ and also satisfies Assumption \ref{A:A3} with the same parameters for $m,l\geq 1$, and that for each $l\geq 1$, the sequence $\{f^{m,l}\}_{m=1}^{\infty}$ converges increasingly and uniformly on compact sets to a function $f^{\infty,l}$ as $m\To +\infty$, and $\{f^{\infty,l}\}_{l=1}^{\infty}$ converges decreasingly and  uniformly on compact sets to the generator $f$ as $l\To +\infty$. Then, $\as$, for each $(y,z)\in \R\times\R^{1\times d}$ and $m,l\geq 1$, we have
\begin{equation}\label{eq:A.13}
|f^{m,l}(\omega,t,y,z)|\leq \ddot\alpha_t(\omega)+{1\over 2}\bar u^2_t(\omega)+\varphi(|y|)+\bar\beta |y| +{\gamma+1\over 2}|z|^2,
\end{equation}
and then
\begin{equation}\label{eq:A.14}
|f^{m,l}(\omega,t,0,0)|\leq \ddot\alpha_t(\omega)+{1\over 2}\bar u^2_t(\omega)+\varphi(0),
\end{equation}
which means that
$$
\E\left[\left(\int_0^T |f^{m,l}(s,0,0)|{\rm d}s\right)^2\right]<+\infty.\vspace{0.1cm}
$$
Consequently, by the classical results we know that for each $m,l\geq 1$, the following BSDE\vspace{0.1cm}
$$
Y_t^{m,l}=\eta+\int_t^T  f^{m,l}(s,Y_s^{m,l},Z_s^{m,l}){\rm d}s-\int_t^T Z_s^{m,l} {\rm d}B_s, \ \ t\in\T\vspace{0.1cm}
$$
admits a unique solution $(Y^{m,l},Z^{m,l})\in \s^2(\R)\times \hcal^2(\R^{1\times d})$, $Y^{m,l}$ converges increasing pointwisely to a process $Y^{\infty,l}$ as $m\To +\infty$, and $Y^{\infty,l}$ converges decreasing pointwisely to a process $Y$ as $l\To +\infty$. Moreover, by the classical a priori estimate on the $L^2$ solution (see Proposition 3.2 in \citet{Bri03}) we know that there exists a uniform constant $c_0>0$ such that for each $m,l\geq 1$ and $t\in \T$,
$$
\E_t\left[\sup_{s\in [t,T]}|Y^{m,l}_s|^2\right]\leq c_0 \exp\left(2(m\vee l)T+2(m\vee l)^2 T\right)\E_t\left[|\eta|^2+\left(\int_t^T |f^{m,l}(s,0,0)|{\rm d}s\right)^2\right],
$$
which together with \eqref{eq:A.14}, the energy inequality for BMO martingales (see for example Section 2.1 in \citet{Kazamaki1994book}) and the facts that $\eta\in L^{\infty}(\R)$, $\ddot\alpha\in \lcal^\infty$ and $\bar u\in {\rm BMO}(\R)$ yields that $Y^{m,l}\in \s^{\infty}(\R)$ for each $m,l\geq 1$, and then $Z^{m,l}\in {\rm BMO}(\R^{1\times d})$ by (i).\vspace{0.2cm}

Finally, it follows from (i) that there exists a uniform constant $A>0$ which is independent of $m$ and $l$ such that $\as$, we have $\sup_{m,l\geq 1} |Y^{m,l}_t(\omega)|\leq A$ and, in view of \eqref{eq:A.13},
$$
\RE\ (y,z)\in [-A,A]\times \R^{1\times d},\ \ |f^{m,l}(\omega,t,y,z)|\leq \ddot\alpha_t(\omega)+{1\over 2}\bar u^2_t(\omega)+\varphi(A)+\bar\beta A +{\gamma+1\over 2}|z|^2.
$$
Thus, we can apply twice the monotonic stability result Proposition 3.1 in \cite{LuoFan2018SD} to obtain the existence of a process $Z\in \hcal^2(\R^{1\times d})$ such that $Y\in \s^\infty(\R)$ and $(Y,Z)$ is a desired solution of BSDE \eqref{eq:A.1}. And, it follows from (i) that $Z\in {\rm BMO}(\R^{1\times d})$.\vspace{0.2cm}

(iii) In view of  (i),  (iii) is proved in the same way as (iii) of \cref{pro:A.1}.\vspace{0.1cm}
\end{proof}

In the following proposition, the generator $g$ has a linear growth in the state variable $z$.

\begin{pro}\label{pro:A.3}
Assume that the generator $f$ satisfies Assumption \ref{A:A4}.
\begin{itemize}
\item [(i)] For any solution $(Y,Z)$ of BSDE \eqref{eq:A.1} such that $Y\in \s^2(\R)$, we have $Y\in \s^{\infty}(\R)$ and $Z\in {\rm BMO}(\R^{1\times d})$ and for each $t\in \T$, there exists a uniform constant $c_0>1$ such that
\begin{equation}\label{eq:A.15}
\|Y\|_{\s^\infty_{[t,T]}}+\left\|Z\right\|_{{\rm BMO}_{[t,T]}}^2
\leq 1+c_0 \exp\left(2\bar\beta T+2\bar\lambda^2 T\right)\left(\|\eta\|_\infty^2
+2\|\dot\alpha\|_{\mcal^\infty_{[t,T]}}^2\right).\vspace{0.1cm}
\end{equation}
And,  if the generator $f$ only satisfies the first inequality for the case of $y<0$ and the second inequality for the case of $y>0$ in \ref{A:A4}, the above conclusion \eqref{eq:A.15} still holds.

\item [(ii)] BSDE \eqref{eq:A.1} admits a minimal (resp. maximal) solution $(Y,Z)$ such that $Y\in \s^2(\R)$ in the sense that for any solution $(\bar Y,\bar Z)$ of BSDE \eqref{eq:A.1} such that $\bar Y\in \s^2(\R)$, we have for each $t\in \T$, $\ps$, $Y_t\leq \bar Y_t$ (resp. $Y_t\geq \bar Y_t$). Moreover, $(Y,Z)\in \s^\infty(\R)\times{\rm BMO}(\R^{1\times d})$.
\item [(iii)] If the generator $f$ further satisfies Assumption \ref{A:A5}, then BSDE \eqref{eq:A.1} admits a unique solution $(Y,Z)$ such that $Y\in \s^2(\R)$. Moreover, $(Y,Z)\in \s^\infty(\R)\times{\rm BMO}(\R^{1\times d})$.\vspace{0.1cm}
\end{itemize}
\end{pro}

\begin{proof} (i) Let $(Y,Z)$ be any solution of BSDE \eqref{eq:A.1} such that $Y\in \s^2(\R)$. In view of the first inequality for the case of $y<0$ and the second inequality for the case of $y>0$ in Assumption \ref{A:A4} and by virtue of the classical a priori estimate on the $L^2$ solution (see Proposition 3.2 in \citet{Bri03}) we can deduce the existence of a uniform constant $c_0>1$ such that for each $t\in \T$ and each $\tau\in \mathcal{T}_{[t,T]}$,
$$
\E_\tau\left[\sup_{s\in [\tau,T]}|Y_s|^2+\int_\tau^T |Z_s|^2{\rm d}s\right]\leq c_0 \exp\left(2\bar\beta T+2\bar\lambda^2 T\right)\E_\tau\left[|\eta|^2+\left(\int_\tau^T \dot\alpha_s {\rm d}s\right)^2\right],
$$
and then, by the energy inequality for BMO martingales and (i) of \cref{rmk:2.1},
$$
\|Y\|_{\s^\infty_{[t,T]}}^2+\left\|Z\right\|_{{\rm BMO}_{[t,T]}}^2
\leq c_0 \exp\left(2\bar\beta T+2\bar\lambda^2 T\right)\left(\|\eta\|_\infty^2+2\|\dot\alpha\|_{\mcal^\infty_{[t,T]}}^2\right).
$$
Then, the desired inequality \eqref{eq:A.15} follows from the previous inequality immediately, and then $(Y,Z)\in \s^{\infty}(\R)\times {\rm BMO}(\R^{1\times d})$.\vspace{0.2cm}

(ii) Define $M:=1+c_0 \exp\left(2\bar\beta T+2\bar\lambda^2 T\right)\left(\|\eta\|_\infty^2
+2\|\dot\alpha\|_{\mcal^\infty}^2\right)$, $\rho^M(x):={Mx\over M\vee |x|},\ x\in \R$ and
$$
f^M(\omega,t,y,z):=f(\omega,t,\rho^M(y),z), \ \ (\omega,t,y,z)\in \Omega\times\T\times\R\times \R^{1\times d}.\vspace{-0.1cm}
$$
It is easy to verify that the generator $f^M$ also satisfies Assumption \ref{A:A4}, and that $\as$, for each $(y,z)\in \R\times \R^{1\times d}$,
$
|f^M(\omega,t,y,z)|\leq \dot\alpha_t(\omega)+\varphi(M)+\bar\beta M+\bar\lambda |z|
$.
Then, by \citet{Lep97} we know that the following BSDE
\begin{equation}\label{eq:A.16}
Y_t=\eta+\int_t^T  f^M(s,Y_s,Z_s){\rm d}s-\int_t^T Z_s {\rm d}B_s, \ \ t\in\T
\end{equation}
admits a maximal solution $(\bar Y,\bar Z)$ and a minimal solution $(\underline {Y},\underline {Z})$ \vspace{0.2cm}in the space $\s^2(\R)\times \hcal^2(\R^{1\times d})$.

We now show that $(\bar Y,\bar Z)$ and $(\underline {Y},\underline {Z})$ are also the desired maximal and minimal solution of BSDE \eqref{eq:A.1}. Indeed, since $f^M$ satisfies Assumption \ref{A:A4}, it follows from (i) and the definition of $f^M$ that both of them belong to the space $\s^{\infty}(\R)\times {\rm BMO}(\R^{1\times d})$ and are also solutions of BSDE \eqref{eq:A.1}. Furthermore, let $(Y, Z)$ be any solution of BSDE \eqref{eq:A.1} such that $Y\in \s^2(\R)$. Then, it follows from (i) and the definition of $f^M$ again that $(Y,Z)\in \s^{\infty}(\R)\times {\rm BMO}(\R^{1\times d})$ and it is also a solution of BSDE \eqref{eq:A.16} in the space $\s^2(\R)\times \hcal^2(\R^{1\times d})$. Consequently, for each $t\in\T$, we have $\underline{Y}_t\leq Y_t\leq \bar Y_t,\ \ps$, which is the desired conclusion.\vspace{0.2cm}

(iii) In view of (i), (iii) can be proved in the same way as (iii) of \cref{pro:A.1}.
\end{proof}



\setlength{\bibsep}{2pt}

\def\cprime{$'$}

\end{document}